\documentclass{amsart}

\usepackage{amsthm}
\usepackage{amscd} 
\usepackage{amssymb}  
\usepackage{tikz}
\usepackage{lipsum}

\newcommand{\bcc}{{\mathbb C}}

\newcommand{\N}{{\mathbb N}}

\newcommand{\R}{{\mathbb R}}
\newcommand{\bbz}{{\mathbb Z}}
\newcommand{\C}{{\mathcal C}}
\newcommand{\Tau}{{\mathcal T}}
\newcommand{\bit}{\begin{itemize}}
\newcommand{\eit}{\end{itemize}}
\newcommand{\ben}{\begin{enumerate}}
\newcommand{\een}{\end{enumerate}}
\newcommand{\ep}{\varepsilon}
\newcommand{\CP}{\mathbb{CP}}
\newcommand{\Aff}{{\rm Aff}}
\newcommand{\ms}{\mathcal{S}}
\newcommand{\mt}{\mathcal{T}}
\newcommand{\mT}{\mathcal{T}}

\allowdisplaybreaks

\numberwithin{equation}{section} 

\newtheorem{Thm}{Theorem}[section] 
\newtheorem{Prop}[Thm]{Proposition} 
\newtheorem{Lem}[Thm]{Lemma} 
\newtheorem{Cor}[Thm]{Corollary} 
 
\theoremstyle{remark} 
\newtheorem{Rem}[Thm]{Remark}

\theoremstyle{definition} 
\newtheorem{dif}[Thm]{Definition} 
\newtheorem{exi}[Thm]{Example}

\newenvironment{Def}
  { \pushQED{\qed}   \dif}
  {\popQED\enddif}

\newenvironment{Exa}
  { \pushQED{\qed}   \exi}
  {\popQED\endexi}

\newcommand{\de}{{\partial}}

\begin{document}
\author{Paolo Salvatore}
\address{Dipartimento di Matematica,  
Universit\`{a} di Roma Tor Vergata, 
Via della Ricerca Scientifica, 
00133 Roma, Italy}
\email{salvator@mat.uniroma2.it}
\thanks{ The author would like to thank the Isaac Newton Institute for Mathematical Sciences for support and hospitality during the programme Homotopy Harnessing Higher Structures. This work was partially supported by the EPSRC grant number EP/R014604/1 and by the Excellence Department Project MATH@TOV number E83C18000100006}

\title{A cell decomposition of the Fulton MacPherson operad}
\begin{abstract}
We construct a small regular cellular decomposition of the Fulton MacPherson operad $FM_2$ 
that is compatible with the operad composition.  The cells are indexed by trees with edges of two colors and vertices labelled by cells of the cacti operad. 
We compute the generating functions counting the cells, that are algebraic. 
\end{abstract}
\maketitle

\section{Introduction} 
The combinatorial approach to the topology of moduli spaces of Riemann surfaces has led to great achievements 
in the past decades, yielding in particular explicit cellular decompositions of the moduli spaces.  The most known 
approach for marked surfaces, due to Harer, Penner and Thurston, makes use of quadratic differentials, or equivalently of hyperbolic metrics.
We refer the reader to the detailed survey by Mondello \cite{Mondello}. There is another less known approach by Giddings and Wolpert
\cite{GW}, that makes use of meromorphic differentials, and leads to a different cellular structure on the moduli space. 
This construction has some analogies with the approach for the case of surfaces with boundary by B\"odigheimer \cite{Boe}. 
The combinatorics of the Giddings-Wolpert cells has been studied by Nakamura \cite{Nakamura} and Garner-Ramgoolam  \cite{Garner-Ramgoolam}.
We concentrate on the genus zero case in this paper. The moduli space of $k+1$ distinct marked points on the Riemann sphere $\mathcal{M}_{0,k+1}$ coincides with the quotient of the configuration space $F_k(\bcc)$ of $k$ distinct points in 
$\bcc$ 
by the affine group $\bcc \rtimes \bcc^*$ generated by translations, dilations, and rotations. In genus zero the Giddings-Wolpert cells can be described explicitly in terms of a flow associated to the multivalued function $$h(z)=\prod_{i=1}^k (z-z_i)^{a_i}$$  as we will explain in section \ref{sec:conf}. 
From a physical point of view, the flow is the integral flow of the vector field that is the sum of $k$ radial fields generated by charges $a_i$ at the respective configuration points $z_i$,  of strength inversely proportional to the distance, and proportional to the charges.
The flow lines in $\bcc$ along which 
the argument $$\arg(h(z))=\sum_{i=1}^k a_i \arg(z-z_i)$$ is constant connect the critical points of $h$ and the zeros of $h$.  
The resulting graph is a tree, and its shape corresponds to a cell of the moduli space. The parameters of the cell are obtained by evaluating $|h(z)|$ on the critical points of $h$, and measuring the angles between flow lines hitting a zero of $h$.  
If we do not mod out by rotations, i.e. consider the moduli space $$FM'_k:= F_k(\bcc)/(\bcc \rtimes \R_+) \simeq F_k(\bcc),$$  then
we obtain a similar cell  decomposition by adding to the previous data the choice of a distinguished flow line of $h$ 
going out of $\infty$. The symmetric group $\Sigma_k$ acts freely on the set of cells.
This description appeared first in the thesis of my master student Egger in 2009 \cite{Egger}. He proved that 
the union of cells of $FM'_k$ where the value $|h(z)|$ agrees on all critical points of $h$
forms a deformation retract $\C_k \subset FM'_k$ that is well known to topologists, the {\em cacti} complex.  
This complex was introduced by McClure-Smith \cite{MS} in a combinatorial form, as it was needed 
for their proof of the Deligne conjecture, and independently by Voronov \cite{Voronov} and Kaufmann \cite{Kaufmann}
in a more geometric form. An element of the cacti complex is a union of circles, called lobes, intersecting in a finite number of points, along a planar tree, together with a (global) base point. 
There is also a local base point of each lobe of a cactus that is the first point of the lobe met moving from the (global) base point of the whole cactus in anti-clockwise direction. 
Egger shows that the lobes of a cactus correspond to the sinks of the flow, i.e. the configuration points, and that some lobes intersect exactly when there is a critical point from which originate flow lines hitting the corresponding points. The base point is the flow line with positive horizontal tangent at infinity. 
The cactus associated to a configuration with equal critical value
$M=|h(\zeta)|$ for any 
critical point $\zeta$ of $h$  can be realized explicitly inside $\mathbb{C}$ as $\{z |\; |h(z)|=M \},$  

\begin{center}
\begin{tikzpicture} [scale=.8] 
\node (4) at (6.5,.8) [circle,draw,inner sep=2pt,label=left:$z_1$] {};
\node (5) at (9.5,1.2) [circle,draw,inner sep=2pt,label=above:$z_2$] {};
\node (w) at (8,1) [circle,fill,inner sep=1pt,label=above left:$\zeta$] {};

\path [->] (w) edge (5) (w) edge (4) ; 
\draw (4) [dashed]  circle (1.5cm);
\draw (5) [dashed] circle (1.5cm);
\node (infi) at (12,.8) {$\infty$};
\draw (5)  .. controls (10,.8) and (11,.8) ..(infi); 

\end{tikzpicture}
\end{center}

The complexes $\C_k$ are regular CW-complexes equipped with the action of the symmetric 
group permuting freely the cells. In fact they are realization of semi-simplicial complexes, as we explain in section
\ref{sec:cacti}. The collection of the cellular chain complexes $C_*(\C_k)$ forms an $E_2$-operad that was used by several authors 
in the proof of the Deligne conjecture, notably McClure-Smith \cite{MS} and Kontsevich-Soibelman \cite{KS}. In a sense
$C_*(\C_k)$ is the smallest algebraic $E_2$-operad.
However the spaces $\C_k$ themselves do not form a topological operad. They only 
form an operad up to higher homotopies \cite{Bas}.
 The composition of cacti 
$c \circ_i d$, roughly speaking, glues  $d$ into the $i$-th lobe of $c$, by identifying the global base point of $d$ and the local base point of the $i$-th lobe of $c$. 

The associativity issue can be fixed at the price of adding 
extra parameters as we explain in our paper \cite{Deligne}. 

\medskip

Let us go back to the full moduli space $FM'_k=F_k(\bcc )/(\bcc \rtimes \R^+)$. 
This admits a compactification $FM(k)$, the Fulton MacPherson space,
defined essentially by Axelrod and Singer \cite{AxSing}, that is a real manifold with corners. For details we refer to section \ref{sec:fulton}.  Roughly speaking the boundary of $FM$ consists of configurations where points are allowed 
to come together recursively.
The collection of the spaces $FM(k)$ forms a topological $E_2$ operad, the Fulton MacPherson operad, that was first 
defined by Getzler-Jones in \cite{GJ}. In a sense this is the smallest topological $E_2$ operad. 
The operad composition $x \circ_i y$ in $FM$ replaces the $i$-th point of $x$ by an infinitesimal rescaled version of $y$.  
Together with its 
higher $E_n$ analogues, $FM$ it is a fundamental tool in the embedding calculus, see for example \cite{Turchin}. 
What is the connection between the cacti complexes $\C_k$ and the Fulton Mac Pherson operad $FM$ ?
Our main result is the following theorem:
\begin{Thm} \label{princ} 

\ben
\item The space $FM(k)$ admits a regular CW decomposition into cells indexed by trees with $k$ leaves, edges of two colours,  and vertices labelled by cells of $\C$, equivariant with respect to the action of the symmetric group $\Sigma_k$.
\item The operad composition along $$FM(k) \times FM(n_1) \times \dots \times FM(n_k) \to FM(n_1+\dots+n_k)$$
of a product of cells is a cell. 
\item The $E_2$ operad of cellular chains $C_*(FM)$ is isomorphic to the standard operad resolution $\Omega B(C_*(\C))$.
\een
\end{Thm}
This result is partially motivated by a conjecture of Kontsevich-Soibelman  \cite{KS} stating that a similar cellular decomposition exists, if the operad $C_*(\C)$ is replaced by the so called minimal operad $\mathcal{M}$.  
Actually $\C_*(C)$ has much less generators than $\mathcal{M}$, and so the number of our cells is much smaller than that conjectured by Kontsevich-Soibelman. We expect our decomposition to be the smallest 
combinatorial  decomposition of the Fulton MacPherson operad.   
We call trees in Theorem \ref{princ} {\em metatrees}, following the terminology  by Kontsevich-Soibelman, because they correspond to trees
with vertices labelled by trees, that in turn have vertices labelled by cacti cells. The functor $B$ is the algebraic bar construction of an
operad, that is a cooperad, and $\Omega$ is the dual cobar construction of a cooperad, that is an operad \cite{Benoit}. Both construction can be described by labelled trees, and so the standard resolution 
$\Omega B O$ of an algebraic operad $O$ can be described in terms of metatrees with labels in $O$.
Therefore condition 3) of the Theorem is about differentials, see section \ref{sec:fulton} for details.  
In particular we are able to list the cells contained in the boundary of a given cell.

\medskip

Recall that in the category of topological operads 
a standard resolution is provided by the Boardman-Vogt construction $W$ \cite{BV}. This has been generalized to other categories by 
Berger and Moerdijk \cite{BM}, so that it coincides with $\Omega B$ for algebraic operads. The resolutions are compatible with 
the chain functor, in the sense that if  $P$ is a cellular topological operad (i.e. the structure maps are cellular), then 
$WP$ is also a cellular topological operad, and $C_*(WP) \cong \Omega B(C_*(P))$. 

Therefore in a sense the Fulton MacPherson operad $FM$ is trying to be the $W$ construction of $\C$, except that the latter is not  really a strict operad.

\medskip

In the simpler $E_1$ scenario instead the $W$ construction does appear. Let $Ass$ be the topological 
associative operad, and let $K$ be the Stasheff operad.  Then $W(Ass) \cong K$. 
This implies that the algebraic $A_\infty$ operad $C_*(K)$ is isomorphic to the resolution $\Omega B(C_*(Ass))$
of the algebraic associative operad $C_*(Ass)$.
Notice that $K \cong FM_1$, where $FM_1$ denotes  
the Fulton MacPherson $E_1$ operad. For a detailed proof of this fact see \cite{Barber}. 

\medskip

We remark that there are higher $E_n$-operad  analogues of the cacti complexes for any $n>2$, related to the surjection operad 
by McClure and Smith \cite{MS-sur}, but there is no hope of extending Theorem \ref{princ} to $n>2$  because these complexes have a larger dimension 
than the corresponding spaces of the $E_n$-operad $FM_n$, and so they cannot generate them in any reasonable sense. 
The complex structure plays a key role in the case $n=2$.

\medskip

Let us go back again to the cell decomposition of the moduli space $FM'_k$. In section \ref{sec:cellular} we construct explicitly the family 
of cell decompositions depending on positive weights $a_1,\dots,a_k$ attached to the configuration points
$z_1,\dots,z_k$, by using the flow of the multivalued function $h(z)=\prod_{i=1}^k(z-z_i)^{a_i}$. 
We stress that our cells are distinct from the Giddings-Wolpert cells and especially suitable from the point of view of the operadic structure. Roughly speaking, the cells correspond to families of nested cacti cells, where each lobe has a child cactus associated to it. This is proved in
Theorem \ref{pocofa}.
Each cactus of the family can be visualized in $\mathbb{C}$ as the connected component of 
$\{z \in \mathbb{C} \; | \; |h(z)|=|h(\zeta)| \}$ containing a critical point $\zeta$.  
\begin{center}
\begin{tikzpicture}
\node (3) at (6.5,0) [circle,draw,inner sep=2pt] {};
\node at (5.85,0) {$z_1$};
\node at (6,.4) {$\zeta$};
\node at (5.85,.8) {$z_2$};
\node at ( 7.7,.5   ) {$\bar{\zeta}$};

\node (4) at (6.5,.8) [circle,draw,inner sep=2pt] {};
\node (34) at (6.5,.4) [circle,fill,inner sep=1pt]{};
\path [->] (34) edge (3)  (34) edge (4);
\node (5) at (9.5,1.2) [circle,draw,inner sep=2pt,label=above:$z_3$] {};
\node (w) at (8,.85) [circle,fill,inner sep=1pt] {};
\draw (5) .. controls (8,.8) .. (4);
\draw (34) [dashed]  ellipse (1.6cm and 1.4cm);
\draw (5) [dashed] ellipse (1.5cm and 1.3cm);
\node (infi) at (12,.8) {$\infty$};
\draw (5)  .. controls (10,.8) and (11,.8) ..(infi); 
\draw (3) [dashed] circle (.4cm);
\draw (4) [dashed] circle (.4cm);

\end{tikzpicture}
\end{center}




Corollary 
\ref{bc} shows that the cellular complex $C_*(FM')$ is isomorphic to the operadic cobar construction $BC_*(\C)$ up to a shift.

\medskip

In section \ref{sec:fulton} we consider the extension of the cell structure to the boundary of $FM(k)$, 
that yields a regular CW-structure on $FM(k)$, also depending on the positive weights $a=(a_1,\dots,a_k) \in \stackrel{\circ}{\Delta}_k$.
This is not trivial, since for example the Fox-Neuwirth decomposition of configuration spaces does not extend to the boundary as
observed by Tamarkin \cite{Voronov2}.

In Theorem \ref{condo} we prove that the cells of $FM(k)$ correspond to metatrees labelled by cells of $\C$, and in Corollary 
\ref{omb} we derive that the chain complex  $C_*(FM)$ is isomorphic to $\Omega B C_*(\C)$.  This shows conditions 1) and 3) of 
Theorem \ref{princ}.

\medskip

If we wonder about the compatibility of the operad composition in $FM$ with the cell structure, there are two issues that need to be addressed. First of all, the cell decompositions of $FM(k)$ depend on weights $a=(a_1,\dots,a_k)$, that live in the open $k$-simplexes. Recall that physically this corresponds to having charges $a_i$ at the configuration points.
If we compose two configurations $x, y \in FM$ with respective weights $a=(a_1,\dots,a_k)$ and $b=(b_1,\dots,b_l)$,
then the natural weight associated to the composition $x \circ_i y$  is

$$a \circ_i b= (a_1,\dots,a_{i-1},a_i b_1,\dots, a_i b_l,a_{i+1},\dots, a_k)$$ 
using the operad structure on the collection of open simplexes given by multiplication of weights \cite{AK}.

Namely, thinking about the physical interpretation of the flow, for the configuration $x \circ_i y$ the charges in the infinitesimal version $y$ look like a single point $x_i$ with charge $\sum_{j=1}^l a_i b_j = a_i$, whereas  
near the cluster, the contributions to the field coming from the points $x_j, j \neq i$ are constant and do not influence the flow. Of course this discussion makes sense only asymptotically, because the flow is not defined in the limit. 
However the cells can be defined and make sense in the limit.

This discussion suggests that the cells should form a fiberwise operad over the operad of open simplexes, rather 
than an ordinary operad.

\medskip

However there is a second issue about base points of cacti and rotations:
if we consider the recursive construction of the cells,  it turns out that that in a composition $x \circ_i y$ the base point of the cactus for the cluster coming from $y$ has a base point determined by the first flow line hitting $x_i$, and so the base point is rotated by 
the normalized tangent vector $\theta^i_a(x) \in S^1$ of the flow line at $x_i$, where $a$ is the weight of $x$. 

The following picture shows the vectors $\theta^i_a(z_1,z_2,z_3)$.

\begin{center}
\begin{tikzpicture}
\node (3) at (6.5,0) [circle,draw,inner sep=1pt] {$z_1$};
\node (4) at (6.5,.8) [circle,draw,inner sep=1pt] {$z_2$};
\node (34) at (6.5,.4) [circle,fill,inner sep=1pt]{};
\path (34) edge (3)  (34) edge (4);

\node (5) at (9.5,1.2) [circle,draw,inner sep=1pt] {$z_3$};
\node (w) at   (7.8,.6)   [circle,fill,inner sep=1pt] {};
\node (infi) at (12,.5) {$+\infty$};
\draw (w) .. controls (8.5,.6) .. (5) ;
\draw (w) .. controls (7.3,.6) .. (4) ;
\draw (infi) .. controls (10.5,.5) .. (5) ;

\draw [<-,thick] (3) -- (6.5,-.5);
\draw [<-,thick] (4) -- (6,1);
\draw [<-,thick] (5) -- (9.2,1.6);

\node at (6.2,-.5) {$\theta^1$};
\node at (6.2,1.3) {$\theta^2$};
\node at (9.6,1.8) {$\theta^3$};

\end{tikzpicture}
\end{center}

So the operadic composition in $FM$ of cells is not a cell, 
  because the distinguished flow lines (base points) are rotated along
the composition.
Suppose that $\sigma$ and $\tau$ are metatrees indexing cells $\sigma_a \subset FM(k)$ and 
$\tau_b \subset FM(l)$, for respective weights $a=(a_1,\dots,a_k)$ and $b=(b_1,\dots,b_l)$.  Consider the metatree $\sigma \circ_i 
\tau$ obtained by grafting $\tau$ to the $i$-th leaf of $\sigma$.
Then the cell $(\sigma \circ_i \tau)_{a \circ_i b}$ is not
 the operadic composition  $\sigma_a \circ_i \tau_b \subset FM(k+l-1)$, but it is equal to
$$\{(x \circ_i \theta^i_a(x)y ) \,  | \,  x \in \sigma_a, y \in \sigma_b\}$$

The map $x \mapsto \theta^i_a(x)$ is null-homotopic, since $\theta^i_a(x)$ tends to $1 \in S^1$ as $a_i$ tends to $1$ and $a_j$ to $0$ for $j \neq i$.
 We fix the problem of the rotations by constructing in Lemma \ref{remrot} a family of self-homeomorphisms 
 $$\Lambda_a:FM(k) \to FM(k)$$
 of $FM$ satisfying 
$$\Lambda_{a \circ_i b}(x  \circ_i y) =  \Lambda_a(x) \circ_i  \Lambda_b( \theta^i_a  (x)^{-1} y )$$ 
We define a new cellular structure, called $a$-structure, on $FM(k)$, taking the image of the former cells via $\Lambda_a$.
Theorem \ref{forst} states that with respect to this structure the $\circ_i$ -composition of an $a$-cell and a $b$-cell
is a $(a \circ_i b)$-cell .

\medskip

The fundamental tool used to prove Lemma \ref{remrot} is the existence of a $SO(2)$-equivariant isomorphism of topological operads.
\begin{equation}
\beta: FM \cong W(FM) \label{fund}
\end{equation}
We announced this result in the non equivariant setting in \cite{Barcelona}, and added an equivariant detailed
proof in \cite{new}, including the higher $E_n$ case. In the case $n=2$ we are interested in, $\beta$ can be constructed explicitly as
a piecewise algebraic map (see Remark 9 in \cite{new}).

In subsection \ref{sub:rot} we construct an explicit family of self-homeomorphisms $g_a: WFM(k) \cong WFM(k)$
and then set $\Lambda_a = \beta^{-1}g_a \beta$. The maps $g_a$ are constructed piecewise, using the (higher) null-homotopies of the $\theta^i$.

\medskip

We still have dependence on the weights along the operadic composition. 
We use once again  the isomorphism (\ref{fund}) in subsection \ref{sub:weights}
to get rid of the weights. 
Namely we build a cell structure on $WFM$ compatible with the operad composition, with cells in bijection with those of $FM \subset WFM$. 
We use the fact that each space $WFM(n)$ is obtained by adding a collar to 
the manifold with corners $FM(n)$, i.e. 
$$WFM(n)= \cup_S ( S \times [0,1]^{codim(S)}),$$ where 
$S$ takes values in the strata of $FM(n)$. 
 This property is at the heart of the proof of the isomorphism (\ref{fund}).
Each metatree $\sigma$ of $FM(n)$ defines a cell of $WFM(n)$ of the form
$$\sigma_{WFM} = \cup_S (\sigma \cap S) \times [0,1]^{codim(S)}$$ In this formula  
the weights of the cells of $FM(n)$ are not indicated and have to be chosen appropriately: for the full stratum $S=FM(n)$ we choose for $\sigma$ the weight $b_n=(1/n,\dots,1/n)$  ($b_n$ is the barycenter of the $(n-1)$-simplex).  More generally
$(\sigma \cap S) \times [0,1]^{codim(S)}$ is a family of cells parametrised  by $t \in [0,1]^{codim(S)}$ where the fiber
$(\sigma \cap S)$ over $t$ has a weight depending on $t$ and $S$.
For example, recall that $\de FM(3)= \coprod_{i=1}^3  \de_i FM(3)$ where $\de_i FM(3)  \cong FM(2) \times FM(2)$
is the component where the two points not indexed by $i$ come together,
and that $$WFM(3) = FM(3) \cup_{\de FM(3)} (\de FM(3) \times [0,1]) .$$
If $\sigma$ is a (family of) cell(s) of $FM(3)$ intersecting the component $\de_i FM(3)$ in 
$\de_i \sigma$, then the corresponding cell of $WFM(3)$ is the union
$$(\sigma_{b_3}) \cup \bigcup_{i=1}^3 ( (\de_i \sigma)_{(1-t)b_3+tP_i}  \times \{t\})$$
 where  $P_i=(a_1,a_2,a_3)$  has coordinates $a_i=1/2$ and
$a_j=1/4$ for $j \neq i$. This cell decomposition of  $WFM$ is operadic, as we prove in Theorem \ref{ultimo}.
 
The operadic cell structure on $W(FM)$ is transferred to an operadic cell structure on $FM$ via the isomorphism (\ref{fund}),
and so condition 2) of Theorem \ref{princ} holds.   This completes the proof of  Theorem \ref{princ}.

\medskip

Along the way we obtain new results about the generating functions counting the number of cells of our spaces.
The generating function counting the cells of the cacti complexes $\C_k$ are known,  as they count the so called Davenport-Schinzel sequences \cite{Gardy}, see Theorem \ref{count-ds}. 
We find an explicit algebraic expression in Theorem \ref{count-open} 
for the generating function counting the open cells of the open moduli spaces $FM'_k$, and we
indicate in Theorem  \ref{count-closed} an algebraic expression for the generating function counting the cells of 
the compactified moduli spaces $FM(k)$. 

\medskip

We comment on the regularity of the cells: they are piecewise algebraic in an infinitesimal sense,
since they are obtained via integral curves of meromorphic differentials with simple poles.
If the weights are rational numbers, then the 
open cells of $FM'_k$  are piecewise algebraic in the strict sense 
and so are the associated cells of the compactification $FM(k)$ in section \ref{sec:fulton}.
The homotopies used to eliminate rotations in subsection \ref{sub:rot}, and those used to eliminate the weights
in subsection \ref{sub:weights} are infinitesimally piecewise algebraic. 
The situation reminds of the regularity of Kontsevich graph PA-forms \cite{Kont} \cite{LV}, that are obtained by integrating piecewise algebraic 
 forms.
 We expect the integral of the Kontsevich forms on our cells to be related to multizeta values, in view of the recent 
 progress on the subject by Banks-Panzer-Pym \cite{BPP}.

\medskip

Which applications and developments can one hope for? 
As we said the Fulton-MacPherson operad is crucial in the embedding calculus. 
Our explicit operadic cellular decomposition could be useful in the study of spaces of embeddings 
from surfaces, that can be described in embedding calculus by configuration spaces that are right modules over $FM$. 
Nate Bottman suggests that our decomposition could be useful in symplectic topology, where a chain level action
of the framed little discs has not been constructed yet (compare conjecture 2.6.1 in \cite{Abou}). 
\medskip

It is natural to ask for a generalization of Theorem \ref{princ} to
the higher genus version of $FM(k)$, i.e. the moduli space of genus $g$ Riemann surfaces
with $k$ marked points, one distinguished puncture, and a tangent ray at such puncture. These moduli spaces 
assemble to form an operad. It would be nice if this operad could be endowed with a cellular structure, 
and described in terms of the cacti complexes.

\medskip
Here is a plan of the paper:
in section \ref{sec:partrees} we define combinatorially the space of labelled trees $Tr_k$ . 
In section \ref{sec:conf} we construct a homeomorphism between $Tr_k$ and the moduli space $\mathcal{M}_{0,k+1}$ depending 
on positive weights.
In section \ref{sec:cacti} we define the cacti complexes $\C_k$, their homotopy operad structure maps, and compute 
their generating series $P$.  In section \ref{sec:cellular} we define combinatorially the space $N_k(\C)$ of nested trees, we show that 
it is homeomorphic to the configuration space $FM'_k$, and this gives an open cell decomposition of the latter. We then compute the generating
series $o$ counting the cells.  In section \ref{sec:fulton} we recall the definition and properties of the Fulton MacPherson spaces $FM(k)$, and prove that they have regular CW decompositions indexed by metatrees. We compute the associated generating series $F$. 
In section \ref{sec:last} we modify the cell decomposition of the spaces $FM(k)$ so that it becomes compatible with the operad composition,
by getting rid of rotations (in subsection \ref{sub:rot}) and weights (in subsection \ref{sub:weights}).

\section{Space of parametrized trees} \label{sec:partrees}

In this section we construct a combinatorial model $Tr_k$ for the moduli space $$F_k(\bcc)/(\bcc \rtimes \bcc^*)$$
of $k$ ordered distinct points in the complex line $\bcc$ modulo affine transformations of the complex line $\bcc$, 
i.e. modulo translations, rotations and dilations. This space is diffeomorphic (and biholomorphic)
to the moduli space $\mathcal{M}_{0,k+1}$ 
of genus 0 Riemann surfaces with $k+1$ marked points, and 
to the configuration space 
$$F_{k-2}(\bcc-\{0,1\})$$ of $k-2$ points in $\bcc$ with 2 points removed. 
Our model will consist of certain planar trees with black and white vertices, with parameters labelling the black vertices and the angles between edges into the white vertices. The model mimicks the combinatorics of flow lines between critical points and configuration points, as we will see in the next section.
The model is equivariant with respect to the action of the symmetric group on $k$ letters $\Sigma_k$.

\medskip

We start by defining certain directed trees that will index the strata of our model $Tr_k$ .
\begin{Def} \label{uno.uno}
An admissible tree $T$ with $k$ leaves consists of:
\begin{itemize}
\item a set $W=\{v_1,\dots,v_k\}$, whose elements are the {\em white vertices}. 
\item  a finite set $B$, whose elements are the {\em black vertices}.  
\item A finite set $E \subseteq V \times V$, whose elements are the {\em directed edges}, where $V = B \coprod W$ is the set of all vertices.
If $e=(x,y) \in E$, then we say that $e$ has {\em source} $x$ and {\em target} $y$.
We say that a vertex $v \in B \coprod W$  is incident to an edge $e$ if it is either the source or the target of $e$.
\item The set $E_v$ of edges incident to a fixed vertex $v \in V $ is equipped with a cyclic ordering. 
\end{itemize}

The following conditions must be satisfied:
\begin{enumerate}
\item A white vertex is not allowed to be a source
\item If $e=(x,y) \in E$, then $(y,x) \notin E$, i.e an edge cannot be inverted.
\item There are no loops, i.e. there is no set of vertices $\{v_1,\dots,v_m\}$ such that for each
$i\in \{1,\dots, m\}$ either $(v_i,v_{i+1}) \in E$ or $(v_{i+1},v_i) \in E$, where $v_{m+1}=v_1$.
\item Any black vertex $v \in B$ is the source of at least 2 distinct edges. It cannot be the target of two edges that are next to each other in the cyclic ordering of $E_v$.   \label{atleast}
\item The tree is connected, i.e. for any two vertices $v,v' \in V$ there is a sequence $(v=v_1,\dots,v_l=v')$ of vertices such that
for each $i=\{1,\dots,l-1\}$ either $(v_i,v_{i+1}) \in E$ or $(v_{i+1},v_i) \in E$.
\end{enumerate}
\end{Def}

Let $\Tau_k$ be the set of isomorphism classes of admissible trees with $k$ leaves. 

\begin{Def} \label{def:realization}
We associate to each $T \in \Tau_k$ in the standard way its {\em realization}, a 1-dimensional simplicial complex $||T||$ with
a $0$-simplex for each vertex and a $1$-simplex for each edge, oriented from the $0$-cell of the source to the $0$-cell of the target. 
\end{Def}

The simplicial complex $||T||$ is equipped with an embedding into $\bcc$ that is uniquely defined up to isotopy, by the condition that the cyclic ordering of the edges incident to a vertex $v$  correponds to the counterclockwise ordering of the associated $1$-simplexes. Hence $||T||$ is a {\em planar tree}.

\begin{Lem} \label{lem:1.2}
An admissible tree $T$ with $k$ leaves has at most $k-1$ black vertices.
\end{Lem}

\begin{proof}
Let $B,W,E$ respectively the set of black vertices, white vertices, and edges of $T$. The Euler characteristic of $||T||$ is 
$$1=\chi(||T||)=|B|+|W|-|E|,$$ 
so
$|E|=|B|+k-1$ \label{be}
 By condition \ref{atleast}, $|E| \geq 2|B|$.   The substitution of the previous formula for $|E|$ in the inequality yields the result. 
\end{proof}

For example $\Tau_2$ contains a single tree with one black vertex $b_1$, 2 white vertices $v_1, v_2$, and two edges. 
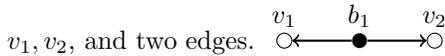
\begin{tikzpicture}  [every node/.style={inner sep=2pt}]
\node (v1) at (0,0) [circle,draw,label=above:$v_1$,inner sep=2pt]  {}  ;
\node (b) at (1,0) [circle,fill,label=above:$b_1$] {} ;
\node (v2) at (2,0) [circle,draw,label=above:$v_2$]  {}  ;
\path [->,thick] (b) edge (v1) edge (v2) ;
\end{tikzpicture}

The set $\Tau_3$ consists of 
\begin{enumerate} 
\item a tree with 1 black vertex $b$ and 3 outgoing edges $(b,v_1), (b,v_2),  (b,v_3)$.
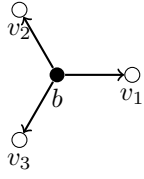
\begin{center}
\begin{tikzpicture}  [every node/.style={inner sep=2pt}]
\node (v1) at (0:1) [circle,draw,label=below:$v_1$]  {}  ;
\node (v2) at (120:1) [circle,draw,label=below:$v_2$]  {}   ;
\node (v3) at (240:1) [circle,draw,label=below:$v_3$]  {}  ;
\node (b) [circle,fill,label=below:$b$]{} ;
\path [->,thick] (b)  edge (v1) edge (v2) edge (v3); 
\end{tikzpicture} 
\end{center}

\item a tree with 2 black vertices $b_1,b_2$, 4 edges $(b_1,v_1), (b_1,v_2),  (b_2,v_2), (b_3,v_3)$.
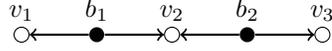
\begin{center}
\begin{tikzpicture}  [every node/.style={inner sep=2pt}]
\node (v1) at (0,0) [circle,draw,label=above:$v_1$]  {}  ;
\node (b1) at (1,0) [circle,fill,label=above:$b_1$] {} ;
\node (v2) at (2,0) [circle,draw,label=above:$v_2$]  {}  ;
\node (b2) at (3,0) [circle,fill,label=above:$b_2$] {} ;
\node (v3) at (4,0) [circle,draw,label=above:$v_3$]  {}  ;
\path [->,thick] (b1) edge (v1) edge (v2) 
(b2) edge (v2) edge (v3) ;
\end{tikzpicture} 
\end{center}

\item Two trees obtained from the tree of case 2)
by permuting cyclically the indices $i$ of $v_i$.
\item A tree with 2 black vertices $b,b'$ and 4 edges $(b,v_1), (b,b'), (b',v_2),(b',v_3)$.
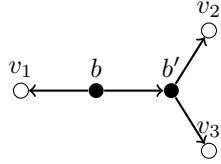
\begin{center}
\begin{tikzpicture}  [every node/.style={inner sep=2pt}]
\node (v1) at (0,0) [circle,draw,label=above:$v_1$]  {}  ;
\node (b) at (1,0) [circle,fill,label=above:$b$]  {}  ;
\node (b') at (2,0) [circle,fill,label=above:$b'$] {} ;
\node (v2) at (2.5,0.8) [circle,draw,label=above:$v_2$]  {}  ;
\node (v3) at (2.5,-0.8) [circle,draw,label=above:$v_3$]  {}  ;
\path [->,thick] (b') edge (v2) edge (v3) 
(b) edge (b') edge (v1) ;
\end{tikzpicture}
\end{center}
\item Two trees obtained from the tree of case 4)
by permuting cyclically the indices $i$ of $v_i$.
\end{enumerate}

We consider next some labellings of these trees.

\begin{Def}
For each tree $T \in 
\Tau_k$, with set of black vertices $B$ , let $\sigma_B$ the topological space of functions
$f: B  \to (0,1]$ such that 
\begin{itemize}
\item for any edge $e=(b_1,b_2) \in E$ of $T$, $f(b_1) \geq f(b_2)$
\item $max(f)=1$.
\end{itemize}
\end{Def}
We observe that any function $f \in \sigma_B$ defines by convex extension
a piecewise linear function $\overline{f}:||T|| \to \R$ such that $\overline{f}(v)=f(v)$ for $v \in B$ and $\overline{f}(v_i)=0$ for $v_i \in W$.
Thus we think of elements of $\sigma_B$ as piecewise linear functions on the realization $||T||$, compatible with the orientations of the edges, with maximum value 1 attained on some black vertex, and minimum value 0 attained exactly on the white vertices.

\begin{Def}
For any white vertex $w \in W$ consider the space $\Delta_w$ of functions  

$g:E_w \to [0,1]$ such that $\sum_{e \in E_w}g(e)=1$.
\end{Def}
This space is a simplex of dimension $|E_w|-1$.

We think of $2\pi g(e)$ as the angle between the direction of the edge $e$ and the next edge in the cyclic ordering.

\begin{Def}
The {\em stratum} associated to a tree $T \in \Tau_k$.
is the space
$$St_T= \sigma_B \times \prod_{w \in W} \Delta_w$$
\end{Def}

\begin{Def}
Let us consider an edge $e=(b,b')$ between black vertices of a tree $T$.
The {\em collapse} of the edge $e$ gives a tree $T/e$ with the same white vertices, a set of black vertices $B/\{b,b'\}$,
and a set of edges $E-\{e\}$. 
\end{Def}

We are finally ready to define our combinatorial model for the moduli space.
\begin{Def} \label{ourmodel}
The space of {\em labelled trees with $k$ leaves} is the quotient
$$Tr_k = (\coprod_{T \in \Tau_k}   St_T ) / \sim $$
with respect to the equivalence relation $\sim$  such that 
\begin{enumerate}
\item For  $e=(b,b')$ and $f(b)=f(b')$
 $$(f:B \to I,  (g_i)_i )_T \sim (f': B/\{b,b'\} \to I, (g_i)_i )_{T/e}$$  
 with $f'([x])=f(x)$. Intuitively this means that we can collapse edges equipped with a constant function.
\item If $g_i(e)=0$ for $e=(b,v_i)$, and the next edge in the cyclic ordering is $e'=(b',v_i)$, then 

\medskip

\begin{itemize}
\item For $f(b')=f(b)$ consider the tree $T'$ with black vertex set $B'=B/\{b,b'\}$, and edge set $E/\{e,e'\}$. Then
$$(f:B \to I,(g_i)_i)_T \sim  (f':B' \to I, (g'_i)_i)_{T'}$$
with $g'_i([x])=g_i(x)$ for $x \in E- \{e\}$ and $f'$ induced by $f$.
\item for $f(b') < f(b)$ consider $T^+$ with the same vertices but the edge $e=(b,v_i)$, replaced by $(b,b')$. Then
$$(f:B \to I, (g_i)_i)_T \sim  (f:B \to I, (g_i)_i)_{T^+}$$  
\item for $f(b')>f(b)$ consider $T^-$ with the same vertices but $e'=(b',v_i)$ replaced by $(b',b)$. Then
$$(f:B \to I, (g_i)_i)_T \sim  (f:B \to I, (g'_i)_i)_{T^-}$$
where $g'_i(e)=g_i(e')$ and $g'_i$ coincides with $g_i$ on other edges incident to $v_i$.
\end{itemize}
\end{enumerate}
Intuitively this means that if the angle between two consecutive edges into a white vertex is zero,
then either the values of the piecewise function on the sources are equal, and  we can merge the two edges into a single one,
or else we can replace the target of the edge from the black vertex with larger value of the piecewise function by the black vertex with smaller value. 
\begin{center}
\begin{tikzpicture}  [every node/.style={inner sep=2pt}]
\node (b1) at (330:1) [circle,fill,label=below:$b$]  {}  ;
\node (b2) at (20:2) [circle,fill,label=below:$b'$]  {}   ;
\node at (.7,0) {0};
\node (w) [circle,draw,label=below:$v_i$]{} ;
\path [<-,thick] (w) edge (b1) edge (b2) ; 
\draw  (0.5,0)  arc (0:15:6mm)  ;
\draw  (0.5,0)  arc (0:-25:6mm);
\huge
\node at (4,0) {$\sim$};
\normalsize
\node (w3)  at (6,0)  [circle,draw,label=below:$v_i$]{} ; 
\node (b4) at (7,0)    [circle,fill,label=below:$b$]  {}  ;
\node (b5) at (8,0) [circle,fill,label=below:$b'$]  {}  ;
\path [<-,thick] (w3) edge (b4) (b4) edge (b5) ; 
\end{tikzpicture} 
\end{center} 
\end{Def} 

We will construct in section \ref{sec:conf} a family of homeomorphisms 

$$\Psi_{a_1,\dots,a_k}:Tr_k \to \mathcal{M}_{0,k+1}$$

depending  on a choice of positive weights $(a_1,\dots,a_k)$ adding up to $1$.
A similar dependence on parameters appears in the combinatorial model for the moduli space by  Harer-Penner-Thurston \cite{Mondello}.

We need first a combinatorial construction.
For an admissible tree $T \in \Tau_k$, we construct a {\em ribbon graph} $\overline{T}$ by adding a new vertex $\infty$ to $T$ and some new edges with source $\infty$.  We will add (possibly multiple) edges with source $\infty$ and target a black vertex $b$, for any black vertex $b$. 
In this new context an edge is not determined by its source and target, and we need the following definition.
\begin{Def}
A ribbon graph $G$ consists of 
\begin{itemize}
\item A finite set $E$ of edges
\item  A finite set $V$ of vertices 
\item Function $s:E \to V$ (source) and  $t:E \to V$ (target) such that $s(e) \neq t(e)$ for any $e$ {\rm (we do not allow loops with a vertex)}
We write $e:s(e) \to t(e)$. 
\end{itemize}
For any $v \in V$ the set of edges incident to $v$ 
$$E_v= \{ e  \in E  \,|\,  v=s(e) \;{\rm  or }\; 	v=t(e)  \}.$$ 
is equipped with a cyclic ordering $\sigma_v$.
\end{Def} 
\begin{Rem}
If $G$ is a tree, then the map $(s,t):E \to V \times V$ is injective and we can identify $E$ to a subset of $V \times V$ as we did in 
definition \ref{uno.uno}
\end{Rem}

The following definition extends definition \ref{def:realization}.
\begin{Def}
The {\em realization} $||\overline{T}||$ of a ribbon graph $\overline{T}$ is a 1-dimensional  CW complex (that in general is not
 a simplicial complex since the vertices do not determine an edge) 
with a 0-cell for each vertex, and a 1-cell for each edge, so that   
$||\overline{T}||= V \coprod (E \times [0,1]) /  \sim$ with 
$s(e) \sim (e,0)$ and  $t(e) \sim (e,1)$.
\end{Def}

Given an admissible tree $T$, 
we define a ribbon graph $\overline{T}$ by adding edges from a new vertex $\infty$ so that ingoing and outgoing edges  at any black vertex alternate in the cyclic ordering.
\begin{Def}
Let $T$ be an admissibile tree with $k$ leaves, with set of vertices $V= B \coprod W$. We define a ribbon graph $\overline{T}$ as follows:
\bit
\item The set of vertices of $\overline{T}$ is $\overline{V}= V \coprod \{\infty\}$
\item For a black vertex $b \in B$ let $o(b)$ the number of edges with source $b$ and $i(b)$ the number of  edges with target $b$.
We know that $o(b) \geq i(b)$ by condition (4) of Definition \ref{uno.uno}. For $o(b)>i(b)$ we add $o(b)-i(b)$ new edges with source $\infty$ and target $b$. 
Let $$E^{''}_b= \{e_1,\dots, e_{o(b)-i(b)} \} \subset E_b$$ 
 be the set of outgoing edges such that  the successor $\sigma_b(e_i)$ in the cyclic ordering 
is also outgoing. For each $e_i$ we add a new edge $e'_i$ with $s(e'_i)=\infty$ and $t(e'_i)=b$.
Let us write $$E'_b= \{e'_1,\dots,e'_{o(b)-i(b)}\}$$
 and $E'=\coprod_{b \in B} E'_b$.
Then the set of edges of $\overline{T}$ is $\overline{E}= E \coprod E'$. 
\item  We define the cyclic ordering $\overline{\sigma}_b$ on $\overline{E}_b = E_b \coprod E'_b$ so that $\overline{\sigma}_b(e_i)=e'_i$,
$\overline{\sigma}_b(e'_i)=\sigma_b(e_i)$ and $\overline{\sigma}_b(e)=\sigma_b(e)$ if $e \in E_b \setminus E^{''}_b$.
\eit
We must specify the cyclic ordering of the set of edges $E'=\coprod_b E'_b$ coming out of $\infty$.

The set of half-edges of $T$ is
$$HE=\{(e,s(e)) | e \in E\} \coprod \{(e,t(e)) | e \in E\}$$ 
This set is equipped with an involution $I: HE \to HE$ swapping source and target of the edges such that 
$I(e,s(e))=(e,t(e))$ and $I(e,t(e))=(e,s(e))$.
It is equipped with a cyclic ordering $\sigma_\infty$ where 
$\sigma_\infty(e,v)= I(\sigma_{v}(e),v)$. 
This induces by restriction 
a cyclic ordering on the subset of edges $\{(e,s(e)) | e \in E^{''}_{s(e)} \}$, hence a cyclic ordering on 
$\coprod_{b \in B} E^{''}_b \cong  \coprod_{b \in B} E'_b =E'$. We choose the
opposite
cyclic ordering  
$\overline{\sigma}$ on $E'=\overline{E}_\infty$.
\end{Def}
\begin{figure}
\begin{tikzpicture}  [every node/.style={inner sep=2pt}]
\node (v1) at (0,0) [circle,draw,label=above:$v_1$]  {}  ;
\node (b) at (1,0) [circle,fill,label=above left:$b$]  {}  ;
\node (b') at (2,0) [circle,fill,label=above:$b'$] {} ;
\node (v2) at (2.5,0.8) [circle,draw,label=below right:$v_2$]  {}  ;
\node (v3) at (2.5,-0.8) [circle,draw,label=above right:$v_3$]  {}  ;
\path [->,thick] (b') edge (v2) edge (v3) 
(b) edge (b') edge (v1) ;
\node (inf) at (5,0) [circle,draw,label=right:$\infty$]{};
\path [->] (inf) edge (b') (inf) edge [bend right=90] (b) (inf) edge[bend left=90] (b) ;
\end{tikzpicture}
\caption{Example of a ribbon graph $\bar{T}$}
\end{figure}
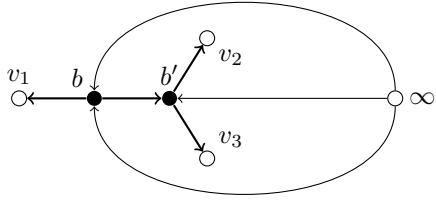
\section{Configurations and labelled trees} \label{sec:conf}

\subsection{From labelled trees to configurations}
We will define a $\Sigma_k$-equivariant function
$$\Psi_{a_1,\dots,a_k} : Tr_k  \to F_k(\bcc)/(\bcc \rtimes \bcc^*) \cong \mathcal{M}_{0,k+1}$$

depending on positive numbers $a_1,\dots,a_k$  called {\em weights}, such that $\sum_{i=1}^k a_i=1$, i.e.
  $(a_1,\dots,a_k)$ is an element of the standard open $(k-1)$-simplex $\stackrel{\circ}{\Delta}_{k-1}$.

In the following definition we fatten each edge $e$ of a labelled tree $tr$, with underlying admissible tree $T$, into a ribbon $F_e$ with pinched ends, and then we glue together these ribbons and the realization of the ribbon graph $||\overline{T}||$, according to the data of the labelled tree,  getting a space $F_{tr}$.

\begin{Def}
\label{directed}

Let $(a_1,\dots,a_k) \in \stackrel{\circ}{\Delta}_{k-1}$ be fixed weights.
Let $ (f,(g_1,\dots,g_k))_T$ represent a labelled tree $tr \in Tr_k$, where $T \in \Tau_k$ is an admissible tree with $k$ leaves.
We can assume that $g_i$ is a positive function for any vertex $v_i$. 

The  realization $||\overline{T}||$ is equipped with a continuous map $\overline{f}:||\overline{T}|| \to [0,+\infty]$ 
extending the map denoted similarly in the previous section.
We define it on the new vertex by $\overline{f}(\infty)=\infty$, and extend it by an increasing arbitrary homeomorphism $\{e\} \times [0,1] \cong [f(b), +\infty]$
on any new edge $e$ coming out of $\infty$ and into $b$.

We associate to each edge $e \in E_{v_i}$ with target a white vertex $v_i$ the quotient 
$$F_e =   ([-\infty,+\infty] \times [0,a_i g_i(e )])/\sim $$
where in the quotient the intervals $$ \{+\infty\} \times [0,a_i g_i(e)] \quad{\rm and }\quad 
 \{-\infty\} \times [0,a_i g_i(e)] 	$$ are collapsed to distinct points, that we call respectively $+\infty$ and $-\infty$.
We define a quotient space
$$F_{tr} :=\coprod_{i=1}^k ( \coprod_{e \in E_{v_i}} F_e ) \coprod ||\overline{T}|| \; / \sim $$
with respect to the following equivalence relation.

Given an edge $e:b \to b'$ of $\overline{T}$  consider the uniquely defined directed path starting with $b$ and ending in a white vertex consisting of edges  $e_j:b_{j-1} \to b_{j}$  for $j=1, \dots, m$ and suitable $m$ , such that $e=e_0, \, b_m=v_i$ is a white vertex, and 
$e_{j+1}$ precedes $e_j$ in the cyclic ordering at $\overline{E}_{b_j}$ for each $j=1,\dots,m-1$. The last edge $e_-=(b_{m-1},v_i) \in E_{v_i}$ labels
the space 
 $F_{e_-}$. 
 We call this path the {\em left path}, because intuitively if we move along the path, then at an intersection point we always choose the left exit.
For $x \in \{e\} \times [0,1] \subset  ||\overline{T}||$  we identify $x \sim (\log(\overline{f}(x)),a_ig_i(e_-))_{F_{e_-}}.$ 
The convention is that $\log(0)=-\infty$ and $\log(+\infty)=+\infty$.

 Viceversa consider a similar directed path, called the {\em right path}, 
 where instead each edge $e'_{j+1}$ follows $e'_{j}$ in the cyclic ordering for $j=1,\dots,n$, 
 the first edge is $e'_0=e$ and  
$e_+:=e'_n$ goes into a white vertex $v_{i'}$.
Then for $x \in \{e_+\} \times [0,1]$ we identify  $x \sim (\log(\overline{f}(x)),0) \in F_{e_+}$. 
\end{Def}
\begin{center}
\begin{tikzpicture}
\node (b) at (0,1)  [circle,fill,inner sep=1pt,label=above:$b$]{};
\node (b') at (1,1)  [circle,fill,inner sep=1pt,label=above:$b'$]{};
\node (b1) at (3,2)  [circle,fill,inner sep=1pt]{};
\node (v1) at (4,2)   [circle,draw,inner sep=2pt,label=above:$v_i$] {};
\node (b2) at (2.5,0)  [circle,fill,inner sep=1pt]{};
\node at (3.5,1.6) {$e_-$};
\node at (3.5,0.4) {$e_+$};
\node (vj) at (4,0)   [circle,draw,inner sep=2pt,label=below:$v_{i'}$] {};
\node at (2.3,1.3) {$F_{e_-}$};
\node at (1.6,0) {$F_{e_+}$};
\path [->] (b) edge (b'); 
\path [->] (b') edge [dotted] (b1); 
\path [->] (b') edge [dotted] (b2); 
\path[->] (b1) edge (v1);
\path[->] (b2) edge (vj);
\node at (.5,1.2) {$e$};
\end{tikzpicture}
\end{center}

\begin{Lem} \label{count}
The space $F_{tr}$ is a 2-dimensional finite CW complex with

 $\chi(F_{tr})=2$.
\end{Lem}
\begin{proof}
The 1-skeleton of $F_{tr}$ is $||\overline{T}||$. The 2-cells are the subspaces $F_e$, with the quotient map from a square as characteristic map.
 They are indexed by the set $F=\cup_{i=1}^k E_{v_i} \, $.
Let us compute the Euler characteristic. 
The number of edges going into white vertices  is equal to the number $|F|$ of 2-cells. Hence the number of edges into black vertices is 
$|E|-|F|$. The number of edges going out of black vertices in $T$ is of course $|E|$. In the ribbon graph $\overline{T}$ we added edges 
from $\infty$ into black vertices to make the number of ingoing and outgoing edges at black vertices equal. The number of new edges is then 
$|E|-(|E|-|F|)=|F|$. Counting also $k$ white vertices and $\infty$,  the Euler characteristic of $F_{tr}$ is then 
$$\chi(F_{tr})= |B|+k+1-|E|-|F|+|F|=2$$ by  the proof of lemma \ref{lem:1.2}.
\end{proof}

\begin{Prop}
There are explicit charts making 
$F_{tr}$ into a genus 0 Riemann surface.
\end{Prop}

\begin{proof}
We cover $F_{tr}$ by a finite number of charts.  
\begin{enumerate}
\item The interior of each 2-cell $F_e$, with $e \in F$, is identified to the open set
 $$\{z \in \bcc \, |  \, 0<Im(z)<a_i g_i(e) \}$$
  by the chart $\phi^{(2)}_{e}:(x,y)_{F_e} \mapsto (x+iy)$.

\item We specify the complex structure near the open 1-cells. Given an edge $e:v_1 \to v_0$ of $\overline{T}$,
the open 1-cell 
$$\{e\} \times (0,1) \subset ||\overline{T}|| \subset F_{tr}$$
 intersects two 2-cells $F_l$ and $F_r$ (that might coincide),
where $l$ and $r$ are the edges obtained following respectively the left and the right path of definition \ref{directed}.
Suppose that $r$ goes into the white vertex
$v_h$.

 Let us consider the open set 
 $$V=\{z \in \bcc \, | \, log(\overline{f}(v_0))< Re(z)<log (\overline{f}(v_1)) \, ,  \,  -\varepsilon < Im(z) <+\varepsilon \} $$
For small $\ep>0$,
an open neighbourhood $U \supset \{e\} \times (0,1)$  is given by the union $U=U_l \cup U_r$, with 
$$U_r=   \{(x,y)_{F_r} \, | \, \phi_r^{(2)}(x,y) \in V\}$$ 
and $$U_l=  \{(x,y)_{F_l} \, | \,  \phi_l^{(2)}(x,y)-a_h g_h(l){\bf i} \in V\}.$$
The chart  $\phi^{(1)}_e:U \to V$ coincides with $\phi^{(2)}_r$ on $U_r$ and with $\phi^{(2)}_l - a_h g_h(l){\bf i}$ on $U_l$.

\item Let us consider the neighbourhood of $\infty \in F_{tr}$ 
$$Z=\bigcup_{F_e}\{ (x,y)_{F_e} \, | \, x>M \}$$ for $M$ sufficiently large.  
 By the proof of lemma \ref{count} 
the set $E_\infty$ 
of  edges coming out of $\infty$  and the set of edges incident to white vertices $F$, indexing the 2-cells, have the same cardinality.
A bijection $\kappa: E_\infty \to F$ is constructed by following the left directed path from $e \in E_\infty$ 
to a white vertex, terminating with an edge $\kappa(e)=e_-$.  
The map is surjective since the path construction can be reversed and therefore any 
edge of $F$ comes from an edge of $E_\infty$. Hence $\kappa$ is bijective.  Let us denote $g: F \to \R^+$ the function
such that $g(e)=g_i(e)$ if $e$ is adjacent to  the white vertex $v_i$. Observe that 
\begin{equation} \label{sum}
\sum_{e \in F}g(e)= \sum_{i=1}^k \sum_{e \in E_{v_i}} g_i(e) = \sum_{i=1}^k a_i = 1 \, .
\end{equation}
We associate to any $e \in E_\infty$ the positive number $g(e_-)$. 
We choose a chart from $Z$ to $\{ z \in \bcc \, | \, Re(z)<e^{-M} \}$. 
Let us choose arbitrarily some $e_1 \in E_\infty$. The opposite cyclic ordering of $E_\infty$ induces a linear ordering on its elements
$e_1< \dots< e_m$. On $Z \cap F_{{e_i}_-}$ the chart sends 
$$\phi^{(0)}_\infty: (x,y)_{F_{{e_i}_-}} \mapsto  exp(-2\pi(x+{\bf i} (y+ \sum_{j=1}^{i-1} g({e_i}_-)  ) ))$$
By equation \ref{sum} the chart is continuous. 
\item Near a white vertex $v_i$ consider a neighbourhood $Z_i$ consisting of all elements $(x,y)_{F_e}$, with $e$ incident to $v_i$, 
such that $x<-N$, with $N$ large enough. As before, the cyclic ordering of the vertices in $E_{v_i}$, after choosing some $e'_1 \in E_{v_i}$, induces a linear ordering $e'_1< \dots <e'_n$ on its elements. 
A chart to $W= \{ z \in \bcc \, | \, Re(z)<e^{-N} \}$ is given by 
$$\phi^{(0)}_{v_i}:(x,y)_{F_{e_j}} \mapsto exp(\frac{2\pi}{a_i}(x+{\bf i}(y+ \sum_{l=1}^{j-1} g_i(e'_l))))   $$
The chart is continuous since $\sum_{l=1}^n g_i(e_l)=a_i$.  

\item We are left with the black vertices. 
Let $b$ be a black vertex. It has $m$ edges coming in and $m$ going out, with $m \geq 2$, that alternate in the cyclic ordering at $b$.
Each incoming $e$ as 1-cell intersects a {\em left} 2-cell $F_{e_-} \, $, labelled by the terminal edge of its left path $e_- \;$, and 
a {\em right} 2-cell $F_{e_+}$ labelled by the terminal edge of its right path $e_+ \;$. 
Let $e'$ be the outgoing edge following $e$ in the cyclic ordering at $b$. Then $e_+= e'_+$.
If $e''$ is the outgoing edge preceding $e$ in the cyclic ordering $b$, then $e_- = e^{''}_-$.
Suppose that the cyclic ordering at $b$ induces the linear ordering $e=e_1,\,e'=e'_1 \, ,\,\dots \, , \, e_m, \, e''=e'_m$ . 
\begin{figure}
\begin{tikzpicture}
\node (b) at (1,1)  [circle,fill,inner sep=1pt,label=above left:$b$]{};
\node (1) at (1,2) {$e$};
\node (2) at (2,1) {$e''$};
\node (3) at (0,1) {$e'$};
\node (4) at (1,0) {$e_3$};
\path [->] (1) edge (b) (b) edge (3)  (4) edge  (b) (b) edge (2);
\node at (0,1.8) {$F_{e_-}$};
\node at (2,1.8) {$F_{e_+}$};
\end{tikzpicture}
\caption{Edges incident to $b$}
\end{figure}
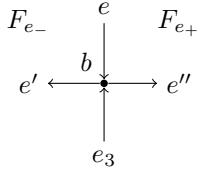
Consider the open neighbourhood $X=X_+ \cup X_-$ of $b$  with
\begin{align*}
X_+= &\bigcup_j\{(x,y)_{F_{(e_j)_+}} \, | \, (x-log(f(b)))^2 +y^2 < \ep^2 \}  \\
X_-= & \bigcup_j\{(x,y)_{F_{(e_j)_-}} \,  | \, (x-log(f(b)))^2 +(y-a_{i(j)} g_{i(j)}((e_j)_-))^2 < \ep^2 \} 
 \end{align*}
where $i(j)$ is the index of the white vertex target of $(e_j)_- \;$ . 
 
Consider the open disk $ D= \{ z \in \bcc \, \, | \,  \, |z| <\ep \}$ for $\ep>0$ small enough.

Let $\sqrt[m]\;: \bcc \to \bcc$ be the branch of the $m$-th complex root given by $\arg(\sqrt[m]{z})=\frac{\arg(z)}{m}$, 
that is continuous for $Im(z) \geq 0$.

A chart $\phi^{(0)}_b: X \to D$ is given 
by 
\begin{align*}
 (x,y)_{F_{(e_j)_+}} &\mapsto \sqrt[m]{ x-log(f(b))+{\bf i} y}\cdot  exp(\frac{2\pi {\bf i}(j-1)}{m}) , \\
 (x,y)_{F_{(e_j)_-}} &\mapsto \sqrt[m]{-x+log(f(b))+{\bf i} (-y+a_{i(j)} g_{i(j)}({e_j}_-))}\cdot  exp(\frac{2\pi {\bf i}(2j-3)}{2m})    
 \end{align*}
Namely the map $\phi^{(0)}_b$ is well defined on the intersections of 2-cells and is a homeomorphism.  
\end{enumerate}
It is easy to see that all changes of coordinates are holomorphic functions. This defines a Riemann surface structure on $F_{tr}\;$.
By Lemma \ref{count} the genus of $F_{tr}$ is zero.

\end{proof}

We are now ready to define the function
$$\Psi_{a_1,\dots,a_k} : Tr_k  \to F_k(\bcc)/(\bcc \rtimes \bcc^*)$$

\begin{Def} \label{def:biholo}
For $tr \in Tr_k$, by the uniformization theorem 
 $F_{tr}$ is biholomorphic to  $\CP^1=\bcc \cup \infty$.  
Choose a biholomorphism $\zeta:F_{tr} \to \CP^1$ satisfying $\zeta(+\infty)=\infty$. 
Then $z_i=\zeta(v_i)$ for $i=1,\dots,k$ are distinct points in $\bcc$ and define an element of the moduli space $F_k(\bcc)/ (\bcc \rtimes \bcc^*)$.
We set $\Psi_{a_1,\dots,a_k} (tr):=[z_1,\dots,z_k]$. 
\end{Def}
 
 \begin{Exa}
 If the underlying tree of $tr$ is 
\begin{center}
  \begin{tikzpicture}
\node (v1) at (0,0) [circle,draw,label=above:$v_1$,inner sep=2pt]  {}  ;
\node (b) at (1,0) [circle,fill,label=above:$b$, inner sep=1pt] {} ;
\node (v2) at (2,0) [circle,draw,label=above:$v_2$,inner sep=2pt]  {}  ;
\path [->,thick] (b) edge (v1) edge (v2) ;
\node at (.5,.2){$e$} ;
\node at (1.5,.2){$e'$};
\end{tikzpicture}
\end{center}

 then $F_{tr}=F_{e} \cup F_{e'}$, where $F(e)$ is the blue region and $F(e')$ the red region; the outer circle represents $\infty$ and  
 $\Psi(tr)=[-1,1]$. 

\begin{center}
\begin{tikzpicture}
  \node (start) at (-1,0) [circle,draw,inner sep=2pt]{};
  \foreach \angle in {0, 10, ..., 360}
    \draw [blue](node cs:name=start,angle=\angle)
   .. controls +(\angle:.9cm)
     ..  (\angle/2+90:2cm);
  \node (stop) at (1,0) [circle,draw,inner sep=2pt]{};
  \foreach \angle in {180, 190, ..., 540}
    \draw [red] (node cs:name=stop,angle=\angle)
   .. controls +(\angle:.9cm)
     ..  (\angle/2+180:2cm);
\node (ori) [circle,fill,inner sep=1pt]{};
\draw [very thick] (start) -- (ori);
\draw [very thick] (stop) -- (ori);
\node (nord) at (0,2) {};
\node (sud) at (0,-2) {};
\draw (ori) -- (nord);
\draw (ori) -- (sud);
\draw circle (2cm);

  \end{tikzpicture}
\end{center}
\end{Exa}

\subsection{From configuration spaces to labelled trees}

We construct the inverse function $\Phi_{a_1,\dots,a_k}: F_k(\bcc)/(\bcc \rtimes \bcc^*) \to Tr_k$ to 
$\Psi_{a_1,\dots,a_k}$. See the figure for an example of a value of $\Phi$.

\begin{tikzpicture}  
\node (1) at (0,0) [circle,draw,inner sep=2pt,label=left:$z_1$] {};
\node (2) at (0,.8) [circle,draw,inner sep=2pt,label=left:$z_2$] {};
\node (z2) at (2.5,1.2) [circle,draw,inner sep=2pt,label=below:$z_3$] {};
\draw (1,0.8) circle (2.2cm);
\huge
\node at (4.5,.5) {$\stackrel{\Phi}{\mapsto}$};
\normalsize
\node (3) at (6.5,0) [circle,draw,inner sep=2pt,label=left:$v_1$] {};
\node (4) at (6.5,.8) [circle,draw,inner sep=2pt,label=left:$v_2$] {};
\node (34) at (6.5,.4) [circle,fill,inner sep=1pt,label=left:$b$]{};
\path [->] (34) edge (3)  (34) edge (4);
\node (5) at (9.5,1.2) [circle,draw,inner sep=2pt,label=above:$v_3$] {};
\node (w) at (8,1) [circle,fill,inner sep=1pt,label=above:$b'$] {};
\path [->] (w) edge (5) (w) edge (4) ; 
\end{tikzpicture}

Given pairwise distinct points $z_1,\dots,z_k \in \bcc$ and positive numbers $a_1,\dots,a_k$ with $\sum_j a_j = 1$ 
consider the multivalued holomorphic function 
$$h(z)=\prod_{i=1}^k (z-z_i)^{a_i} = exp(\sum_{i=1}^k  a_i log(z-z_i) )$$ on $\bcc-\{z_1,\dots,z_k\}$. 
Its norm $|h(z)|=\prod_{i=1}^k |z-z_i|^{a_i}$ is univocally defined. 

The derivative $h'$ vanishes in $z \in \bcc-\{z_1,\dots,z_k\}$ if and only if 
$$h'(z)/h(z)=d(\log(h(z)))/dz= \sum_{i=1}^k  \frac{a_i}{z-z_i} =0 $$
 if and only $$p(z):= \sum_{i=1}^k {a_i}\prod_{j \neq i}(z-z_j) =0 \;.$$

 The monic polynomial $p$ has at most $k-1$ distinct roots. Let $Crit$ be the set of its roots, that are the critical points of $h$.
 
 We study curves $\gamma: (a,b) \to  \bcc -\{z_1,\dots,z_k\}-Crit$ whose image $\Gamma={\rm Im}(\gamma)$ is defined locally by the equation 
 $$\Gamma = \{z \in \bcc \, | \, \arg(h(z))=\theta \} $$ with $\theta \in \R$ constant.
 Such a curve $\gamma$ is non-singular and well defined since 
 $\arg(h(z))= \sum a_i \arg(z-z_i)$ has a continuous branch
 on any simply connected domain $D \subset \bcc - \{z_1,\dots,z_k\}$,
 defined up to addition by a constant $\sum_{i=1}^k{n_i a_i}$, with $n_i \in \bbz$. 
We claim that the curve $\gamma$ is an integral flow line of the vector field 
 $$E(z)= \sum_{i=1}^k \frac{a_i (z-z_i)}{|z-z_i|^2}\, ,$$ that is non-singular on $\bcc - \{z_1,\dots,z_k\}-Crit(f)$.
 The field is the superposition of $k$ radial fields $E_i(z)=\frac{a_i (z-z_i)}{|z-z_i|^2}$
 respectively centered at $z_i$, of strength inversely proportional to the distance from the points $z_i$, and directly proportional to the parameters $a_i$. 
This idea is at the origin of our combinatorial model for the configuration space. 
Namely the field $E$ admits a potential $$Re(\log(h(z)) =\log|h(z)|=\sum_{i=1}^k a_i \log|z-z_i|\, ,$$
and so by a standard argument of complex analysis \cite{Ahlfors} the flow lines of $E$ are exactly the level curves of 
 $$Im(\log(h(z))=\arg(h(z)).$$
 
 Furthermore, using the fact that the multivalued analytic function $\log(h(z))$ is defined outisde $\{z_1,\dots,z_k\}$ and has non-zero derivative outside $Crit$, we have that the family of the curves $\gamma$, as $\theta$ varies, define a foliation with non-compact leaves of $\bcc-\{z_1,\dots,z_k\}-Crit$, that have a regular parametrization by $t=|h(z)|$. Let us study their behavior near a critical point $z_0 \in Crit$. 

For us a flow line of $E$ is a a maximal connected curve $$\gamma:(a,b) \to \bcc-\{z_1,\dots,z_k\}-Crit$$
satisying $\gamma'(t)=E(\gamma(t))$ for any $t \in (a,b)$. 

\begin{Def}
We say that a flow line $\gamma:(|h(z_0)|,a) \to \bcc-\{z_1,\dots,z_k\}-Crit$
 of $E$ is incoming at $z_0 \in Crit$  if 
$$\lim_{t \to |h(z_0)|^+} \gamma(t)=  z_0,$$
and a flow line line $\gamma:(b,|h(z_0)|) \to \bcc-\{z_1,\dots,z_k\}-Crit$ is outgoing at $z_0$ if 
$$\lim_{t \to |h(z_0)|^-} \gamma(t)= z_0$$
\end{Def}

\begin{Lem}\label{lem:crit}
Given a critical point  in $z_0 \in Crit$ that is a zero of order $m-1$ of $p$, there are exactly $m$ incoming and $m$ outgoing flow lines at $z_0$.
\end{Lem}
\begin{proof} The multivalued function $\log(h(z))-\log(h(z_0))$ has in $z=z_0$ a zero of order $m$.  
Therefore near $z_0$ there is a holomorphic coordinate $w$ with $w(z_0)=0$ and $\log(h(z)/h(z_0))=w^m$.  
The flow lines are described in this coordinate system by $Im(w^m)=c$ with fixed $c \in \R$.  The incoming and outgoing flow lines are locally
the components of $$\{w \in \bcc^* \, | \, Im(w^m)=0\}\, .$$ The outgoing flow lines 
correspond to $$\arg(w)=e^{\frac{2 \pi {\bf i}j}{m}},$$ for $j=0,\dots,m-1$, since $\log|h(z)|-\log|h(z_0)|=Re(w^m)=w^m>0$ 
on such components. Instead the incoming flow lines correspond
to $$\arg(w)=e^{\frac{2 \pi {\bf i}(2j+1)}{2m}}$$ for $j=0,\dots,m-1$, since $\log|h(z)|-\log|h(z_0)|= Re(w^m)=w^m<0$  
on such components. 
\end{proof}

This shows that the index of the vector field $E$ at the singular point $z_0$ is $-m$.

It is easy to see that near a zero $z_i$ of the function $h$ 

By performing a change of variables
$$w^{a_i} = h(z)=(z-z_i)^{a_i} \prod_{j \neq i} (z-z_j)^{a_j}$$ near $z_i$ we see that
the level curves, or flow lines, near $z_i$  in an appropriate coordinate system 
have the form $\gamma(t)=z_i + t w$, with $|w|=1$. The flow lines are all incoming into $z_i$ since $\lim_{t \to 0^+}\gamma(t)=z_i$. 
So the index of the vector field $E$ at the singular point $z_i$ is $+1$.

Near $\infty$ we perform a change of variables $$w^{a_1+\dots+a_k} = h(z)$$
and see that the flow lines have the form $\gamma(t)=t\cdot \zeta$, for $t$ large enough. The flow lines are all outgoing from  $\infty$ since 
$\lim_{t \to +\infty}\gamma(t)=\infty$. The index of $E$ at the singular point $\infty$ is also $+1$.

Of course this computation respects the Poincare-Hopf index theorem since the index contributions are $-(k-1)$ from the black vertices, $k$ from the zeros of $h$
and $+1$ from $\infty$, so that $-(k-1)+k+1= 2 = \chi(F_{tr})$.

By the characterization of the flow lines, and the compactness of $\CP^1$, any flow line $\gamma:(a,b) \to \bcc-\{z_1,\dots,z_k\}-Crit$ extends to the embedding 
of a closed interval $\bar{\gamma}:[a,b] \to \CP^1$, with $$\bar{\gamma}(a),\bar{\gamma}(b) \in Crit \cup \{z_1,\dots,z_k\} \cup \infty \, ,$$
and possibly $a=-\infty$ or $b=+\infty$. We call the extension a {\em closed flow line}.

\begin{Def}\label{deftree}
Let us start from $(z_1,\dots,z_k) \in F_k(\bcc)$.
 We associate a tree with black and white vertices $T=T(z_1,\dots,z_k) \in \Tau_k$ to the configuration, defined as follows.

\bit
\item  There are $k$ white vertices $\{v_1,\dots,v_k\}$ associated respectively to the points $\{z_1,\dots,z_k\}$.
\item The set of black vertices is the set $Crit$ of critical points of $h$.
\item An edge $e:b \to v$ of $T$ with source $b$ and target $v$  is a flow line of $E$ going out of the critical point $b$ 
and into either another critical point $v$, or into a zero $v=v_i$ of $h$.  
\item The anti-clockwise ordering of flow lines induced by the standard orientation of $\bcc$ 
defines a cyclic ordering on the set of edges  incident to a given vertex. 
\eit
\end{Def}

\begin{Prop}
The data in Definition \ref{deftree} define an admissible tree  $T$ with $k$ leaves.
\end{Prop}
\begin{proof}
We check the conditions in defintion \ref{uno.uno}.
Condition (1) is satisfied because zero is not a source of a flow line. 
Condition (2) is satisfied because a flow line from $b$ to $v$ always points towards a vertex with $|h(v)| < |h(b)|$. 
Condition (4) is satisfied because  the incoming and outcoming flow lines at a critical point alternate, we removed the incoming flow lines coming from $\infty$, and there are $m>1$ outgoing flow lines, where $m$ is as in Lemma \ref{lem:crit}.
The realization $||T||$ can be embedded as a subcomplex of $\CP^1$ by sending each vertex to the corresponding point, and each edge to the corresponding closed flow line. 
In order to verify conditions (3) and (5) it is sufficient to check that $||T||$ is contractible. The complement $\CP^1-||T||$ contracts into $\infty$ by following the flow lines,
so by Alexander duality $||T||$ has trivial homology, and is a contractible 1-complex.  
\end{proof}

 \begin{Def} 
We define next a labelled tree $\overline{\Phi}_{a_1,\dots,a_k}(z_1,\dots,z_k) = tr \in Tr_k$ with $k$ leaves
consisting of the following data, according to definition \ref{ourmodel}:

\bit
\item  The tree $T=T(z_1,\dots,z_k)$  from Definition \ref{deftree}. 
\item The function $f : Crit \to (0,1]$ defined
for each black vertex $b \in Crit$ 
by $f(b)=|h(b)|/M \in (0,1]$,
where $M=\max\{|h(x)| \, , \, x \in Crit  \}$. 
\item For each $i=1,\dots,k$ the following function $g_i:E_{v_i} \to (0,1]$. 
For any incoming edge (flow line) $e \in E_{v_i}$ at a white vertex $v_i$, 
let $e' \in E_{v_i}$ be the next  edge (flow line) in the cyclic ordering. We choose a continuous 
branch of $\arg(h)$ defined for the flow lines between $e$ and $e'$, as one moves in the anti-clockwise direction around $z_i$.
The function $\arg(h)$ is constant along the flow lines $e$ and $e'$. The difference of the two values $\arg(e)<\arg(e')$ is unambigously defined. Then
 $$g_i(e):= \frac{\arg(e')-\arg(e)}{2\pi a_i} \in (0,1].$$
We include the special case when $e=e'$ declaring $g_i(e)=1$.
\eit
\end{Def}

\begin{Lem}
The construction $\overline{\Phi}_{a_1,\dots,a_k}(z_1,\dots,z_k)$ is invariant under the action of the affine group  $\bcc \rtimes \bcc^*$ 
on $F_k(\bcc)$, and thus defines a function 
$$\Phi_{a_1,\dots,a_k}: F_k(\bcc) / (\bcc \rtimes \bcc^*) \rightarrow Tr_k$$ 
\end{Lem}

\begin{proof}
The function $h$ depends on the $n$-tuple of points $Z=(z_1,\dots,z_k)$, hence we write $h=h_Z$, and so do $p=p_Z$ and 
$Crit=Crit_Z$. Consider an element of the affine group $\alpha:z \mapsto \lambda z + \mu$ with $\lambda \in \bcc^*$ and $\mu \in \bcc$,
and set $\alpha(Z)=(\alpha(z_1),\dots,\alpha(z_k))$. 
The roots of the rational function 
$$p_{\alpha(Z)} = \sum_{i=1}^k \frac{a_i}{z-\alpha(z_i)}$$
 form the set of critical points $Crit_{\alpha(Z)}=\alpha(Crit_Z)$.
 By a slight abuse let us identify a flow line and its image.
The flow line $\gamma$
 associated to $Z$ has locally the form $\{ z \, | \, \sum_{i=1}^k a_i \arg(z-z_i)=c \}$.
The corresponding flow line associated to $\alpha(Z)$ has locally the form $\{z \, | \, \sum_{i=1}^k a_i \arg(z-\alpha(z_i))=  c + \arg(\lambda)\}$, 
so it is just $\alpha(\gamma)$.
 The differences of the arguments of flow lines into the zeros of $h_{\alpha(Z)}$ are the same as for $h_Z$. 
The value of $|h_{\alpha(Z)}|$ at the point $\alpha(b)$ is
$$|h_{\alpha(Z)}(\alpha(b))| = |\lambda|^{\sum_{i=1}^k a_i  }   \cdot |h_Z(b)|=|\lambda| \cdot |h_Z(b)|$$ 
So all values get multiplied by the constant $|\lambda|$, and their ratio remains the same. 
Therefore $\Phi_{a_1,\dots,a_k}(Z)= \Phi_{a_1,\dots,a_k}(\alpha(Z))$.

\end{proof}
\begin{Rem}
Let us write 
$$\Phi(a_1,\dots,a_k;z_1,\dots,z_k)= \Phi_{a_1,\dots,a_k}(z_1,\dots,z_k),$$  
$$\Psi(a_1,\dots,a_k;tr)= \Psi_{a_1,\dots,a_k}(tr).$$

We notice that the function $$\Phi: \, \stackrel{\circ}{\Delta}_k \times F_k(\bcc) /(\bcc \rtimes \bcc^* )\rightarrow Tr_k$$
 is $\Sigma_k$-equivariant under the action of the symmetric group 
both on the $a_i$'s and the $z_i$'s.  The same holds for 
$$\Psi: \, \stackrel{\circ}{\Delta}_k \times Tr_k \rightarrow F_k(\bcc) /(\bcc \rtimes \bcc^* )$$
\end{Rem}
\begin{Thm} \label{thm:bij}
The functions $\Phi_{a_1,\dots,a_k}$ and $\Psi_{a_1,\dots,a_k}$ are inverse to each other. 
\end{Thm}
\begin{proof}
We show first that $\Phi \circ \Psi$ is the identity of $Tr_k$.
Let us choose a labelled tree $tr=(T,f,(g_i)) \in Tr_k$. The Riemann surface $F_{tr}$ is equipped with a meromorphic differential $dw$ defined 
on the interior of each 2-cell by  $dw=dx + {\bf i}dy$,
that has for each $i=1,\dots,k$ a simple pole at the white vertex $v_i$ with residue $a_i$, and a simple pole
at $\infty$ with residue $-1$. The differential $dw$ vanishes exactly
at the black vertices $b \in B$. 
The biholomorphism $\zeta: F_{tr} \to \CP^1$ from definition \ref{def:biholo}
satisfies   $\zeta^*(\frac{h'(z)}{h(z)} dz)=dw$,
since $\frac{h'(z)}{h(z)}dz$ is the unique meromorphic differential having simple poles at $z_i$ with residue $a_i$, and a simple pole at $\infty$ with residue $-1$.
So $\zeta$ induces a bijection between the set of black vertices $B$ and the set $Crit \subset \bcc$ of critical points of $\log(h)$, that are the zeros of $\frac{h'(z)}{h(z)}dz$. 
By integrating we have that on a 2-cell 
$$w:=x+{\bf i}y =\log(h(\zeta(x,y)))+c $$ for a constant $c \in \bcc$.
Recall that the realization $||\bar{T}||$ is embedded into $F_{tr}$. 
A closed flow line $\Gamma \subset \CP^1$ has the form $\Gamma= \zeta(\{e\} \times [0,1])$, for an edge $e$ of the ribbon graph $\bar{T}$,
since $\arg(h)=Im(log(h))=y-\Im(c)$ is constant along $\Gamma$. 
Thus $\zeta$ induces a correspondence between edges of $T$ and flow lines not coming out of $\infty$.
This shows that $T$ is the admissible tree that gives the "shape"  of the labelled tree $(\Phi \circ \Psi)(tr)$.
Let $M=|h(\zeta(b_0))|$ be the maximum value of $|h|$ at some critical point $\zeta(b_0)$.
For any black vertex $b \in B$ we have $$|h(\zeta(b))/h(\zeta(b_0))|=exp(x(b)-x(b_0))=exp(\log(f(b)))/1=f(b)$$
Finally the difference of the arguments $\arg(\log(h))$ for adjacent flow lines $\zeta(e \times [0,1])$ and  $\zeta(e' \times [0,1])$ 
going into $z_i$  is $y(e')-y(e)=a_i g_i(e)$. 
This shows that $$\Phi(\Psi(tr))= (T,f,(g_i))=tr$$
We show next that $\Psi \circ \Phi$ is the identity of $F_k(\bcc) /(\bcc \rtimes \bcc^* )$.
Let us choose $(z_1,\dots,z_k) \in F_k(\bcc)$. There is a cell decomposition of $\CP^1$ that has:
\bit
\item the points $z_1, \dots, z_k$ and $\infty$ as 0-cells ;
\item  the closed flow lines of the field from or to critical points $E$ as 1-cells ;
\item a 2-cell $\sigma_\gamma$ for each flow line $\gamma$ from $\infty$ into a critical point. Its interior consists of all flow lines $\bar{\gamma}$ 
from $\infty$ into a zero $z_i$ with 
$$\arg(h(\gamma))<  \arg(h(\bar{\gamma}))< \arg(h(\gamma'))$$
where $\gamma'$ is the flow line into a critical point preceding $\gamma$ in the cyclic ordering at $\infty$.
\eit

Let us write $tr:=\Phi(z_1,\dots,z_k)$.
There is a cellular homeomorphism $\eta: \CP^1 \to F_{tr}$  
sending the 2-cell $\sigma_\gamma \subset \CP^1$ to $F_e \subset F_{tr}$ 
via $z \mapsto  (\log|h|, \arg(h))_{F_e}$, where $e=\gamma_+$ in the ribbon graph $\bar{T}$
and $T$ is the admissible tree underlying the labelled tree $tr$.  
Actually $\eta$ is a biholomorphism .  
This shows that $\Psi(\Phi(z_1,\dots,z_k))=\Psi(tr)=[z_1,\dots,z_k]$.

\end{proof}

We show the continuity of $\Phi$ both with respect the parameters $(a_1,\dots,a_k) \in \stackrel{\circ}{\Delta_k}$ and 
the configuration points $(z_1,\dots,z_k) \in F_k(\bcc)$.

\begin{Thm} \label{thm:cont}
The function $$\Phi: \, \stackrel{\circ}{\Delta}_k \times F_k(\bcc) /(\bcc \rtimes \bcc^* )\rightarrow Tr_k$$ is continuous.
\end{Thm}
\begin{proof}

The collection of roots of the monic polynomial $p(z)$, counted with multiplicity, form an element of the symmetric product
 $r \in SP^{k-1}(\bcc)$ that depends continuously on the parameters $a_i$'s and the $z_i$'s  since the coefficients of $p(z)$ also depend continuously on them. 
The underlying set of $r$ is the set $Crit$ of critical points of $h$. 
For a small variation of the parameters a critical point $b$ of order $m$,
i.e. a zero of the polynomial $p$ of order $m-1$, either moves to a critical point $b'$ of the same order,
or it splits into a collection of $j>1$ critical points $b_1,\dots,b_j$  of respective lower orders
$m_1,\dots,m_j$ satisfying $\sum_{l=1}^j(m_l-1)=m-1 \, .$   We deal with both cases at once allowing $j \geq 1$.
The values of $|h|$ and $\arg(h)$ at the critical point $b$ vary continuously, in the sense
that their value on $b_1,\dots,b_j$ are arbitrarily close to their respective value on $b$.
In particular the maximum $M= \max_{Crit} |h|$ varies continuously.  
Consider a small deformation of the parameters $a_i, z_i$.
The following cases can occurr:
\bit
\item  The number of critical points does not change, and $\arg(h)$ remains equal for critical points connected by a flow line. Then the admissible tree $T \in \Tau_k$ labelling $tr=\Phi(a_1,\dots,a_k,z_1,\dots,z_k) \in Tr_k$ is constant;
the flow lines 
vary by a small deformation, 
and so do the "angles" between them at the zeros $z_1,\dots,z_k$. 
By "angle" between successive flow lines $\gamma_1, \gamma_2$ in the cyclic ordering at $z_i$ we mean the difference $\arg(h(\gamma_2))-\arg(h(\gamma_1))$. 
This ensures that $f:Crit \to (0,1]$ and $g_i:E_{v_i} \to (0,1]$ vary by a small deformation and so does the labelled tree $tr$. 
\item  The number of critical points is constant, but $\arg(h)$, along the deformation, has distinct values on critical points originally connected by a flow line.  
 In this and the next case the labelled tree $tr$ changes "shape", i.e. the underlying admissible tree $T$ varies.
\begin{center}
\begin{tikzpicture}  [every node/.style={inner sep=2pt}]
\node (b1) at (357:1) [circle,fill]  {}  ;
\node (b2) at (7:2) [circle,fill]  {}   ;
\node (w) [circle,draw]{} ;
\path [<-,thick] (w) edge (b1) edge (b2) ; 
\huge
\node at (4,0) {$\dashleftarrow$};
\normalsize
\node (w3)  at (6,0)  [circle,draw]{} ; 
\node (b4) at (7,0)    [circle,fill]  {}  ;
\node (b5) at (8,0) [circle,fill]  {}  ;
\path [<-,thick] (w3) edge (b4) (b4) edge (b5) ; 
\end{tikzpicture} 
\end{center} 

There is still a small deformation if we consider {\em broken flow lines},
 the concatenation of flow lines from a critical point towards a white vertex, obtained either by following the left path or the right path. 
The identifications 
defining the topology of $Tr_k$ ensure that $tr$ varies by a small deformation. 

\item The number of critical points of $Crit$ increases by $n$ along the deformation. 
 Then the number of flow lines increases by $2n$.
 Some of the original (broken) flow lines split into multiple very close flow lines, or some new very short flow lines connect critical points
originated by splitting a single critical point along the deformation.
\begin{center}
\begin{tikzpicture}  [every node/.style={inner sep=2pt}]
\node (v1) at (0:1) [circle,draw]  {}  ;
\node (v2) at (120:1) [circle,draw]  {}   ;
\node (v3) at (240:1) [circle,draw]  {}  ;
\node (b) [circle,fill]{} ;
\path [->,thick] (b)  edge (v1) edge (v2) edge (v3); 
\huge
\node at (2.4,0) {$\dashrightarrow$};
\normalsize
\node (v4) at (5,0) [circle,draw]  {}  ;
\node (v5) at (3.5,1) [circle,draw]  {}   ;
\node (v6) at (3.5,-1) [circle,draw]  {}  ;
\node (bb) at (4,.1) [circle,fill]{} ;
\node (bb2) at (4,-.1)  [circle,fill]{} ;
\path [->,thick] (bb)  edge (v4) edge (v5) ;
\path [->,thick](bb2) edge (v4) edge (v6);

\end{tikzpicture} 
\end{center}
Also in this case $tr$ varies by a small deformation. 
\eit
 \end{proof}
\begin{Thm} \label{moduli-homeo}
The map $\Phi_{a_1,\dots,a_k}: \mathcal{M}_{0,k+1} \to Tr_k$ is a homeomorphism. 
\end{Thm}
\begin{proof}
We know that the function $\Phi_{a_1,\dots,a_k}$ is continuous and bijective, by Theorem \ref{thm:bij} and Theorem \ref{thm:cont}
for fixed parameters $a_1,\dots,a_k$.
We extend $\Phi$ it to a continuos function $\bar{\Phi}$ between compactifications. 
For $k>2$ consider the compactification $(\bcc^k-\{0\})/(\bcc \rtimes \bcc^*) \cong \CP^{k-2}$ 
 of the moduli space
$F_k(\bcc)/(\bcc \rtimes \bcc^*)$. We define the extension 
$$\bar{\Phi}:  (\bcc^k-\{0\})/(\bcc \rtimes \bcc^*) \rightarrow Tr_k^+$$
 to the one point
compactification $Tr_k^+$ sending the complement of the moduli space $K=(\bcc^k-\{0\} - F_k(\bcc))/(\bcc \rtimes \bcc^*)$ to the point at infinity. 
We want to show that $\bar{\Phi}$ is continuous at any point of $K$. 
The roots of the polynomial $$p(z) =  \sum_{i=1}^k {a_i}\prod_{j \neq i}(z-z_j)  $$ 
depend continuously on the zeros $z_i's $ of $h=\prod_{i=1}^k (z-z_i)^{a_i}$ for  $(z_1,\dots,z_k) \in (\bcc^k-\{0\})$. 
If $[z_1, \dots, z_k] \in K$ and  the value $\rho=z_{i_1}=\dots=z_{i_m}$ is repeated exactly $m$-times, then $\rho$
is a root of $p(z)$ of order $m-1$, and $h(\rho)=0$. Since not all zeros coincide, there will be at least one root  $\alpha$ of $p(z)$ with $h(\alpha) \neq 0$.
For a sufficiently small neighbourhood $U$ of $(z_1,\dots,z_k)$, any configuration $(y_1,\dots,y_k) \in U \cap F_k(\bcc)$ defines 
a multivalued function $h$ with at least one critical point $\rho'$ arbitrarily close to $\rho$,
and a critical point $\alpha'$ arbitrarily close to $\alpha$.
The ratio $|h(\rho')|/|h(\alpha')|$ is well defined, arbitrarily small, and labels a black vertex of the labelled tree 
$\Phi([y_1,\dots,y_k])$. This tree is arbitrarily close to the point at infinity in $Tr_k^+$, since a black vertex of a labelled tree in $Tr_k$ is never labelled by 0.
This shows that the induced bijective map $\Phi^+: (F_k(\bcc)/(\bcc \rtimes \bcc^*))^+ \to Tr_k^+$ between one-point compactifications is continuous, and therefore a homeomorphism, by the closed map lemma.
So the restriction $\Phi$ is also a homeomorphism.

\end{proof}
We will consider a finer compactification of the moduli space due to Axelrod-Singer, called the Fulton-MacPherson compactification, in section \ref{sec:fulton}.

\section{Cacti and open cells of the moduli space} \label{sec:cacti}

In this section we recall the definition and properties of the cacti complexes and the cacti operad.  Then we describe the cells of the moduli space in terms of cacti.
A combinatorial construction of the cacti was introduced by McClure and Smith \cite{MS}, and a geometric construction 
was introduced by Voronov \cite{Voronov} and Kaufmann \cite{Kaufmann}.
We compared the two approaches in \cite{Deligne}. Here we introduce the cacti via 
$n$-fold semi-simplicial sets.
We recall the definition of $n$-fold semi-(co)simplicial objects.
\begin{Def} 
Consider the category $\Delta_+$ with objects the totally ordered sets $[k]=\{0,\dots,k\}$, for $k \in \N$, and morphisms the injective monotone maps. 
A $n$-fold semi-simplicial (resp. semi-cosimplicial) object in a category $C$ is a contravariant (resp. covariant)  functor $X: (\Delta_+)^n \to C$.
We set $X_{a_1,\dots,a_n} = X([a_1],\dots,[a_n])$, and define the
face operator $$\partial_i^{(j)}=X(id_{[a_1]}, \dots, in_i, \dots,id_{[a_n]}): X_{a_1,\dots,a_j,\dots,a_n} \to X_{a_1,\dots,a_j-1,\dots,a_n}$$ for $j=1,\dots,n$ and $i=0,\dots,a_j$ 
where $in_i:[a_j-1] \to [a_j]$ is the monotone map missing $\{i\}$. 

There is a natural $n$-fold semi-cosimplicial space  $(\Delta^*)^n$ sending $([a_1],\dots,[a_n])$ to the product of simplices  $\prod_{j=1}^n \Delta^{a_j}$. 
Given a $n$-fold semi-simplicial set $X$, 
its geometric realization  $|X|$ is the topological space  
$$|X| = \coprod_{a_1,\dots,a_n} (X_{a_1,\dots,a_n} \times \prod_{j=1}^n \Delta^{a_j}) /\sim \; , $$  
the quotient by the relations $(f^*(x), y) \sim (x,f_*(y))$, where
$$f:[a_1,\dots,a_n] \to [b_1, \dots, b_n]$$ is a morphism of $(\Delta_+)^n, \; x \in X_{b_1,\dots,b_n}$ and $y \in \prod_{j=1}^n \Delta^{a_j}\, .$
\end{Def}
In particular $|X|$ is a regular CW-complex with a cell $|x|$ of dimension $\sum_{j=1}^n{a_j}$ 
for each $x \in X_{a_1,\dots,a_n}$.

We introduce the cacti complex $\C_k$ as the geometric realization of a $k$-fold semi-simplicial set. This combinatorial definition follows McClure and Smith \cite{MS}, although they did not mention $k$-fold semi-simplicial sets explicitly.

\begin{Def} \label{defcacti}
The cacti complex $\C_k$ is the geometric realization of a $k$-fold semi-simplicial set $X^k_{\bullet,\dots,\bullet}$
that has an element in multidegree $(m_1,\dots,m_k)$ for each surjective map 
$f:\{1,\dots,m+k\} \to \{1,\dots,k \}$ with $m=\sum_{j=1}^k m_j$ 
such that 
\ben
\item $f^{-1}(j)$ has cardinality $m_j+1$ for each $j=1,\dots,k$.
\item $f(a) \neq f(a+1)$ for $a=1,\dots,m+k-1$ 
\item there are no values $1 \leq a<b<c<d \leq m+k$ for which 
$f(a)=f(c) \neq f(b)=f(d)$  (complexity condition)   \label{compl}
\een
The corresponding cell (product of simplices) in the realization is labelled
either by $f$ or by the finite sequence 
$f(1) \dots f(m+k)$.

Let us write $f^{-1}(j)=\{x_0,\dots,x_{m_j}\}$ with 
$x_0 < \dots < x_{m_j}$. 
For $m_j >0$ and $0 \leq i \leq m_j$ 
the $i$-th face in the $j$-th coordinate 
$$\partial_i^{(j)}(f):\{1,\dots,m+k-1\} \to \{1,\dots,k\}$$ of $f$ is 
$f(1) \dots \widehat{f(x_i)}   \dots f(m+k)$. The notation means that the $x_i$-th value $f(x_i)=j$ is removed from the sequence.
\end{Def}

Some remarks are in order.
\bit
\item The face operator is well defined: by condition \ref{compl}, since $m_j>0$, there is at least another element $x_u \neq x_i$ such that 
$f(x_u)=f(x_i)=j$. Then $\partial_i^{(j)}$ is suriective. Moreover
 $f(x_i-1) \neq f(x_i + 1)$ by condition \ref{compl} and so also $\partial_i^{j}$ satisfies condition \ref{compl}.
\item The geometric realization is the quotient 
$$\C_k= \coprod_{m_1,\dots,m_k} ( \prod_{j=1}^k \Delta_{m_j} \times  X^k_{m_1,\dots,m_k}  ) / 	\sim $$
where $(s_1,\dots,d_i(s_j), \dots,  s_k , f)  \sim   (s_1,\dots,s_k ,   \partial_i^{(j)}f  ) $,
and $d_i: \Delta_{m_j-1} \to \Delta_{m}$ is the standard inclusion of the $i$-th face. 
\item The symmetric group $\Sigma_k$ acts freely on $X^k_{\bullet,\dots,\bullet}$ by post-composition with the maps into $\{1,\dots,k\}$.
\eit

\begin{Prop}
The highest dimension of a cell of $\C_k$ is $k-1$.
\end{Prop}

\begin{proof}
By induction on $k$, removing all the occurrences of $k$ from the sequence of a cell. 
\end{proof}

\begin{Exa}
The complex $\C_2$ has 
\bit
\item  Two $0$-cells $12$ and $21$;
\item  Two $1$-cells $121$ and $212$
\eit
And is homeomorphic to a circle.

\medskip

If we indicate the action of $\sigma \in \Sigma_3$, then the complex $\C_3$ has 
\bit
\item six $0$-cells of the form $\sigma(123)$ 
\item eighteen  $1$-cells $\sigma (1231)$, $\sigma (1232)$, $\sigma (1213)$
\item  twelve  $2$-cells $\sigma ( 12131)$  $\sigma(12321)$.
\eit
\end{Exa}

\begin{Thm} (McClure-Smith) \cite{MS-cos}
The complex $\C_k$ is $\Sigma_k$-equivariantly homotopy equivalent to the configuration space $F_k(\bcc)$. 
\end{Thm}

The complex $\C_k$ is denoted by $\mathcal{F}(k)$ in \cite{Deligne}.

We associate to each element $x \in \C_k$ a piecewise linear map
$c_x:[0,k] \to [0,1]^k$, the {\em cactus} map.

\medskip

Consider $x=(t^{(1)},\dots,t^{(k)}; f) \in \C_k$
with $$t^{(j)}=(t^{(j)}_0,\dots,t^{(j)}_{m_j}) \in \Delta_{m_j}\, ,$$ 
$f:\{1,\dots,m+k\} \to \{1,\dots,k\}$ and $\sum_{j=1}^k m_j =m$. 

Let $g: \{1,\dots,m+k\} \to \{0,\dots,\max_j{m_j}\}$ be the function defined 
by $$g(r)= | \{ s \, | \,  s<r ,   f(r)=f(s) \}|$$  Of course $g(r) \leq m_{f(r)}=|f^{-1}(f(r))|-1$.
This induces an ordering of the (multi)simplicial coordinates
$(t^{(f(1))}_{g(1)}, \dots , t^{(f(m+k))}_{g(m+k)} )$ of $x$. 
Let us write $t_i :=t^{(f(i))}_{g(i)}$.
If we set $y_{r}  = \sum_{s=1}^{r} t_s$  (with the convention $y_0=0$),
then we obtain a decomposition of the interval $[0,k]=\cup_{r=1}^{m+k}[y_{r-1},y_r]$ into $m+k$ closed intervals.
\begin{Def}
The cactus map 
$c_x=(c_x^{(1)},\dots, c_x^{(k)} ):[0,k] \to [0,1]^k$   
is characterized by 
\bit
\item $ c_x(0)=(0,\dots,0) $
\item if $y \in [y_{r-1}, y_r]$ and   $j \neq f(r)$ then  $ c_x^{(j)}(y)=c_x^{(j)}(y_{r-1})$    
\item if $y \in [y_{r-1}, y_r]$ then $c_x^{(f(r))}(y)= c_x^{(f(r))}(y_0)+(y-y_{r-1})$ 
\eit
\end{Def}

Thus the curve $c_x$ describes the motion of a point that moves along a single coordinate at a time, with speed 1.  
In particular $c_x(k)=(1,\dots,1)$.

\begin{Def}
The cactus $C_x$ of $x \in \C_k$ is the image of the composition  
$$[0,1]^k \stackrel{c_x}{\longrightarrow} [0,1]^k \twoheadrightarrow [0,1]^k/(\{0,1\})^k \cong (S^1)^k$$
\end{Def}
The cactus $C_x$ is a connected union of $k$ oriented and ordered circles, called the {\em lobes}, that intersect at a finite number of points.
The $j$-th lobe is the image of the union of intervals $\bigcup_{f(r)=j} [y_{r-1},y_r]$.
Two lobes intersect at most at one point, by the complexity condition 3) in Defintion \ref{defcacti}. The base point $P \in C_x$ is 
$P=[c_x(0)]=[c_x(k)]$.
The multi-simplicial coordinates correspond to the lengths of the arcs between intersection points, or between the base point and an intersection point.
The cactus $C_x$ admits a planar embedding compatible with the orientations. If the base point is also an intersection point, then 
the convention is to draw it slightly shifted into the positive direction of the first lobe crossed by the curve, namely that indexed by $f(1)$.
In this way the drawing describes univocally $x$.

We draw the cacti representing all cells in $\mathcal{C}_2$ and $\mathcal{C}_3$ up to permutations. 
The boundary of a cell contains the cells obtained by pushing the base point into an intersection point, or by collapsing a chord between adjacent intersection points into a single point. 
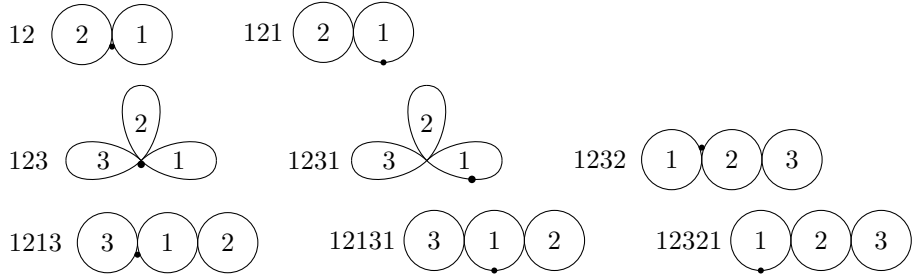
\medskip

\begin{tikzpicture}  [scale=.8] 
\draw  circle (.5cm)  ;
\draw node {2};
\draw   (1,0) circle (.5 cm) ;
\draw (1,0) node{1};
\fill (0.5,-0.2) circle (.05 cm);
\draw (-1,0) node{12};
\end{tikzpicture} \quad \quad
\begin{tikzpicture}  [scale=.8]  
\draw  circle (.5cm)  ;
\draw node {2};
\draw   (1,0) circle (.5 cm) ;
\draw (1,0) node{1};
\fill (1,-0.5) circle (.05 cm);
\draw (-1,0) node{121};
\end{tikzpicture}

\medskip

\begin{tikzpicture} [scale=1]
\draw  (0,0) to [out=45,in=0] (0,1);
\draw  (0,0) to [out=135,in=180] (0,1);
\draw (0,0) to [out=45,in=90] (1, 0);
\draw  (0,0) to [out=-45,in=-90] (1,0);
\draw  (0,0) to [out=135,in=90] (-1,0);
\draw  (0,0) to [out=225,in=270] (-1,0);
\node at (0,0.5) {2};
\node at (0.5,0) {1};
\node at (-0.5,0) {3};
\fill  (0,-0.05) circle  (.05cm);
\draw (-1.5,0) node{123};
\end{tikzpicture}
\quad \quad
\begin{tikzpicture} [scale=1]
\draw  (0,0) to [out=45,in=0] (0,1);
\draw  (0,0) to [out=135,in=180] (0,1);
\draw (0,0) to [out=45,in=90] (1, 0);
\draw  (0,0) to [out=-45,in=-90] (1,0);
\draw  (0,0) to [out=135,in=90] (-1,0);
\draw  (0,0) to [out=225,in=270] (-1,0);
\node at (0,0.5) {2};
\node at (0.5,0) {1};
\node at (-0.5,0) {3};
\fill  (0.6,-0.25) circle  (.05cm);
\draw (-1.5,0) node{1231};
\end{tikzpicture}
\quad \quad
\begin{tikzpicture}  [scale=.8] 
\draw  circle (.5cm)  ;
\draw node {1};
\draw   (1,0) circle (.5 cm) ;
\draw (1,0) node{2};
\draw   (2,0) circle (.5 cm) ;
\draw (2,0) node{3};
\fill (0.5,0.2) circle (.05 cm);
\draw (-1.2,0) node{1232};
\end{tikzpicture}

\medskip

\begin{tikzpicture}[scale=.8]   
\draw  circle (.5cm)  ;
\draw node {3};
\draw   (1,0) circle (.5 cm) ;
\draw (1,0) node{1};
\draw   (2,0) circle (.5 cm) ;
\draw (2,0) node{2};
\fill (0.5,-0.2) circle (.05 cm);
\draw (-1.2,0) node{1213};
\end{tikzpicture}
\quad \quad
\begin{tikzpicture}  [scale=.8] 
\draw  circle (.5cm)  ;
\draw node {3};
\draw   (1,0) circle (.5 cm) ;
\draw (1,0) node{1};
\draw   (2,0) circle (.5 cm) ;
\draw (2,0) node{2};
\fill (1,-0.5) circle (.05 cm);
\draw (-1.2,0) node{12131};
\end{tikzpicture}
\quad \quad
\begin{tikzpicture} [scale=.8]  
\draw  circle (.5cm)  ;
\draw node {1};
\draw   (1,0) circle (.5 cm) ;
\draw (1,0) node{2};
\draw   (2,0) circle (.5 cm) ;
\draw (2,0) node{3};
\fill (0,-0.5) circle (.05 cm);
\draw (-1.2,0) node{12321};
\end{tikzpicture}

\subsection{Operations on the cacti spaces}

The cacti spaces have the structure of an operad only up to homotopy, in the sense that 
there are composition maps
$$\theta_{n_1,\dots,n_k}: \C_k \times \C_{n_1} \times \dots \times \C_{n_k} \rightarrow \C_{n_1 + \dots + n_k}$$ 
satisfying the associativity axioms of operad structure maps  \cite{Benoit} only up to (higher) homotopy. See Remark 2.3.19 in \cite{Kaufmann} for an example where strict associativity fails.
The only point $id \in \C_1$ is a unit up to homotopy.  The composition maps satisfy strictly the equivariance axioms.

In addition there are multiplication maps  $$\mu_{k,h}:\C_k \times \C_h \to \C_{k+h}$$ that satisfy strict associativity. The composition maps and the multiplication maps define the structure of an operad with multiplication up to (higher) homotopy.
We define first the multiplication maps $\mu_{k,h}$.
\begin{Prop} (McClure-Smith) \cite{MS-cos}
For each $k,h \in \N_+$ there is a binary operator, the star product
$$\star \, : X^k_{m_1 ,\dots, m_k} \times X^h_{n_1,\dots, n_h} \to X^{h+k}_{m_1,\dots,m_k,n_1,\dots,n_h}$$
defined by 
$$(f \star g)(i) = \begin{cases}
f(i) \;{\rm for }\; i \leq m+k \\
g(i-m-k)+k \;{\text for }\; i > m+k  
\end{cases}
$$ with $m= \sum_{j=1}^{k} m_j$.
This is a levelwise injective map of $(k+h)$-fold semi-simplicial sets, and
so it induces an embedding of the realizations $\mu_{k,h}:\C_k \times \C_h \to \C_{k+h}$.
The star product satisfies strict associativity, and so do the maps $\mu_{k,h}$.
\end{Prop}
Then we consider the homotopy operad composition maps, also due to McClure and Smith. 

\begin{Def} \label{bigdef}
For $(n_1,\dots,n_k) \in (\N_+)^k$, and $n=n_1+ \dots + n_k$,  there are 
embeddings $$\theta_{n_1,\dots,n_k}: \C_k \times \C_{n_1} \times \dots \times \C_{n_k} \rightarrow \C_{n} $$
defined as follows. 
For $x \in \C_k$ 
the cactus map $c_x:[0,k] \to [0,1]^k$ is a piecewise oriented isometry onto its image, and so is
$c_{x_j}: [0,n_j] \to [0,1]^{n_j}$,  for $x_j \in \C_{n_j}$ and $j=1,\dots,k$.
Consider the product of dilations $D:[0,1]^k \cong \prod_{j=1}^k [0,n_j]$ such that $D(t_1,\dots,t_k)=(n_1t_1,\dots,n_k t_k)$.
There is a unique piecewise oriented isometry onto its image $c': [0,n] \to \prod_{j=1}^k[0,n_j]$ and a unique piecewise linear homeomorphism
$\alpha:[0,k] \cong [0,n]$ such that $c' \circ \alpha = D \circ c_x$.
Then $\theta(x,x_1,\dots,x_k) \in \C_n$ is uniquely defined by 
 $$c_{\theta(x,x_1,\dots,x_k)}=(\prod_{j=1}^k c_{x_j})   \circ c' : [0,n] \to [0,1]^n$$
 \end{Def}
 
 There are partial composition operations $\C_m \times \C_n \to \C_{m+n-1}$ sending $(x,y) \mapsto x {\circ_i}y$  for each $m,n \geq 1$ and  $i=1,\dots,m$, defined as usual by 
$$x \circ_i y = \theta(id, \dots, y, \dots, id),$$ with $y$ appearing at the $i$-th entry, and $id \in \C_1\, $.
 The partial composition operations do not satisfy associativity, as shown by Kaufmann  \cite{Kaufmann}, 
 but they do satisfy commutativity in the sense that
 $$(x \circ_i y) \circ_{j+n-1} z = (x \circ_j  z) \circ_i y$$ for $i<j \, , \, x \in \C_m, \, y \in \C_n, \, z \in \C_r \;$. 
 In addition they are strictly compatible with the action of the symmetric groups .
One can recover the structure maps $\theta$ as iterated partial composition operations similarly as in \cite{Benoit},vol. I, 2.1.
 \
  
We describe explicitly the structure map $\theta$ in terms of cells and multisimplicial coordinates.
Given a cell $\sigma_f \subset \C_n$ labelled by $f: \{1, \dots, m+n\} \to \{1,\dots,n\}$, consider the projection $\pi: \{1,\dots,n\} \to \{1,\dots,k\}$ 
such that $\pi(i)=j$ if $$\sum_{l=1}^{j-1}n_l   < i  \leq \sum_{l=1}^j n_l \, .$$
We have that $\pi \circ f: \{1,\dots,m+n\} \to \{1,\dots,k\}$ satisfies the complexity condition of Definition \ref{defcacti}
if and only if the cell $\sigma_f$ is contained in the image of $\theta$.

Consider the equivalence relation on $\{1,\dots,m+n\}$ generated by 
$i \sim i+1$ if $\pi(f(i))=\pi(f(i+1))$. The quotient by $\sim$ inherits a unique structure of totally ordered set such that 
$$\pi': \{1 \dots, m+n\} \to \{1 \dots, m+n\}/\sim$$
 is a monotone map
and so we identify  $$\{1 \dots, m+n\}/\sim \, \cong \,  \{1,\dots, u+k\}$$ for some $u \in \N$.
Let $f': \{1,\dots, u+k \} \to \{1,\dots,k\}$ be the map induced by $f$ on the quotient. Now $f'$ satisfies the condition 
(2) of Definition \ref{defcacti}.
It also satisfies the complexity condition (3), and so it labels
some cell $\sigma_{f'} \subset \C_{k}$.
With a similar procedure, for $j \in \{1,\dots,k\}$,  let $$A_{j}=(\pi \circ f)^{-1}(j) \subset \{1,\dots,m+k\}.$$
Let us impose a similar relation $\sim_j$ on the totally ordered set $A_j$ generated by $i \sim_j i'$ if $i'$ is the successor of $i$ in $A_j$ and 
$f(i)=f(i')$. Then we have monotone maps of totally ordered set 
$$A_j \stackrel{\pi_j}{\longrightarrow} A_j/\sim_j \, \cong \, \{1, \dots,  u_j+n_j\}.$$ The map  $f_{| A_j}-\sum_{l=1}^{j-1}n_l:A_j \to \{1,\dots,n_j\}$ induces a map on 
the quotient  $$f_j:\{1,\dots,u_j+n_j\} \to \{1,\dots,n_j\}$$ labelling a cell $\sigma_{f_j} \subset \C_{n_j}$.
One verifies that $u+\sum_{j=1}^k u_j = m$ and so the cells 
$\sigma_{f'} \times \prod_{j=1}^k \sigma_{f_j}$ and $\sigma_f$ have the same dimension. 
\begin{Thm}
The structure map satisfies $$\theta_{n_1,\dots,n_k}(\sigma_{f'} \times \prod_{j=1}^k \sigma_{f_j}) \supseteq \sigma_f$$ and on such cells
in multisimplicial coordinates is defined as follows:
we identify a point $P$  of a cell $\sigma_f$ labelled by $f: \{1,\dots,m+n\} \to \{1,\dots,n\}$ via 
its multisimplicial coordinates
$t_i$ with $i \in \{1,\dots,m+n\}$, and write $P=(t_i)$. We have that $t_i \geq 0$ and $\sum_{ f(i)=j}t_i = 1$ for $j \in \{1, \dots,n\}$.
Then  $(t_i) \in \sigma_f$ is equal to   
$$(t_i) =\theta_{n_1,\dots,n_k}   ( (x_{i'}), (y^{(1)}_{i_1}),\dots,(y^{(k)}_{i_k})),$$
with
$$x_{i'}  =  \frac{1}{n_{f'(i')}} \sum_{ \pi'(i)=i'} t_{i} $$ 
 and   
$$y^{(j)}_{i_j}     =   \sum_{ \pi_j(i)=i_j  } t_i  $$
\end{Thm}     
\begin{Cor}
The homotopy operad structure maps $\theta$ are cellular injective maps, that define on the family of complexes of cellular chains 
$C_*(\C)=\{ C_*(\C_n)\} _{n \geq 1}$ a structure of differential graded operad (an operad in the category of chain complexes of abelian groups). 
\end{Cor}

The operad $C_*(\C)$ is isomorphic to $\mathcal{S}_2$,  a natural $E_2$ suboperad of the $E_\infty$ 
surjection operad $\mathcal{S}$ by McClure and Smith
\cite{MS-sur}.
From the point of view of the pictures of the cacti, it is easy to detect whether the cell $\sigma_f$ is contained in the image of $\theta$. 
This happens if and only if for any $x \in \, \stackrel{\circ}{\sigma_f}$ 
the union of the lobes of the cactus $C_x$ labelled by $l$, with $\sum_{j=1}^{i-1}n_j  < l \leq \sum_{j=1}^i n_j$, is connected, 
for each $i=1,\dots,k$.

The structure maps of the operad $C_*(\C)$ are dual to 
the homomorphisms 
$$C^*(\theta): C^*(\C_{n_1+\dots+n_k}) \rightarrow C^*(\C_{k}) \otimes C^*(\C_{n_1}) \otimes \dots \otimes C^*(\C_{n_k}) $$
sending $$\sigma_f \mapsto \pm \sigma_{f'} \otimes \sigma_{f_1} \otimes \dots \sigma_{f_k}$$ with the notation of definition \ref{bigdef},
when the cell $\sigma_f$ is in the image of 
$\theta$, and $\sigma_f \mapsto 0$ otherwise. 
The signs are determined by the orientation of the cells.

\subsection{Generating series of the number of cells} \label{sub:count}

We show how to compute the generating function counting the number of cells, that corresponds to the so called Davenport-Schinzel 
sequences and was first computed in \cite{Gardy}.

\begin{Def}
Let $c_m(\C_k)$ be the number of the $m$-cells of $\C_k$. By the free action of the symmetric group on the set of cells
$k!$ divides  $c_m(\C_k)$. 
We define the formal series with integer coefficients 
$$P(x,t)= \sum_{m,k \in \N} \frac{c_m(\C_k)}{k!} t^m x^k$$ 
\end{Def}

\begin{Thm} \label{count-ds}
$$P(x,t)= \frac{1-x-\sqrt{(1-x)^2-4xt}}{2t} $$ where the left hand side is intended as asymptotic expansion of the right hand side for $x,t \rightarrow 0$. 
\end{Thm}
\begin{proof}
We split $P=B+W$ where
\bit
\item  $B$ is the generating function for the cells labelled by  $f$ with $f(1) \neq f(m+k)$.
\item  $W$ is the generating function for the cells labelled by $f$ with $f(1)=f(m+k)$.
\eit
Geometrically the cells of $B$ correspond to cacti with the base point on the intersection of lobes, that we call black cacti,
and those of $W$ to cacti with the base point not coinciding with the intersection of lobes, that we call white cacti. 
We have the identity 
$$B=\sum_{j \geq 2} W^j  = \frac{W^2}{1-W}$$
by the following reason: given a cell $f$ in $B$ we can decompose it uniquely as $\sigma(f_1 \star \dots \star  f_j)$ for some $j \geq 2$, and 
$\sigma$ is a permutation. Geometrically this means that a black cactus is obtained uniquely by joining 2 or more white cacti at the base point,
and has degree equal to the sum of the degrees of the white cacti. Then

\begin{equation}
P=B+W=\frac{W}{1-W}\label{PW} 
\end{equation}
On the other hand we have the equation
\begin{equation}
W=x (1 + tP +t^2 P^2 + \dots) = \frac{x}{1-tP} \label{last}
\end{equation}
because any white cactus is obtained uniquely by  gluing  $r$ cacti to a single lobe away from the base point, 
and has degree equal to the sum of the degrees of the $r$ cacti raised by $r$.

Substituting the value of \ref{PW} in \ref{last} and solving a quadratic equation we obtain
\begin{equation}
W=\frac{1+x \pm \sqrt{(1-x)^2-4xt}}{2(t+1)} \label{quadr}
\end{equation}
But for $t=0$ we expect $W(x,0)=x$ since the only white cell in degree 0 is labelled by $1$, so the sign in \ref{quadr} is a minus.
The substitution of expression \ref{quadr} in \ref{PW}  concludes the proof.
 
\end{proof}

\section{Cellular structure of the open moduli space} \label{sec:cellular}

The aim of this section is to describe an open cellular decomposition of the moduli space $FM'_k := F_k(\bcc)/(\bcc \rtimes \R^+)$.

As a preliminary result we compare spaces of cacti and spaces of labelled trees.
For $k >1$ there is a free action of $S^1$ on $\C_k$ that moves the base point {\em clockwise}.  
The action commutes with the action 
of the symmetric group $\Sigma_k$.
The quotient $\C'_k=\C_k/S^1$ can be considered as the space of "unbased cacti".

\begin{Prop} \label{unbcact}
The projection $\C_k \to \C'_k$ is a trivial principal $S^1$-bundle and so $\C_k \cong \C'_k \times S^1$.
\end{Prop}
To see this consider the continuous section   $\C'_k \to \C_k$ that assigns as base point of an unbased cactus the first point of the lobe $2$ encountered after lobe $1$ by moving 
along the cactus according to the orientation of the lobes.

Let $Tr_k^{1} \subset Tr_k$ be the subspace of the labelled trees with $k$ leaves with all black vertices labelled by the maximum possible value 1.
\begin{Prop}\label{ctr} (Egger \cite{Egger})
There is a $\Sigma_k$-equivariant homeomorphism $\C'_k \cong Tr_k^{1}$.
\end{Prop}
\begin{proof}
A labelled tree in $Tr_k^{1}$ has no edges between black vertices. Its only parameters are of the form $g_i(e)$, where $e$ is an edge going into a white vertex $v_i$.
If we collapse all edges, then we obtain a planar unbased cactus with the lenghts of the arcs between intersection points specified by the same parameters.
The lobes of the cactus correspond to the white vertices. 
The following picture describes a  cactus and the corresponding tree.
\begin{center}
\begin{tikzpicture}  [scale=.8, every node/.style={inner sep=2pt}]
\draw  circle (.5cm)  ;
\draw node {1};
\draw   (1,0) circle (.5 cm) ;
\draw (1,0) node{2};
\draw   (2,0) circle (.5 cm) ;
\draw (2,0) node{3};

\huge
\node at (4,0) {$\leftrightarrow$};
\normalsize

\node (v1) at (6,0) [circle,draw,label=above:$v_1$]  {}  ;
\node (b1) at (7,0) [circle,fill] {} ;
\node (v2) at (8,0) [circle,draw,label=above:$v_2$]  {}  ;
\node (b2) at (9,0) [circle,fill] {} ;
\node (v3) at (10,0) [circle,draw,label=above:$v_3$]  {}  ;
\path [->,thick] (b1) edge (v1) edge (v2) 
(b2) edge (v2) edge (v3) ;

\end{tikzpicture}
\end{center}
The correspondence between $\C'_k$ and $Tr_k^{1}$ is bicontinuous and $\Sigma_k$-equivariant. 
\end{proof}

Consider the $S^1$-action on $F_k(\bcc)$ defined by $\lambda \cdot (z_1,\dots,z_k)=(\lambda z_1,\dots,\lambda z_k)$. This commutes 
with the action of the symmetric group $\Sigma_k$. So both $F_k(\bcc)$ and $\C_k$ are $(S^1 \times \Sigma_k)$-spaces.
\begin{Prop} \label{4.3}
The space of cacti $\C_k$ is a $(S^1 \times \Sigma_k)$-equivariant deformation retract of the configuration space $F_k(\bcc)$.
\end{Prop}

\begin{proof}

The space $Tr_k^{1}$ is a deformation retract of the space of all labelled trees $Tr_k$. The retraction $r:Tr_k \to Tr_k^{1}$ sets all labels of the black vertices to 1, and the deformation homotopy
rescales their value. 
Let $Tr'_k \cong Tr_k \times S^1$ be the pullback of the trivial $S^1$-bundle $\C_k \to \C'_k$ along the homeomorphism of Proposition \ref{ctr}. We think of this space as a labelled tree together with a base point.
The configuration space $F_k(\bcc)$ deformation retracts onto a subspace homeomorphic to
$FM'_k=F_k(\bcc)/(\bcc \rtimes \R^+) \cong S^1 \times \mathcal{M}_{0,k+1}$. 
In section \ref{sec:conf} 
we built a homeomorphism $$\Phi_{a_1,\dots,a_k}:  \mathcal{M}_{0,k+1}  \to Tr_k$$  using the flow lines of a field associated to the configuration points.
For a fixed configuration $(z_1,\dots,z_k)$, the flow lines of its field coming out of $\infty$ correspond to elements of $S^1$.
Namely a flow line $\gamma:(a,+\infty) \to \bcc$ with $\lim_{t \to +\infty}\gamma(t)=\infty$ corresponds to 
$z=\lim_{t \to +\infty} \gamma'(t)/|\gamma'(t)| \in S^1$. 
If we do not mode out by rotations, then 
we obtain a homeomorphism $$\Phi'_{a_1,\dots,a_k}: FM'_k= F_k(\bcc) / (\bcc \rtimes \R^+) \rightarrow Tr'_k$$ 
similar to the homeomorphism
$\Phi_{a_1,\dots,a_k}$
with the additional data that the distinguished flow line with horizontal tangent vector at infinity (i.e. corresponding to $1 \in S^1$)
gives the base point in $Tr'_k$, by measuring its tangent direction when it hits its configuration point target. 
Now $Tr'_k$ deformation retracts onto $\C_k$  by Propositions \ref{ctr} and \ref{unbcact} . All deformation retractions
and homeomorphisms in the proof are  $(S^1 \times \Sigma_k)$-equivariant.

\end{proof}

We introduce some combinatorial definitions that will be useful.
\begin{Def}
A nested tree on $k$ leaves is a collection $\mathcal{S}$ of subsets of $\{1,\dots,k\}$ of cardinality at least 2,  called {\em vertices},
such that 
\bit
\item $\{1,\dots,k\} \in \mathcal{S}$. This vertex is called the {\em root}.
\item if $S_1 \in \mathcal{S}$ and $S_2 \in \mathcal{S}$ then either
$S_1 \cap S_2= \emptyset$, or 
$S_1 \subseteq S_2$, or $ S_2 \subseteq S_1$.
\eit
We denote by $N_k$ the set of all nested trees on $k$ leaves, and $\mathcal{S'} =\mathcal{S}-\{1,\dots,k\}$.
\end{Def}

Given a nested tree $\mathcal{S}$ on $k$ leaves, consider the partial order on its vertices defined by 
$S_1 < S_2$ when $S_1 \subset S_2$. The root $\{1,\dots,k\} \in \mathcal{S}$ is a maximal vertex. 
The set $E(\ms) \subset \ms \times \ms$ of pairs $S_1,S_2$ such that $S_1 < S_2$, and there 
is no vertex $S_3 \in \mathcal{S}$ with $S_1 < S_3 < S_2$, is called the set of {\em internal edges} of $\ms$. 
We say that the internal edge $(S_1,S_2)$ goes out of $S_1$ and into $S_2$. For each $i=1,\dots,k$ let 
$\ms^-_{i}$ be the minimal element of $\ms$ containing $i$. We say that $i$ is an open edge going into $\ms^-{i}$.
The {\em valence} $|S|\in \N$ 
 of a vertex $S$ is the number of edges (either internal or open) going into it.
The terminology is suggested by the following construction.

\begin{Def}
The realization  $||\ms||$ is a 1-dimensional complex defined as
$$||\ms||:=  \mathcal{S}\,  \coprod \, (E(\ms) \times [0,1]) \, \coprod \, (\{1,\dots,k\} \times [0,1)\,)  \;  /\sim $$ 
where 
\bit
\item  For $(S_1,S_2) \in E(\ms)$, $((S_1,S_2),0) \sim S_2$ and $((S_1,S_2),1) \sim S_1$ 
\item For $i=1,\dots,k, \; (i,0) \sim \ms^-_{i}$
\eit
\end{Def}
The complex $||\ms||$ has a 0-cell for each vertex, a closed 1-cell for each internal edge, and a half-open 1-cell 
for each open edge $i \in \{1,\dots,k\}$. 

In the following example we enumerate the nested trees on leaves, and we draw the corresponding realizations.

\begin{Exa}
Let us indicate a nested tree by inserting a set of parenthesis for each vertex that is not the root. This notation is unique up to permutations of the indices and of the parenthesis, preserving the inclusion relations between parenthesis, and the set of indices inside each parenthesis.
There are 4 nested trees on 3 leaves:

\medskip

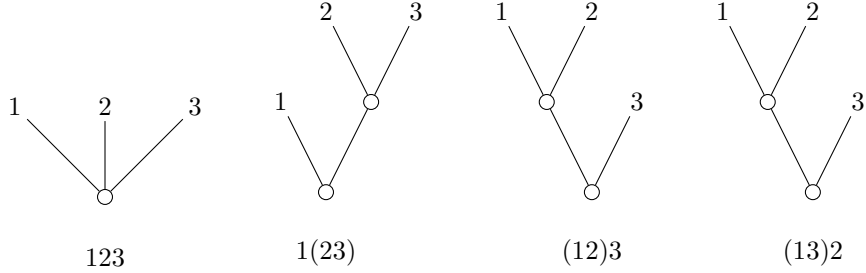
\begin{tikzpicture}  [every node/.style={inner sep=2pt},scale=.8]
\node  [circle,draw]{}[grow'=up]
child {node{1}}
child {node{2}}
child {node{3}}  ;
\draw (0,-1) node{123};
\end{tikzpicture}
\quad \quad
\begin{tikzpicture}  [every node/.style={inner sep=2pt},scale=.8]
\node[circle,draw]{}  [grow'=up]
child {node{1}}
child {node[circle,draw]{}  
child{node{2}} 
child {node{3}}     };
\draw (0,-1) node{1(23)};
\end{tikzpicture}
\quad \quad
\begin{tikzpicture}[every node/.style={inner sep=2pt},scale=.8]
\node[circle,draw] {} [grow'=up]
child { node [circle,draw]  {}   child {node{1}} child{node{2}}  }
child {node{3}}    ;
\draw (0,-1) node{(12)3};
\end{tikzpicture}
\quad \quad
\begin{tikzpicture}[every node/.style={inner sep=2pt},scale=.8]
\node [circle,draw]{} [grow'=up]
child { node [circle,draw]{}    child {node{1}} child{node{3}}  }
child {node{2}}     ;
\draw (0,-1) node{(13)2};
\end{tikzpicture}

The first tree has a vertex of valence 3, and the others have each 2 vertices of valence 2. 
\end{Exa}

The combinatorics of nested trees is well suited to describe iterated compositions in an operad. We sketch how, and refer to the literature 
for more details \cite{Benoit}.

Let $O$ be a topological operad, and $\ms$ a nested tree on $k$ leaves, with $k>2$.  
Then there is an induced map $$\theta_\ms:\prod_{S \in \ms} O(|S|) \to O(k)$$ constructed out of the structure maps of the operad $O$.
The map $\theta_\ms$ is the composition of as many $\circ_i$ operations (multiplied by an identity)  as the number of the internal edges of the nested tree $\ms$.
In a coordinate free approach one should consider $O$ as a functor from the category of finite sets and bijections to the category of topological spaces. Then  $O(in(S)) \cong O(|S|)$, where $in(S)$ is the set of edges going into $S$, and the identification depends on 
the choice of a bijection $in(S) \cong \{1,\dots,|S|\}$.

\

We have seen that the cacti complex do not form an operad. However we can still associate to a nested tree on $k$ leaves, with $k>2$, embeddings
$$\theta_\ms: \prod_{S \in \ms} \C_{|S|} \to \C_k$$  One has to specify the composition order since the partial composition products do not
satisfy associativity. We choose to compose by collapsing first the internal edges closer to the leaves, and then proceeding towards the root. 

We use the maps $\theta_\ms$ to define a space of labelled {\em nested} trees $N_k(\C)$, that we will identify in Lemma \ref{lemmone} to $Tr_k \times S^1$.  
The definition of this space is similar to that of the $W$-construction
by Boardman and Vogt \cite{BV}: its elements are nested trees with vertices labelled by cacti, and internal edges labelled by 
a parameter  varying between 0 and 1. As opposed to the Boardman-Vogt convention, we remove an edge and compose the cacti labelling its ends when the length of the edge is 1.  Another important point is that $\C$ is not a strictly associative operad, and so care is required when composing. Finally another difference is that the length 0 is excluded (it will appear later). 
A useful way to draw a labelled nested tree is to associate each edge over a vertex to a lobe of the cactus indexing that vertex,
and drawing each internal edge so that its length is $1$ minus its label. 

\begin{center}
\begin{tikzpicture} [scale=.7]
\node (1) at (-1.5,4)  {1};
\node (2) at (1.5,4)  {2};
\node (3) at (3.5,4) {3};

\node (basis) at (2,0) [inner sep=13pt] {};
\node (12) at  (0,2.5) [inner sep=16pt] {};

\node (ext) at (-.3,2.2) [inner sep=15pt]{};

\path (ext) edge (1);
\path  (12) edge (2) (34)  ;
\draw  (1.5,0) circle (0.5) ;
\draw (2.5,0) circle (0.5);
\node (be) at (1.5,0) [inner sep=15pt] {};
\path (be) edge (12);

\node (bi) at (2.5,0)  [inner sep=15pt] {};
\path (bi) edge (3);

\draw (-.3,2.2) circle (.42);
\draw (.3,2.8) circle (.42);
\node at (-.3,2.2) {};    
\node at (.3,2.8) {};  
\fill  (.72,2.8) circle (0.05);
\fill (2.95,-0.2) circle (0.05);
\end{tikzpicture}
\end{center}
Let us proceed with the formal definition.
\begin{Def} \label{defopen}
The space of nested trees on $k$ leaves with vertices labelled by cacti and internal edges labelled by numbers in $(0,1)$
is the quotient
$$N_k(\C):= \coprod_{\ms \in  N_k} (\prod_{S \in \ms} \C_{|S|} \times (0,1]^{\ms'} )\, / \, \sim $$ 
by the equivalence relation that we specify next.
Let us write
\bit
\item $\lambda_S \in (0,1]$, with $S \in \ms'$ for the label of an internal edge $(S,T)$ of $\ms$ 
\item  $x_S \in \C_{|S|}$ for the cactus labelling the vertex $S$.
\eit
Let $\ms-S_0$ be the nested tree $\ms$ with $S_0 \in \ms'$ removed.
If $\lambda_{S_0}=1$, with $S_0 \in \ms'$,  then
 
$$((x_S)_{S \in \ms}, (\lambda_S)_{S \in \ms'}) \sim  ((x'_S)_{S \in \ms-S_0}, (\lambda_S)_{S \in \ms'-S_0}    ) ,$$
where 
$$\theta_\ms((x_S)_{S \in \ms}) = \theta_{\ms-S_0} ((x'_S)_{S \in \ms-S_0}) \in \C_k  $$ 
\end{Def}

Let $(S_0,T_0)$ be the internal edge going out of $S_0$.
In general $x_S \neq x'_{S}$ for $S \subseteq T_0$, but
$x_S=x'_S$ for $S \nsubseteq T_0$.

We notice that the realization $||\ms-S_0||$ is the quotient of the complex $||\ms||$ that 
collapses the closed interval of the 1-cell of $(S_0,T_0)$  to a point.  

\begin{Lem} \label{lemmone} 
For $k>1$ there is a $\Sigma_k$-equivariant homeomorphism
$$\gamma_k: Tr'_k \cong  N_k(\C)$$ 
\end{Lem}

\begin{proof}
Given a labelled tree $tr \in Tr'_k$ with underlying tree $T \in \Tau_k$, 
we shall construct a nested tree $\ms$ with the vertices labelled by cacti, and the internal edges labelled by numbers in the interval $(0,1)$.
Let us remove from $T$ all black vertices labelled by $1$ (they form a non-empty set), and the edges going out of these vertices.
The resulting graph $F$ is a forest (disjoint union of trees). For any component $T'$ of $F$ containing at least a black vertex (equivalently $T'$ is not just a white vertex),
let $S$ be the set of indices of the white vertices belonging to $T'$; let $\lambda_S$ be the maximum label of a black vertex in $T'$. If we divide by $\lambda_S$
the labels of the black vertices in $T'$ 
then we  obtain a new labelled tree  $tr'$ with underlying tree $T'$
and white vertices labelled by elements of $S$,  that can be identified (not canonically) to an element of $Tr_{|S|}$.

We can iterate the procedure, where $tr$ is replaced by $tr'$ and $T$ by $T'$. The procedure stops when 
there are no black vertices left.  At each step of the procedure we need to choose a component of a forest. There is a finite number of distinct iterations.
We define $\ms$ as the nested tree on $k$ leaves containing all $S$ obtained from $tr$ by a finite iteration of this procedure, as well as 
$\{1,\dots,k\}$. 

We specify next the cacti labelling the vertices of $\ms$.
Let $\beta_k:Tr'_k \to \C_k$ be the deformation retraction from the proof of Proposition \ref{4.3}.
We claim that the cactus element 
$x=\beta_k (tr) \in \C_k$
 belongs to the image of the embedding 
$\theta_\ms: \prod_{S \in \ms} \C_{|S|} \to \C_k$.  This holds because for each $S \in \ms$ the union of the corresponding lobes is connected.
We define the labels in $x_S \in \C_{|S|}$ for $S \in \ms$
so that $\theta_\ms((x_S)_{S \in \ms})= x$.

We must specify the labels of the internal edges of the nested tree $\ms$. 
Any internal edge of $\ms$ is the unique edge going out of some vertex $S \in \ms'$ that is not maximal.
Then we associate to this edge the label $\lambda_S \in (0,1)$  defined recursively above. 
\medskip

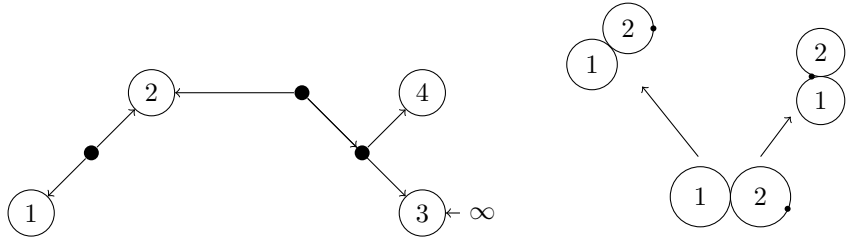
\begin{figure}
\begin{tikzpicture} [scale=.8]
\node (1) at (0,0) [circle,draw] {1};
\node (2) at (2,2) [circle,draw] {2};
 \node (12) at (1,1) [circle,fill,inner sep=2pt,label=left:$s$]{};
\path [->] (12) edge (1)  (12) edge (2);
\node (c) at (4.5,2) [circle,fill,inner sep=2pt, label=above:$1$]{};
\node (34) at (5.5,1) [circle,fill,inner sep=2pt,label=left:$t$]{};
\node (3) at (6.5,0)  [circle,draw] {3};
\node (infty) at (7.5,0) {$\infty$};
\node (4) at (6.5,2)  [circle,draw] {4};
\path [->] (c) edge (2)  (c) edge (34) edge (3) (34) edge (4) (infty) edge (3);
\end{tikzpicture} 
\begin{tikzpicture} [scale=.8]
\node (1) at (-1.5,4)  {1};
\node (2) at (1.5,4)  {2};
\node (3) at (3.5,4) {4};
\node (4) at (5,3) {3};
\node (basis) at (2,0) [inner sep=13pt] {};
\node (12) at  (0,2.5) [inner sep=16pt] {};
\node (34) at (3.5,2)[inner sep=20pt] {};
\path (-.8,3) edge (1);
\path  (12) edge (2) (34) edge (3) ;
\path  (4,2)   edge  (4);
\draw  (1.5,0) circle (0.5) ;
\draw (2.5,0) circle (0.5);
\node (be) at (1.5,0) [inner sep=15pt] {};
\path (be)  edge (12) ; 
\node (bi) at (2.5,0)  [inner sep=15pt] {};
\draw (-.3,2.2) circle (.42);
\draw (.3,2.8) circle (.42);
\node at (-.3,2.2) {};    
\path (bi) edge (34);
\node at (.4,1) {$s$};
\node at (2.7,1) {$t$};
\node at (.3,2.8) {};  
\fill  (.72,2.8) circle (0.05);
\fill (2.95,-0.2) circle (0.05);
\draw (3.5,2.4) circle (0.4); 
\draw (3.5,1.6) circle (0.4);
\node at (3.5,2.4) {}; 
\node at (3.5,1.6) {}; 
\fill (3.90,1.6) circle (0.05);
\end{tikzpicture}
\caption{Pictures of $tr \in Tr'_4$ and $\gamma_4(tr) \in N_4(\C)$.}
\end{figure}

\medskip

The inverse homeomorphism $\gamma_k^{-1}: N_k(\C) \to Tr'_k $ 
 sends $((x_S)_{S \in \ms}, (\lambda_S)_{S \in \ms'})$ to a labelled tree with a base point  $tr \in Tr'_k$,
defining a cactus $\beta_k(tr) \in \C_k \, $.
Our labelled tree $tr$ has $k$ white vertices, and a black vertex $b$ for each intersection point of the lobes of a cactus $x_S$, with $S \in \ms$. The label of this black vertex is the 
product $f(b):=\prod_{S \subseteq T \in \ms'} \lambda_T$. The tree is constructed inductively as union of trees $tr_S$. 
Suppose that the trees corresponding to the edges into $S$ have been constructed. We include the possibility of an open edge into $S$, that gives a single lobe. 
A black vertex $b$ of $tr_S$ is the intersection of $m$ lobes of the cactus $x_S$ labelled $l_1,\dots,l_m$. Each lobe $l_i$ corresponds to an edge $(S_i,S)$ into $S$,  that has been associated to a
labelled tree $tr_i$ with a base point. For $i=1,\dots,m$ we draw an edge $e_i$ going out of $b$ and into the tree $tr_i$, that hits it according to the base point :
if the base point of the cactus is not an intersection point of lobes, then the edge $e_i$ hits the white vertex associated to the lobe.
If the base point is an intersection point then the edge $e_i$ hits the black vertex sitting over it with the largest label.
This procedure determines a tree $tr_S$ with a base point  that depends on the base point of the cactus $x_S$.
At the end of the inductive process the root vertex $R=\{1,\dots,k\}$ yields $tr:=tr_R$
\end{proof}

\begin{Lem} \label{lemmuno} 
There is a family of homeomorphisms $$\Xi_{a_1,\dots,a_k}:FM'_k=F_k(\bcc)/(\bcc \rtimes \R^*)  \to N_k(\C)$$ 
depending continuously on $(a_1,\dots,a_k) \in \stackrel{\circ}{\Delta}_k$ 
\end{Lem}

\begin{proof}
$$\Xi_{a_1,\dots,a_k}= \gamma_k \circ \Phi'_{a_1,\dots,a_k},$$ where $\gamma_k$ is the homeomorphism in Lemma \ref{lemmone}
and $\Phi'_{a_1,\dots,a_k}$ the homeomorphism in the proof of Proposition \ref{4.3}.

\end{proof}

We are now able to exhibit an explicit cellular decomposition of the configuration space modulo translations and positive dilations.
$F_k(\bcc)/(\bcc \rtimes \R^+)$.

\begin{Thm} \label{pocofa}

There is a cell decompositions of $FM'_k= F_k(\bcc)/(\bcc \rtimes \R^+)$ for $k>1$ into open cells
for each $(a_1,\dots,a_k) \in \stackrel{\circ}{\Delta}_k$ .

 A cell is determined by :
\bit
\item A nested tree $\mathcal{S}$ on $k$ leaves ;
\item For each vertex $v$ of $\mathcal{S}$, a cell $\sigma_v$ of the cactus complex $\C_{|v|}$ 
\eit
Such cell has dimension equal to $$\sum_{v \in \mathcal{S}} \dim(\sigma_v)+ |\mathcal{S}|-1$$ and is
homeomorphic to $\prod_{v \in \mathcal{S}}\sigma_v \times (0,1)^\mathcal{S'}$.  
The closure of the cell is homeomorphic to $\prod_{v \in \mathcal{S}}\sigma_v \times (0,1]^\mathcal{S'}$.
The cellular decomposition is compatible with the action of the symmetric group $\Sigma_k$.
\end{Thm}

\begin{proof}
$FM'_k$ is homeomorphic to $N_k(\C)$ by Lemma  \ref{lemmuno}. The cell in the statement consists of all labelled nested trees with "shape"  (underlying nested tree)  $\mathcal{S}$, such that each vertex $v$ is labelled by a point of $\sigma_v$, and the edges have arbitrary labelling. 
This endows $N_k(\C)$ with a cellular decomposition since the structure maps of $\C$ are cellular maps and embeddings.
\end{proof}

\begin{Cor} \label{bc}
The induced CW-structure on the one-point compactification
 $(FM'_k)^+$ 
has a reduced cellular chain complex satisfying
$$s \widetilde{C}_*((FM'_k)^+) \cong B(C_*(\C))(k),$$  where $s$ is the shift in degree by $+1$, 
$B$ is the operadic bar construction, $C_*(\C)$ is the differential graded operad of cellular chains  on the cacti spaces, and $\widetilde{C}_*((FM'_k)^+)=C_*((FM'_k)^+,\infty)$.  

\end{Cor}

\begin{proof}
The operadic bar construction $B(C_*(\C))$  (\cite{Benoit}, Part 2, Appendix C), if we forget the differentials, is the free co-operad
on the symmetric sequence of graded abelian groups $s C_*(\C_k)$ with $k >1$ 
(the unit is removed). 
 Its differential splits as $d=d_{int}+d_{ext}$ where $d_{int}$ is the differential of the free co-operad on the same symmetric sequence in the category of chain complexes,
and $d_{ext}$ is the co-operad co-derivation that is induced by the operad structure maps of $C_*(\C)$. 
In practice for $k>1$
$$B(C_*(\C))(k) = \oplus_{\ms \in N_k} \otimes_{S \in \ms} sC_*(\C_{|S|}),$$ the differential $d_{int}$ is the differential of the tensor product, 
and $d_{ext}$ is the signed sum of all partial composition operations  $C_*(\C_{|S|}) \otimes C_*(\C_{|T|}) \to C_*(\C_{|S|+|T|-1})$ (tensored with the identities of the other factors),
associated to some internal edge $(S,T)$ of some nested tree $\ms \in N_k$.
The CW decomposition of the one point compactification $(FM'_k)^+$ induced by the cell decomposition of  $FM'_k$ from Theorem \ref{pocofa} realizes this construction on the topological level, up 
to a shift by $1$ on the dimension of the cells. 
\end{proof}

We remark that the collection of the suspensions $\Sigma((FM'_k)^+)\cong \Sigma(N_k(\C)^+)$ resembles
an operadic topological bar construction  ( \cite{thesis} \cite{ching})  of $\C$, but the construction is not defined because $\C$ is not a topological operad. 

\subsection{Generating series counting cells of the open moduli spaces} \label{sub:genser}

Let $o_m(FM'_k)$ be the number of $m$-cells of the moduli space $FM'_k = F_k(\bcc)/(\bcc \rtimes \R^+) $.
The symmetric group $\Sigma_k$ for $k>1$ acts freely on the set of cells of $FM'_k$. 
Similarly as in  subsection \ref{sub:count}
we define a generating formal series 

$$o(x,t)= \sum_{m \geq 0,\, k>1 } \frac{o_m(FM'_k)}{k!} t^m x^k$$ 
with integral coefficients.  
Let $P(x,t)$ the generating series in subsection \ref{sub:count} and 
$\tilde{P}(x,t)=P(x,t)-x$. This series counts all cacti cells except the $0$-cell $id \in \C_1$.

\begin{Prop} \label{quad1}
The formal series $o(x,t)$ and $\tilde{P}(x,t)$ satisfy the identity
\begin{equation}  \label{equad1} t \tilde{P}(t o(x,t),t) + t \tilde{P}(x,t) = t o(x,t) \end{equation} 
\end{Prop}

\begin{proof}
By theorem  \ref{pocofa} the cells of $FM'_k$ are indexed by nests on $k$ leaves with the vertices labelled by cacti cells.
The series $to(x,t)$  counts all cells, with dimension raised by 1. 
The summand $t \tilde{P}(x,t)$ counts the cells indexed by nested trees with a single vertex.
The summand $t \tilde{P}(t o(x,t),t)$ counts cells labelled by nested trees with more than one vertex, that are obtained
by grafting a nested tree with a single vertex together with a number of nested trees equal to the valence of the vertex. 
\end{proof}

We apply a similar procedure as in Theorem \ref{count-ds} to compute $o(x,t)$. 

\begin{Thm} \label{count-open}
The generating function counting the cells of the open moduli spaces is 

\begin{gather*}
 o(x,t) = 
-\frac{  2t + \sqrt{(2t + (2t + 3)f(x,t) + 1)^2 - (f(x,t)^2 - 1)((2t + 3)^2 - 1)}   +   (2t + 3) f(x,t) + 1            }     {t((2t + 3)^2 - 1)} 
\end{gather*}
where
$$f(x,t)= \sqrt{(x - 1)^2 - 4tx}+x+2tx-2$$ 

\end{Thm}
\begin{proof}
If we collect $\sqrt{(1-o)^2-4to}$ in  equation \ref{equad1} and square
we obtain a functional quadratic equation in $o$
$$ ((3+2t)^2-1)o^2+(2(3+2t)f+4t+2)o + f^2-1 = 0 $$
The correct sign in the solution is deduced by imposing the 
condition $o(x,t)=0$ for $x=0$. 
\end{proof}

\section{Compactification of configuration spaces and spaces of trees} \label{sec:fulton}

We recall the definition of the Fulton MacPherson space $FM(k)$, that is a compactification of the open moduli space
$FM'_k= F_k(\bcc)/ (\bcc \rtimes \R^+)$. There is a version $FM_n(k)$ of it for any $n \geq 1$  but we are only interested 
in $FM=FM_2$ here.
The Lie group of rotations and positive dilations  $\Aff:= \bcc \rtimes \R^+$ acts smoothly and freely on  $\bcc^k - \{0\}$ for $k>1$.
The quotient   $   (\bcc^k - \{0\})/\Aff$  is diffeomorphic to the $(2k-3)$-sphere $S(\bcc^k) \cap H$
that is the intersection of the complex hyperplane  $H=\{(z_1,\dots,z_k) \, | \, z_1 + \dots +z_k=0   \}$
with the unit sphere of $\bcc^k$. Namely each class of the quotient has a unique representative in the $(2k-3)$-sphere.
 The diffeomorphism is $\Sigma_k$-equivariant with respect to the action permuting the coordinates.
Consider for $2 \leq j \leq k-1$ and $S=\{i_1,\dots,i_j\}$ the generalized diagonal 
$$\Delta_S= \{[z_1,\dots,z_k] \, | \, z_{i_1}=\dots=z_{i_j}  \, \}  \subset  (\bcc^k - \{0\})/(\bcc \rtimes \R^+)$$
that is an oriented submanifold diffeomorphic to $S^{2k-2j-1}$.
Consider then the oriented blow-up $Bl_S$ of $(\bcc^k - \{0\})/(\bcc \rtimes \R^+)$ along 
$\Delta_S$. Since $FM'_k$ does not intersect any generalized diagonal $\Delta_S$ there is a natural inclusion 
$i_S:FM'_k \to Bl_S$ for each $S$. 
\begin{Def}
The Fulton MacPherson space $FM(k)$ is the closure of the image of 
$$\prod_{S \subset \{1,\dots,k\},  1<|S|<k} i_S: FM'_k \rightarrow \prod_S Bl_S$$
\end{Def}

It is crucial for this paper that the Fulton-MacPherson space $FM(k)$ is a manifold with corners, and also a manifold with faces.
This is due to Axelrod-Singer \cite{AxSing}, Lambrechts-Volic \cite{LV} and Sinha \cite{Sinha}.
We recall first the definition of a manifold with corners and of a manifold with faces.

\begin{Def}
A $n$-dimensional oriented smooth manifold with corners $M$ is a space that looks locally like the cube $[0,1]^n$.
It comes with an atlas inducing smooth and oriented changes of coordinates. 
A codimension $m$ open stratum is a component of the space $M_{-m}$ of points that have a neighbourhood 
diffeomorphic to a neighbourhood of a point in $\{0\}^m \times (0,1)^{n-m} \subset [0,1]^n$.  A closed stratum is the closure
of an open stratum in $M$.
\end{Def}

\begin{Def}
A manifold with faces is a manifold with corners, such that each codimension $m$ stratum is the transverse intersection of
$m$ codimension 1 strata, called the faces. 
\end{Def}

We refer to chapter 4 of \cite{Verona} for more details on manifolds with faces. A cube is a standard example of manifold with faces, whereas the "eye-drop" is a manifold with corners but not a manifold with faces.

We need a preliminary definition in order to describe the strata of $FM(k)$.

\begin{Def} 
For each nested tree $\mathcal{S} \in N_k$ on $k$ leaves, and each element $S \in \mathcal{S}$, define $\hat{S}$ as the quotient of $S$ with respect to 
the equivalence relation such that $s \sim s'$ if there is a proper subset $T \subset S$ with $s,s' \in T$ and $T \in \mathcal{S}$. 
If we look at the vertex associated to $S$ in the realization tree $||\mathcal{S}||$, we can identify $\hat{S}$ to the set of incoming edges of the vertex. 
We define the configuration space $F_{\hat{S}}(\bcc)$ as the space of embeddings $\hat{S} \hookrightarrow 
 \bcc$,  that can be identified with the space of maps 
$x:S \to \bcc$ such that  $x(s)=x(s')$ if and only if $s,s' \in T$ with $T$ as above. We use the notation $x_{S,s}:=x(s)$ for $s \in S$,
and $x_S:=x$.  When we take quotients of these configuration spaces by the action of the affine group $Aff=\bcc \rtimes \R^+ $ we pick the unique representatives such that 
$\sum_{s \in S}x_{S,s}=0$ and $\sum_{s \in S} |x_{S,s}|^2 = 1$.
\end{Def}

The following theorem gives a characterization of the strata of $FM(k)$, and of the charts around each point.  We refer to \S 5.3-5.4 in \cite{AxSing}, and  \S 5.9.2 in \cite{LV}. 

\begin{Thm} \cite{AxSing} \cite{LV} \cite{Sinha}  \label{thm:strata}
The space $FM(k)$ is a manifold with corners of dimension $2k-3$. Its strata correspond bijectively to nested trees on $k$ leaves. 
The open stratum $\stackrel{\circ}{FM(\mathcal{S})}$  corresponding to a nested tree $\mathcal{S}$ 
has codimension equal to the number $|\mathcal{S}|-1= |\mathcal{S}'|$ of elements of
$\mathcal{S}$ decreased by one. 
There is a diffeomorphism
$$\beta_\mathcal{S}: \prod_{S \in \mathcal{S}} F_{\hat{S}}(\bcc)/\Aff \cong \stackrel{\circ}{FM(\mathcal{S})}$$
extending to a diffeomorphism onto the closed stratum
$$\overline{\beta_\mathcal{S}}:\prod_{S \in \mathcal{S}} FM(|\hat{S}|)   \cong FM(\mathcal{S})\, .$$
 \end{Thm}

In particular the codimension 0 stratum of $FM(k)$ is $FM'_k=F_k(\bcc)/\Aff$, and the former is a compactification of the latter.
There are explicit charts for $FM(k)$ constructed as follows.
Given a point $$(x_S)_{S \in \mathcal{S}}  \in  \prod_{S \in ms} F_{\hat{S}}(\bcc)/\Aff \cong \stackrel{\circ}{FM(\mathcal{S})} \, ,$$ 
for sufficiently small open neighbourhoods $U_S \subset  F_{\hat{S}}(\bcc)/\Aff $ of $x_S$ ( $S \in \mathcal{S}$) , and for sufficiently small constants $\varepsilon_{S'} >0$
($ S' \in \mathcal{S}'$)
 there is a smooth embedding  
$$\gamma_\ms: \prod_{S \in \ms}  U_S \times \prod_{S'  \in \mathcal{S'}}    [0,\varepsilon_{S'})     \hookrightarrow
 FM(k)$$ onto a neighbourhood of 
 $$\beta_\ms((x_S)_{S \in \ms})=\gamma_\ms((x_S)_{S \in \ms}, (0)_{S' \in \ms'})$$
such that  for $y_S \in U_S$ and $0<\delta_{S'} <\varepsilon_{S'}$
 $$\gamma_\ms((y_S)_S,(\delta_{S'})_{S'})=[z_1, \dots,  z_k] \in  F_k(\bcc)/\Aff$$ and the $j$-th coordinate is 
$$z_j =   \sum_{j \in S \in \mathcal{S}} (  \prod_{S \subseteq S' \in \mathcal{S}' } {\delta_{S'}}  ) y_{S,j}    $$

\begin{Exa}

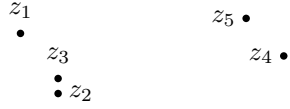
The following picture represents an element in the image of $\gamma_\ms$, with $\ms= (1(23))(45)$ of the form
$$[z_1,z_2,z_3,z_4,z_5] = \gamma_\ms (x_{23},x_{123},x_{45},x_{12345},\delta_{23},\delta_{123},\delta_{45}) $$ 
with $x_{23} = [z_2,z_3],  \, x_{123}=[z_1,\zeta], \, x_{45}=[z_4,z_5], \, x_{12345} =[h,w]. $
The radii of the circles centered at $\zeta, h$ and $ w$  are respectively proportional to $\delta_{23}, \delta_{123}$ and 
$\delta_{45}$.
\bigskip

\begin{center}
\begin{tikzpicture}  [scale=1.6]
\node at (0,.8) [inner sep=1pt,draw,circle, label=above:$z_1$] {};
\node at (.5,0) [inner sep=1pt,draw,circle,label=below right:$z_2$] {};
\node at (.5,.2) [inner sep=1pt,draw,circle,label=above right:$z_3$] {};
\node (zeta) at (.5,.1)  [inner sep=1pt,fill,circle,label=right:$\zeta$] {};
\draw (zeta) circle (.25cm); 

\node at (3.5,0.5) [inner sep=1pt,draw,circle,label=left:$z_4$] {};
\node at (3,1) [inner sep=1pt,draw,circle,label=right:$z_5$] {};
\node (w) at (3.25,.75)  [inner sep=1pt,fill,circle,label=left:$w$] {};
\node (h) at (.3,.4)  [inner sep=1pt,fill,circle,label=left:$h$]{};
\draw (h) circle (.8cm); 
\draw (w) circle (.5cm);

\end{tikzpicture}
\end{center}


\end{Exa}

\begin{Rem}
The space $FM(k)$ is indeed a manifold with faces.  For each proper subset $S \subset \{1,\dots,k\}$ of cardinality $1<|S|<k$ consider the nested tree
$\ms(S)= \{S, \{1,\dots,k\}\}$. Then $FM(S):=FM(\mathcal{S}(S))$ is a codimension 1 stratum (a face), containing the configurations where the points with labels in $S$ come together (with respect to the other points).
For any nested tree $\ms$ on $k$ leaves the corresponding stratum is an intersection of faces $FM(\ms) = \cap_{S \in \ms'} FM(S)$, and this intersection is transverse by the formula of the explicit charts in Theorem \ref{thm:strata}. 
\end{Rem}

\subsection{Extension to the boundary}

We define a compact space of labelled nests $\bar{N}_k(\C)$ following essentially definition \ref{defopen},
but we allow the internal edges to be labelled by 0.
\begin{Def}
The space $\bar{N}_k$ is the quotient 
$$\bar{N}_k(\C):= \coprod_{\ms \in  N_k} (\prod_{S \in \ms} \C_{|S|} \times [0,1]^{\ms'} )\, / \, \sim $$ 
by the same relation as in definition \ref{defopen}.
\end{Def}
Clearly $\bar{N}_k(\C)$ is a compactification of $N_k(\C)$.

We wish to extend the homeomorphism 
$\Xi_{a_1,\dots,a_k}: FM'_k  \cong  N_k(\C)$ from Lemma \ref{lemmuno} 
to the boundary 
$$\partial FM(k)=FM(k)-FM'_k$$ for $k>2$ in order to obtain a homeomorphism 
$\overline{\Xi}_{a_1,\dots,a_k}:FM(k) \cong \bar{N}_k(\C) $.

We define first some functions associating an angle to a configuration point.
\begin{Def}
Let $$\theta^i_{a_1,\dots,a_k}: FM'_k \to S^1$$ for each  $i=1,\dots,k$ be the following continuous function.
For a given configuration $(z_1,\dots,z_k)$ with weights $(a_1,\dots,a_k)$  
the flow lines out of $\infty$ correspond to $S^1$ and have a cyclic ordering corresponding
to the standard orientation of $S^1$. Any flow line $\gamma$ out of $\infty$ can be extended to a piecewise smooth flow line,
that continues after hitting any critical point along the rightmost flow line going out. The flow line hits eventually some configuration point $z_j$.
Let $\gamma_i$ be the first flow line in the cyclic ordering, starting from $1 \in S^1$, hitting the configuration point
$z_i$. 
Let $\gamma_i:(0,b) \to \bcc$ be a parametrization with $\lim \gamma_{t \to 0^+}(t)=z_i$. Then
$$\theta^i_{a_1,\dots,a_k}(z_1,\dots,z_k)= \lim_{t \to 0^+} \gamma'_i(t) / |\gamma'_i(t)|$$.
\end{Def}
The following picture shows the vectors $\theta^i(z_1,z_2,z_3)$.

\begin{center}
\begin{tikzpicture}
\node (3) at (6.5,0) [circle,draw,inner sep=1pt] {$z_1$};
\node (4) at (6.5,.8) [circle,draw,inner sep=1pt] {$z_2$};
\node (34) at (6.5,.4) [circle,fill,inner sep=1pt]{};
\path (34) edge (3)  (34) edge (4);

\node (5) at (9.5,1.2) [circle,draw,inner sep=1pt] {$z_3$};
\node (w) at   (7.8,.6)   [circle,fill,inner sep=1pt] {};
\node (infi) at (12,.5) {$+\infty$};
\draw (w) .. controls (8.5,.6) .. (5) ;
\draw (w) .. controls (7.3,.6) .. (4) ;
\draw (infi) .. controls (10.5,.5) .. (5) ;

\draw [<-,thick] (3) -- (6.5,-.5);
\draw [<-,thick] (4) -- (6,1);
\draw [<-,thick] (5) -- (9.2,1.6);

\node at (6.2,-.5) {$\theta^1$};
\node at (6.2,1.3) {$\theta^2$};
\node at (9.6,1.8) {$\theta^3$};

\end{tikzpicture}
\end{center}

\begin{Prop}
The function $\theta^i_{a_1,\dots,a_k}$ is continuous, depends continuosly on the parameters $a_1,\dots,a_k$, and is nullhomotopic for each $i=1,\dots,k$.
\end{Prop}

\begin{proof}
The continuity follows from the continuous dependence of the space of flow lines on the data.
The nullhomotopy between $\theta_i$ and the constant function with value $1 \in S^1$ is provided by letting  $a_i \to 1$ and $a_j \to 0$ for $j \neq i$. 
\end{proof}

For notational convenience we adopt a coordinate-free approach:
for each finite set $R$ of $r$ elements we define $FM'_R$ as the space of injective 
maps $R \hookrightarrow \bcc$ modulo the action of the affine group $\bcc \rtimes \R^+$ on the target. 
Of course an ordering of $R$ induces a diffeomorphism $FM'_R \cong FM'_r$.  Similarly we define  $N_R \cong N_{r}$ as the set of 
isomorphism classes of nested tree with $R$ as set of leaves, and the space $N_R(\C) \cong N_{r}(\C)$  of labelled nested trees.
For a given function $a:R \to \R^+$, 
there is a homeomorphism
$\Xi_a: FM'_R \cong N_R(\C)$, that for a given ordering $y_1,\dots,y_{r}$ of the elements of $R$ corresponds
to $\Xi_{a(y_1)/A,\dots,a(y_r)/A}$, with $A=\sum_{i=1}^r a(y_i)$, under the identifications above.
We also have maps $\theta^i_a:FM'_R \to S^1$ for each $i \in R$.

\begin{Def} \label{def:xi}
Consider the function
$$\overline{\Xi}_{a_1,\dots,a_k}: FM(k) \rightarrow \bar{N}_k(\C)$$
defined stratum-wise as follows.
Recall that for a given open stratum $\stackrel{\circ}{FM(\ms)}$,
with 
$\ms \in N_k$ nested tree on $k$ leaves, according to our coordinate-free notation
we have a diffeomorphism
$$\beta_\ms: \prod_{S \in \ms} FM'_{\hat{S}} \cong  \stackrel{\circ}{FM(\ms)}$$  introduced in Theorem \ref{thm:strata}

Choose an element $x_S \in FM'_{\hat{S}}$ for each $S \in \ms$. 
In order to define the labelled nested tree 
$$\overline{\Xi}_{a_1,\dots,a_k}(\beta_\ms((x_S)_S)) \in \bar{N}_k(\C)$$ we need to specify a nested tree $T \in N_k$, and the labels of its vertices and internal edges. 
For each $S \in \ms$ consider the quotient projection 
$\pi_S:S \to \hat{S}$ and the function $a_{S}: \hat{S} \to \R^+$ defined by
$a_{S}(\iota)= \sum_{i \in \pi_S^{-1}(\iota)  } a_{i}$. 
Let $T_S \in N_{\hat{S}}$ be the nested tree 
that is the shape of  $\Xi_{a_S}(x_S) \in N_{\hat{S}}(\C)$. The tree $T \in N_k$ is obtained intuitively by grafting together all the nested trees $T_S$ with $S \in \ms$. 
Rigorously $T$ is the collection of all sets $\pi^{-1}_S(U) \subseteq \{1,\dots,k\}$, with $U \in T_S$ and $S \in \ms$.
The realization $||T||$ is indeed homeomorphic to the union of the complexes $||T_S||$, topologized appropriately. We specify the labels 
in $[0,1]$ of the internal edges
of $T$:
\bit
\item If an internal edge comes from a nested tree $T_S$, i.e. has the form 

$(\pi^{-1}_S(S_1),\pi^{-1}_S(S_2))$, with 
$e=(S_1,S_2) \in E(T_S)$, then its label is the same as the label of $e$ in $\Xi_{a_S}(x_S) \in N_{\hat{S}}(\C)$.   
\item If an internal edge does not come from a nested tree then its label is set to $0$.
\eit

We need to rotate the configurations $x_S$ labelling vertices $S$, that are not the root, by an appropriate angle. 
We define inductively an angle $\alpha_S \in S^1$ for  $S \in \ms$.
If $S=\{1,\dots,k\}$ is the root then $\alpha_S=1$. 
Suppose that $S \in \ms', \, (S,R)$ is the edge going out of $S$, and $\alpha_R$ has been defined. 
Given the quotient projection $\pi_{R}:R \to \hat{R}$, the set $\pi_R(S)$ consists of a single element, that we indicate by the same symbol. Then
$$\alpha_S:=  \alpha_R  \cdot  \theta^{\pi_{R}(S)}_{a_R}(\alpha_R^{-1} \cdot x_R )    $$
Here $\alpha_R^{-1} \cdot x_R$ is the rotation of the configuration $x_R \in FM'_{\hat{R}}$ by the angle $\alpha_R^{-1}$.

We are now able to define the label of a vertex of $T$, that is the cactus labelling the corresponding vertex of 
$\Xi_{a_S}(  \alpha_S^{-1} \cdot x_S)$.   
\end{Def}

It follows from the definition that
the restriction of $\overline{\Xi}=\overline{\Xi}_{a_1,\dots,a_k}: FM(k) \to \bar{N}_k(\C)$ to each open stratum is continuous.
Let us consider the compatibility between different strata. 
\begin{Lem} \label{xicont} 
The  function $\overline{\Xi}$ is continuous.
\end{Lem}
\begin{proof}
We need to understand what happens when some points come together in the compactified configuration space $FM(k)$.
For the sake of simplicity we just describe the case when two out of three points come together.  Consider $[z_1-\ep w,z_1+\ep w ,z_2] \in FM'_3 \subset FM(3)$ for $\ep>0$. The limit 
for $\ep \to 0$ is associated to the nested tree 
$\ms = \{R,S\}$ with $R=\{1,2,3\}$ and $S=\{1,2\}$ and is exactly $\beta_\ms(x_R,x_S) \in \partial FM(3)$ with 
$x_R=[z_1,z_2]$ and $x_S=[-w,w]$.   
For fixed weights $a_1,a_2,a_3$ and $\ep>0$ there are two (possibly coinciding) critical points $w_1^\ep,w_2^\ep$
of $$h(z)=(z-z_1+\ep w)^{a_1} (z-z_1-\ep w)^{a_2} (z-z_2)^{a_3}$$
 that are the zeros of
$$p_\ep(z)= a_1 (z-z_1-\ep w)(z-z_2) + a_2(z-z_1+\ep w)(z-z_2) + a_3 (( z-z_1)^2-\ep^2 w^2) $$
\begin{figure} 
\begin{tikzpicture}  
\node (1) at (0,0) [circle,draw,inner sep=2pt,label=left:$z_1+\varepsilon w$] {};
\node (2) at (0,.8) [circle,draw,inner sep=2pt,label=left:$z_1-\varepsilon w$] {};
\node (12) at (.2,.4) [circle,fill,inner sep=1pt,label=left:$w_1^\varepsilon$]{};
\draw [->] (12) .. controls (.2,.2) ..  (1)  ;
\draw [->] (12) .. controls (.2,.6) ..  (2)  ;
\node (z2) at (4,.8) [circle,draw,inner sep=2pt,label=above:$z_2$] {};
\node (w) at (2,.6) [circle,fill,inner sep=1pt,label=above:$w_2^\varepsilon$] {};
\draw [->] (w) .. controls (2.5,.6) .. (z2);
\draw [->] (w) .. controls (1.5,.6) .. (12);
\end{tikzpicture}
\caption{Flow lines between critical points and zeros} 
\end{figure}
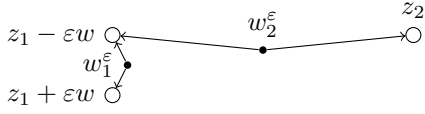
The zeros depend continuously on $\ep$, and in the limit tend to the zeros of 
$$p_0(z)=(z-z_1) ((a_1+a_2)(z-z_2) + a_3 (z-z_1))\, ,$$ that are $w_1^0 = z_1$ and 
$w_2^0$ that is the critical point associated to the configuration $[z_1,z_2]$ with weights $(a_1+a_2,a_3)$.
For $\ep \to 0$ we have that $|h(w_1^\ep)| \to 0$ and $|h(w_2^\ep)| \to |h(w_2^0)| \neq 0$, so the ratio of the critical values tends to 0.
This implies that the label of the unique internal edge $(S,R)$ of the labelled nested tree $\Xi_{a_1,a_2,a_3}[z_1-\ep w,z_1+\ep w ,z_2]$ tends to 0.
The cactus labelling the root vertex $R$ tends to the cactus of $\Xi_{a_1+a_2,a_3}[z_1,z_2]$. On the other hand the cactus labelling $S$ is obtained by studying the flow lines near $z_1$.  The argument of the factor $(z-z_2)^{a_3} \neq 0$ of $h(z)$ is almost constant near $z_1$ whereas the argument of 
$(z-z_1+\ep w)^{a_1} (z-z_1-\ep w)^{a_2}$ varies near $z_1$. The corresponding unbased cactus tends to the
unbased cactus underlying $\Xi_{a_1,a_2}[z_1-\ep w ,z_1+\ep w]= \Xi_{a_1,a_2}[-w,w]$. 
However the base point comes from the first flow line (in cyclic ordering after the flow line 1) hitting the cactus with 2 lobes corresponding to the 
configuration points $z_1-\ep w, z_1+ \ep w$. The tangent direction to this flow line at the intersection point with a circle of centre $z_1$ and radius $\sqrt{\ep}$
(tending to $z_1$ more slowly than the configuration points) 
tends to $\theta:= \theta^1_{a_1+a_2,a_3}[z_1,z_2]$  for $\ep \to 0$. The cactus of $S$ is obtained by 
rotating the cactus $\Xi_{a_1,a_2}[-w,w]$ by $\theta^{-1}$. 
In the general case for any nested tree $\ms$ on $k$ leaves we must show the continuity of $\overline{\Xi}$ when a
configuration $$[z_1,\dots,z_k]=\gamma_{\ms}((x_S)_S,(\ep_{S'})_{S'})  \in FM'_k$$ approaches 
$\beta_\ms((x_S)_S) \in FM(\ms)$, where $x_S \in FM'_ {\hat{S}}$, and the variables $\ep_{S'} >0$ tend to 0.
This is done similarly as above, studying the critical points  and the flow lines of $h$ as $\ep_{S'} \to 0$.

\end{proof}

\begin{Thm} \label{xibar} 
The map $\overline{\Xi}=\overline{\Xi}_{a_1,\dots,a_k}:FM(k) \to \bar{N}_k(\C)$ is a homeomorphism for each choice of parameters $a_1,\dots,a_k >0$.
\end{Thm}
\begin{proof}
The map $\overline{\Xi}$ is one to one, and continuous by Lemma \ref{xicont}. So it is a homeomorphism because its domain is compact and its range is Hausdorff.
\end{proof}

\subsection{Cellular decomposition}
By Theorem \ref{xibar} the Fulton Mac-Pherson space $FM(k)$ inherits a cellular decomposition for each choice of
parameters $a_1,\dots,a_k >0$, since $\bar{N}_k(\C)$ has a regular cellular decomposition.
The cells are labelled by nested trees on $k$ leaves, with each vertex $v$ labelled by a nested tree on $|v|$ leaves,
that in turn has the vertices labelled by cacti cells. 
This double passage motivates the name \em{metatrees} used by Kontsevich and Soibelman  \cite{KS}.
\begin{Thm} \label{condo}
There is a finite regular CW decomposition of $FM(k)$, depending on parameters $(a_1,\dots,a_k) \in (\R^+)^k/\R^+ \cong \stackrel{\circ}{\Delta}_k$.
A cell is determined by 
\bit
\item A nested tree $\ms$ on $k$ leaves. 
\item For each vertex $v$ of $\ms$, a nested tree $\ms_v$ on $|v|$ leaves. 
\item For each vertex $u$ of $\ms_v$ a cell $\sigma_u$ of the cactus complex $\C_{|u|}$
\eit
Such cell has dimension equal to 
$$\sum_{v \in \ms} (|\ms_v|-1 + \sum_{u \in \ms_v} dim(\sigma_u)    )    $$
and its closure is homeomorphic to 
$$\prod_{v \in \ms} ([0,1]^{\ms'_v} \times \prod_{u \in \ms_v} \sigma_u )$$
\end{Thm}
\begin{proof}
The space $\bar{N}_k(\C)$ has a regular CW decomposition. A fixed cell $\tau$ is determined by 
\bit
\item A nested tree $\mathcal{T}$ on $k$ leaves 
\item for each vertex $u \in \mathcal{T}$ a cacti cell $\sigma_u$ of $\C_{|u|}$ 
\item A decomposition $I=I_0 \coprod I_1$ of the set of internal edges of $\mathcal{T}$  . 
\eit
The cell contains all labelled nested trees with shape $\mathcal{T}$, with each vertex $u \in \mt$ labelled 
by an element of $\sigma_u \subset \C_{|u|}$, with each internal edge in $I_0$ labelled by 0, and each internal edge in $I_1$ 
labelled by a positive number in $(0,1)$. The closure of this cell $\tau$ is the union of the codimension 1 cells with a similar description that are obtained
\ben  \label{cond}
\item Replacing one cell $\sigma_u \subset C_{|u|}$ by a cell in its boundary $\sigma'_u \subset \partial \sigma_u$ 
\item Replacing an internal edge in $I_1$ going out of $u$ (labelled by $\sigma_u$) and into $u'$ (labelled by $\sigma_{u'}$) by a vertex 
labelled by a cell in the image of a partial composition product $\sigma_u \times \sigma_{u'} \to \C_{|u|+|u'|-1}$.  
\item Moving one internal edge  from $I_1$ to $I_0$ (so that the internal edge gets labelled by 0)
\een
The connection with the description in the statement is the following.   
Intuitively, if we cut all internal edges in $I_0$ (this makes sense in the realization $||\mt||$),
then we get a forest (a finite set of nested trees $\ms_v$). If we collapse each internal edge of the forest (in $I_1$) to a point $v$,
and glue back the internal edges in $I_0$,  then we obtain a nested tree $\ms$. 
More precisely $\ms$ contains all vertices $v$ of $\mt$ that are not sources of edges in $I_1$ (i.e. the root and the sources of edges in $I_0$).
The nest $\ms_v$ contains all vertices of $\mt$ that are connected  to $v$ by a sequence of edges in $I_1$. 
Let us consider the homeomorphism $\overline{\Xi}=\overline{\Xi}_{a_1,\dots,a_k}:FM(k)\cong \bar{N}_k(\C)$ of Theorem \ref{xibar}. We transfer the CW structure of $\bar{N}_k(\C)$ 
to $FM(k)$ via this homeomorphism. 
\end{proof}
\rm

\begin{Exa}
The following picture represents a cell of $\bar{N}_6(\C)$, encoded by a nested tree with one vertex
consisting of the dashed circle, that is labelled by the nested labelled tree inside, and two other vertices that are labelled by
cacti cells ( regarded as labelled corollas). Equivalently the cell is encoded by a nested tree with four vertices labelled by cacti cells, and internal edges of two types:  two edges in $I_0$,  and one edge in $I_1$, represented by the thick segment in the picture.  The cell has dimension $2$, since $3$ cacti cells have dimension $0$, the cacti cell in the middle has dimension $1$, and the thick edge contributes by $1$.

\bigskip

\begin{tikzpicture} [scale=1]
\draw  (0,0) to [out=45,in=0] (0,1);
\draw  (0,0) to [out=135,in=180] (0,1);
\draw (0,0) to [out=45,in=90] (1, 0);
\draw  (0,0) to [out=-45,in=-90] (1,0);
\draw  (0,0) to [out=135,in=90] (-1,0);
\draw  (0,0) to [out=225,in=270] (-1,0);
\node (c2) at (0,0.5) [inner sep=20pt] {} ; 
\node (c6) at  (0.5,0)[inner sep=15pt] {}; 
\node (c1) at (-0.5,0)[inner sep=15pt] {} ; 
\fill  (0,-0.05) circle  (.05cm);
\node (126) at (0,0) [inner sep=12pt]{};
\node (2) at (0,2){2}; 
\node (6) at (1,2){6};
\node (1) at (-1,2){1};
\draw (2) -- (c2);
\draw (1) -- (c1);
\draw (6) -- (c6);
\node (root) at (2,-2)[inner sep=20pt]{};

\draw (3.2,1)  circle (.4cm)  ;
\node (c5) at (3.2,1) [inner sep=12pt]{};
\draw   (4,1) circle (.4 cm) ;
\node (c4) at (4,1) [inner sep=10pt]{};
\fill (3.6,0.8) circle (.05 cm);
\node (cmez) at (3.6,1) [inner sep=15pt]{};

\draw (2.2 ,-.5)  circle (.4cm)  ;
\node (c3) at (2.2,-.5) [inner sep=15pt]{};
\node (345) at (2.6,-.5) [inner sep=15pt]{};
\draw   (3,-.5) circle (.4 cm) ;
\node (c54) at (3,-.5) [inner sep=15pt]{};
\draw (c54) -- (cmez);

\node (3) at (2,2){3};
\node (5) at (3.5,2){5};
\node (4) at  (5,2) {4};
\draw (c3) -- (3) ;
\draw (c5) -- (5);
\draw (c4) -- (4) ;

\fill (3.4,-.5) circle (.05 cm);
\draw [dashed] (2.5,-1.4) circle (1.7cm); 

\draw (1.6,-2)  circle (.4cm)  ;
\node (r1) at (1.6,-2) [inner sep=12pt]{};
\draw   (2.4,-2) circle (.4 cm) ;
\node (r2) at (2.4,-2) [inner sep=15pt]{};
\fill (2,-1.8) circle (.05 cm);
\node (cmez) at (3.6,1) [inner sep=15pt]{};

\draw (r1) -- (126);
\draw [very thick] (r2) -- (345);

\end{tikzpicture}

\end{Exa}

\begin{Rem}
If $a_1=\dots=a_k$ then the homeomorphism $\overline{\Xi}$ is $\Sigma_k$-equivariant. In this case $\Sigma_k$ acts cellularly on $FM(k)$
\end{Rem}

\begin{Cor} \label{omb}
The cellular chain complex of the CW-structure on $\bar{N}_k(\C) \cong FM(k)$ for $k>1$ satisfies 
$$C_*(FM(k)) \cong  \Omega B  (C_*(\C)) (k) ,$$
where $B$ is the operadic bar construction and $\Omega$ the operadic cobar construction.
\end{Cor}

\begin{proof}
The operadic bar construction was described in the proof of corollary \ref{pocofa}. 
The operadic cobar construction $\Omega(\mathcal{D})$ (\cite{Benoit}, Part 2, Appendix C) of a cooperad $\mathcal{D}$ of  chain complexes, 
if we forget the differentials, is the free operad on the  sequence
of graded abelian groups $s^{-1}\mathcal{D}(k)$ with $k>1$. Its differential splits as $d=d_{int}+d_{ext}$, where $d_{int}$ is the differential of the free operad 
and $d_{ext}$ is the operad derivation induced by the cooperad structure maps of $\mathcal{D}$. Therefore 
$$\Omega B(C_*(\C))(k) =  \oplus_{\ms \in N_k}  \otimes_{v \in \ms}   \oplus_{\ms_v \in N_{|v|}}  s^{-1}( \otimes_{u \in \ms_v}      sC_*(\C_{|u|}) )$$
This complex is naturally isomorphic as a graded abelian group to the cellular chain complex of $\bar{N}_k(\C) \cong FM(k)$. 
The isomorphism is compatible with the differential. Namely 
the internal differential of the cobar construction takes care of the part of the boundary of type $(1)$  
and $(2)$ in the proof of Theorem \ref{condo}, as seen in corollary \ref{bc}.
The external differential  of the cobar construction takes care of the part of the boundary of type $(3)$ in the proof of Theorem \ref{condo}.
\end{proof}

\begin{Rem}
Corollary \ref{omb} shows that the dg-operad $C_*(FM)$ is isomorphic to the standard cofibrant resolution of the dg-operad  $C_*(\C)$,
that is condition 3) of Theorem \ref{princ}.
\end{Rem}
 
\subsection{Generating series counting cells of the Fulton Mac-Pherson spaces}
Let $f_m(FM(k))$ be the number of $m$-cells of the Fulton Mac-Pherson space $FM(k)$. The symmetric group $\Sigma_k$ acts freely 
on the set of cells of $FM(k)$ for $k>1$. We define the a generating formal series 
$$F(x,t)= \sum_{m\geq 0, k>1}  \frac{f_m(FM(k))}{k!} t^m x^k$$
with integral coefficients.  Recall that $o(x,t)$ is the generating function counting the cells of the open moduli spaces $FM'_k$,
that we computed in subsection \ref{sub:genser}.
\begin{Prop}
The formal series $o(x,t)$ and $F(x,t)$ satisfy the identity 
$$o(t^{-1}F(x,t),t)+o(x,t) = F(x,t)$$  \label{ofm}
\end{Prop}
\begin{proof}
The proof is similar to that of Proposition \ref{quad1}
Namely $t^{-1}o(x,t)$ counts the cells of $FM(k)$ labelled by a tree $\ms$ with a single vertex (that is the cells of the interior $FM'_k$)
shifted down in dimension by 1. The expression
$t^{-1}o(t^{-1}F(x,t),t)$ counts all cells of $FM(k)$  (shifted down by 1) labelled by a tree $\ms$ with more than one vertex, obtained by grafting the root
vertex labelled by an open cell together with arbitrary trees labelling cells of Fulton MacPherson spaces.  Therefore
the sum $$t^{-1}o(t^{-1}F(x,t),t)+t^{-1}o(x,t) = t^{-1}F(x,t)$$ counts all cells (shifted down by 1) and this yields the statement.
\end{proof}

\begin{Thm} \label{count-closed}
The function $F(x,t)$ is an algebraic expression that can be computed explicitly
\end{Thm}
\begin{proof}
We deduce from equation \ref{ofm} , 
by collecting 
$$\sqrt{((2t+3)^2-1)t(F-o)+2t+(2t+3)f(F/t,t)+1}$$
on one side and squaring once, 
the equation 
\begin{gather*}(2t+(2t+3) f(F/t,t)+1)^2-(f(F/t,t)^2-1)((2t+3)^2-1) = \\
 (((2t+3)^2-1)t(F-o)+2t+(2t+3)f(F/t,t)+1)^2 
 \end{gather*}
Collecting $\sqrt{(F/t	-1)^2-4F}$ and squaring again we obtain a quartic functional equation in $F$
$$aF^4 + b F^3 + c F^2 + d F + e = 0 $$
The requirement $F(0,t)=0$ implies that the correct solution is $$F = -b/(4a)-S-1/2  \sqrt{-4S^2-2p+q/S}$$
following the convention in \cite{wiki}. The expression of $F$ is too long to be reported here.
\end{proof}

\section{Operad composition and cellular decomposition} \label{sec:last}
\rm
The (familiy of) cellular decompositions of $FM(k)$ from the previous section are not compatible with the operad composition, that are not cellular maps. This depends on two issues, as anticipated in the introduction. 
Remember that a cellular decompositions of $FM(k)$ is pulled back from that of $\bar{N}_k(\C)$ by a homeomorphism
$\overline{\Xi}_a:FM(k) \cong \bar{N}_k(\C)$ introduced in Definition \ref{def:xi}, depending on weights $a=(a_1,\dots,a_k)   \in (\R_+)^k$.
 The first issue is that via these homeomorphisms the composition in the Fulton-MacPherson operad  corresponds to grafting nested trees, as one expects, but in addition the base points of the cacti indexing the vertices in the upper tree rotate around.
We will fix this in subsection \ref{sub:rot} by deforming appropriately the homeomorphisms $\overline{\Xi}_a$ to homeomorphisms $\Xi'_a$ so that the rotations disappear, as stated in Theorem \ref{forst}.

The second issue is that after this fix the Fulton-MacPherson operad composition is compatible with the cell structures in the sense that we have to multiply the weights indexing the cell structures. However we would like a single cellular decomposition in each arity. In subsection \ref{sub:weights} we fix this second issue. 

It is remarkable that both issues are fixed by using the Boardman-Vogt construction $W$ and the operad isomorphism 
$W(FM) \cong FM$ from \cite{new}.

\subsection{Removing the rotations} \label{sub:rot}

Since $\overline{\Xi}_a$ does not depend on rescalings we normalize the parameters so that $a$ is in the open $k$-simplex
$$\stackrel{\circ}{\Delta}_k= \{(a_1,\dots,a_k) \in \R_+ \, | \, a_1 + \dots +a_k=1\}$$
\begin{Def}
Given $a=(a_1,\dots,a_k) \in \stackrel{\circ}{\Delta_k}$ and
 $b=(b_1,\dots,b_l) \in \stackrel{\circ}{\Delta_l}$, for $i=1,\dots,k$ let 
$$a \circ_i b = (a_1,\dots,a_{i-1},a_i b_1, \dots,a_i b_l,a_{i+1},\dots,a_k) \in \stackrel{\circ}{\Delta}_{k+l-1}$$
This defines an operad structure on the open simplexes introduced by Arone and Kankaarinta in \cite{AK}, such that the operad
structure maps are homeomorphisms. 
\end{Def}
Recall from the previous section that $\theta^i_a:FM(k) \to S^1$ are nullhomotopic maps, parametrized by 
$a \in \stackrel{\circ}{\Delta}_k$. The nullhomotopy 
$h^i_a:FM(k) \times [0,1] \to S^1$ is given by 
$$h^i_a(x,t)=\theta^i_{t a_1,\dots,t a_{i-1}, t a_i + (1-t),  t a_{i+1},\dots,  ta_k     }(x)$$ for  $t>0$ and 
$h^i_a(x,0)=1 \in S^1$, so that 
$h^i_a(x,1)=\theta^i_a(x)$.

Let $W$ be the Boardman-Vogt construction for topological operads \cite{BV}.
In our case $WFM$ consists of nested trees with vertices labelled by elements of $FM$, and internal edges with length in $[0,1]$,
modulo the relation that collapses an internal edge of length $0$, and composes the labels of its vertices.
The composition in $WFM$ grafts labelled tree and gives length $1$ to the new internal edge.
We will indicate an extended $\circ_i$-notation with an upper index indicating the length of the edge that we are grafting.
For example $x \circ_1^{t} y \in WFM(3)$, for $x,y \in FM(2)$ is the nested tree $\{\{1,2\},\{1,2,3\}\}$ with root label $x$, non-root label $y$, 
and internal edge of length $t$. Sometimes we will also use $WFM$ or $FM$ as superscripts to specify the operad where we are composing.

\medskip

We will need twice the following result, that we announced in \cite{Barcelona}.
We posted a new detailed equivariant proof in \cite{new}, that applies to any Fulton-MacPherson operad $FM_n$ with $n \geq 1$,
although we are interested in the case $FM=FM_2$.

\begin{Thm} 
There exists a $SO(2)$-equivariant isomorphism of topological operads $$\beta: FM \cong W(FM)$$
\end{Thm}

The isomorphism $\beta$ can be constructed explicitly in our case $n=2$.
Namely by the homeomorphism $\overline{\Xi}_{a_1,\dots,a_k}:FM(k) \cong \bar{N}_k(\C)$ of Theorem \ref{xibar} we can construct an explicit collar of the boundary 
$\de FM(k) \subset FM(k)$ by rescaling edge lengths in $\bar{N}_k(\C)$ (see Remark 9 of \cite{new}). The collar can be chosen piecewise algebraic if the weights are rational.
We use these collars in the proof of the main theorem of \cite{new} to get an explicit description of $\beta$.

\begin{Thm} \label{forst}
There is a family of homeomorphism $\Xi'_a :FM(k) \to \bar{N}_k(\C)$ depending continuously on $a \in \stackrel{\circ}{\Delta}_k$
inducing a cellular structure on $FM(k)$, that we call its $a$-structure.
The operad composition $\circ_i: FM(k) \times FM(l) \to  FM(k+l-1)$ is cellular if we put
\bit
\item the $a$-structure on $FM(k)$,
\item  the $b$-structure on $FM(l)$, 
\item  the $a \circ_i b$-structure on $FM(k+l-1)$.
\eit
 If $\sigma \in \Sigma_k$ is a permutation, then the induced map
$\sigma_*:FM(k) \to FM(k)$ is cellular if we put the $a$-structure on the domain and the $\sigma^{-1}(a)$-structure on the range. 
\end{Thm}

Theorem \ref{forst} will follow from the following lemma.

\begin{Lem} \label{remrot}
There is a $\Sigma_k$-equivariant continuous family of homeomorphism 
$$\Lambda: \stackrel{\circ}{\Delta}_{k-1} \times FM(k) \to FM(k)$$
$\Lambda(a,x)=\Lambda_a(x)$  such that for any nested tree $\ms$ on $k$ leaves
$$\Lambda_{a}(\beta_\ms((x_S)_S)) =   \beta_\ms((\Lambda_{a_{\hat{S}}}(\alpha_S^{-1} x_S))_S)$$
\label{doublestar}
with $\alpha_S$ as in Definition \ref{def:xi}.  
\end{Lem}

We prove the Lemma 
\begin{proof}
We fix a $SO(2)$-equivariant operad isomorphism $\beta:FM \cong WFM$ and work on $WFM$.
We will construct a continuous family of self-homeomorphism 
$$g:\stackrel{\circ}{\Delta}_{k-1} \times WFM(k) \to WFM(k) $$
$g(a,x)=g_a(x)$
such that $\Lambda=\beta^{-1} g (\beta \times \Delta_k)$ respects \ref{doublestar}.
So we need to verify that 
\begin{equation} 
\label{compcomp}
g_a(\circ_{S \in \ms}^{WFM} \, y_S ) =  \circ_{S \in \ms}^{WFM}\,  g_{a_S} (\alpha_S^{-1} y_S) \end{equation}

where the operadic product takes place in $WFM$, and the family $\{\alpha_S\}_{S \in \ms}$ associated to the nested tree $\ms$ 
with labels $x_S:=\beta^{-1}(y_S)$ is defined inductively from the root to higher vertices.
Actually it is enough to verify this equation for single $\circ_i$ compositions.

\

The map  $g_a=g(a,\_)$ will act by keeping the same shape of labelled trees, the same egde lengths, and rotating the label of the vertices by appropriate angles.

{\bf Arity $k=2$}: we start with the projection $$g: \stackrel{\circ}{\Delta}_1 \times WFM(2) \to WFM(2),$$ so that $g_a$ is the identity for any
$a \in \stackrel{\circ}{\Delta}_1$.

{\bf Arity $k=3$}: 
 the elements of $WFM(3)$ are either trees with a single vertex,  labelled by  $z \in FM(3) \subset WFM(3)$,
or trees with an internal edge of length $t$, and two vertices labelled by $x,y \in FM(2)$, i.e.

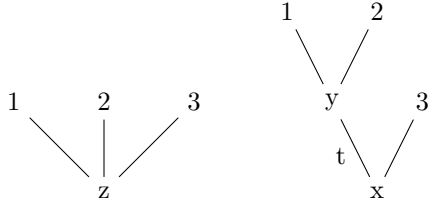
\medskip

\begin{center}
\begin{tikzpicture} [scale=.8]  
\node {z}[grow'=up]
  child {node{1}} child{node{2}}  
child {node{3}}    ;
\end{tikzpicture}
\quad \quad
\begin{tikzpicture} [scale=.8]  
\node at (-0.6,0.6) {t};
\node {x}[grow'=up]
child { node{y}   child {node{1}} child{node{2}}  }
child {node{3}}    ;
\end{tikzpicture}
\end{center} 

 In the first case 
we set $g_a(z)=z$ for any $z \in FM(3) \subset WFM(3)$ and $a \in \stackrel{\circ}{\Delta}_2$.
In the second case we set $$g_a(x \circ_1^t y)=(x \circ_1^t   h^1_{a_1+a_2,a_3}(x,t)^{-1} y)$$ and extend $g$ to the remaining
labelled trees by requiring that $g$ is equivariant with respect to the action of the symmetric group.
This definition is consistent for $t=0$ since $h^1_{a_1+a_2,a_3}(x,0)=1$.
Condition \ref{compcomp} is satisfied in this case because in $WFM(3)$ the operadic composition is $x \circ_1^{WFM} y = x \circ_1^1 y$ (with an internal edge of length 1),
and so 
 $$g_a(x \circ_1^{WFM} y)= x \circ_1^{WFM}  \theta^1_{a_1+a_2,a_3}(x)^{-1} y  $$
 as required in equation \ref{compcomp}.

{\bf Arity $k=4$} We have various cases:
\begin{enumerate}
\item \label{uno} If $v \in FM(4) \subset WFM(4)$ then we set $g_a(v)=v$.
\item \label{due} For labelled trees $v=(x \circ_2^t z ) \circ_1^s y$ we set 
$$g_a(v)=(x \circ_2^t  h^3_{a_1,a_2,a_3+a_4}(\beta^{-1}(x \circ_1^s y),t)^{-1}z )\circ_1^s  h^1_{a_1+a_2,a_3,a_4}(\beta^{-1}(x \circ_2^t z),s)^{-1} y$$ 

\item \label{tre} For trees $v=A \circ_3^t z$, with $A \in FM(3)$, we set 
$$g_a(v)=A \circ_3^t h^3_{a_1,a_2,a_3+a_4}(\beta^{-1}(A),t)^{-1}z$$
\item \label{quattro} Similarly, for $v=x \circ_1^s B$, with $B \in FM(3)$, we set 
$$g_a(v)=x \circ_1^s h^1_{a_1+a_2+a_3,a_4}(x,s)^{-1} B$$ 
\item \label{cinque} For labelled trees of the form
$v=(x \circ_2^s y ) \circ_2^t z$ with two consecutive internal edges, we have to subdivide 
the square given by $s,t\in [0,1]$ into two triangles defined by $s \leq t$ and $s \geq t$. 
$$g_a(v)=\begin{cases}
 {\rm For } \quad  s \leq t \quad
 (x \circ_2^s h^2(x,s)^{-1}y) \circ_3^t   h^3_{a_1,a_2,a_3+a_4}(\beta^{-1}(x \circ_2^{s/t} y),t)^{-1}   z   \\
{\rm For } \quad s \geq t \quad  
 (x \circ_2^s h^2(x,s)^{-1}y) \circ_3^t   \theta^3_ {sa_1,ta_2,1-s+(s-t)a_2+s(a_3+a_4)}(x \circ_2^0 y)^{-1}  z
\end{cases} $$

\end{enumerate}
The definition is consistent in cases (\ref{uno}) and (\ref{tre}): namely, for $t=0$,
$g_a$ is the identity on $$v=A \circ_3^0 z = A \circ_3^{FM} z \in FM(4)$$ (the operadic composition is in $FM$).
Similarly cases (\ref{uno}) and (\ref{quattro}) are consistent for $s=0$ because $g_a$ is the identity on $$v=x \circ_1^0 B= x \circ_1^{FM} B \in FM(4).$$
For $s=0$ in case (\ref{due}) $$g_a(v)
=(x \circ_1^{FM} y) \circ_3^t  h^3_{a_1,a_2,a_3+a_4}(\beta^{-1}(x \circ_1^{FM} y),t)^{-1} z $$
gives the compatibility with case (\ref{tre}). Symmetrically the compatibility holds when $t=0$.

Let us check the gluing conditions in case (\ref{cinque}).
For $s=0$  we set $A= x \circ_2^0 y$ so that $v=A  \circ_2^t z$.
Since $h^2(x,0)=1$ and $0=s \leq t$,  $g_a(v)= A \circ_3^t h^3(\beta^{-1}(A),t)^{-1}z$ that coincides with the definition in case (\ref{tre}).
For $t=0$, we set $B=y \circ_2^0 z$ so that $v=x \circ_2^s B$.  Since $0=t \leq s$ 
$$g_a(v)=(x \circ_2^s h^2(x,s)^{-1} y) \circ_2^0   \theta^3_{sa_1,0,(1-s)+s(a_2+a_3+a_4)}(x \circ_2^0 y)^{-1}   z $$
Now $$\theta^3_{sa_1,0,(1-s)+s(a_2+a_3+a_4)}(x \circ_2^0 y) = h^2_{a_1,a_2+a_3+a_4}(x,s)$$  and so the definition is compatible with the
$\circ_2$ version of  case (\ref{quattro}) since by $SO(2)$-equivariance of the operad composition in $FM$
$$h^2_{a_1,a_2+a_3+a_4}(x,s)^{-1}(B)=(h^2_{a_1,a_2+a_3+a_4}(x,s)^{-1}   y   )\circ_2^0  (h^2_{a_1,a_2+a_3+a_4}(x,s)^{-1} z)$$


Let us consider the compatibility with the composition in $WFM$ (equation \ref{compcomp}). 
We must consider trees with an internal edge of length 1, that are decomposable in $WFM$.
In case (\ref{tre})  with $t=1$ $$g_a(v)=A \circ_3^1 \theta^3(\beta^{-1}(A))^{-1} z$$ and the compatibility holds.
Namely if we indicate the root by $R$ and the non-root vertex by $S$, we have that $\alpha_R=1$, $\alpha_S=\theta^3(\beta^{-1}(A))$ and
  $g_{a_1,a_2,a_3+a_4}(A)=A$. 
In case (\ref{quattro}) with $s=1$ we have that
$$g_a(v)=x \circ_1^1  \theta^1(x)^{-1} B $$ and the equation \ref{compcomp} holds.
Namely $g_{a_2,a_3,a_4}$ is the identity on $(\theta^1(x)^{-1} B)$ and $B$ because both $\theta^1(x)^{-1} B$ and $B$ belong to the $SO(2)$-invariant subspace 
$FM(3) \subset WFM(3)$, that is fixed by $g_{a_2,a_3,a_4}$.

Consider in case (\ref{due}) $s=1$. Then 
$v=(x \circ_2^t z) \circ_1^{WFM} y$. Now 
$$g_a(v)=(x \circ_2^t  h^3(\beta^{-1}(x\circ_1^{WFM}y) ,t)^{-1} z )     \circ_1^{WFM} \theta^{1}(\beta^{-1}(x \circ_2^t z))^{-1} y $$
Since $\beta$ is an operad isomorphism and it is the identity in arity 2, $\beta^{-1}(x \circ_1^{WFM} y)=x \circ_1^{FM} y$. 
Notice that $$h^3_{a_1,a_2,a_3+a_4}(x \circ_1^{FM} y,t )=h^2_{a_1+a_2,a_3+a_4}(x,t),$$ since the flow defining these angles is the same, 
and therefore equation \ref{compcomp} holds. Symmetrically the equation holds when $t=1$ in case (\ref{due}).

Let us check the equation \ref{compcomp} in case (\ref{cinque}).
For $t=1$ $v= (x \circ_2^s y) \circ_3^{WFM} z$. Let us set $A=x \circ_2^s y$.  
Since $s \leq t=1$ we have that 
$$g_a(v)=(x \circ_2^s  h(x,s)^{-1} y) \circ_3^{WFM} \theta^3_{a_1,a_2,a_3+a_4}(\beta^{-1}(A))^{-1} z$$
But $g_{a_1,a_2,a_3+a_4}(x \circ_2^s y)=(x \circ_2^s h^2(x,s)^{-1} y)$, so that the root vertices labels  respect \ref{compcomp}, and the 
$g$-invariant vertex label
$z$ gets rotated by the angle $\theta^3(\beta^{-1}(A))^{-1}$, as required.

For $s=1$, $v=x \circ_2^{WFM} (y \circ_2^t z)$.
Since $t \leq s=1$ we have that 
$$g_a(v)=x \circ_2^{WFM} (  \theta^2(x)^{-1} y     \circ_2^t   \theta^3_{a_1,ta_2, (1-t)a_2+a_3+a_4} ( x \circ_2^0 y)^{-1}   z  )$$  
Now 
\begin{align*}
\theta^3_{a_1,ta_2, (1-t)a_2+a_3+a_4} ( x \circ_2^0 y) &= \theta^2_{a_1,a_2+a_3+a_4}(x) \theta^2_{ta_2, (1-t)a_2+a_3+a_4}(   \theta^2(x)^{-1} y  ) = \\
&=\theta^2_{a_1,a_2+a_3+a_4}(x) h^2_{a_2,a_3+a_4}(  \theta^2(x)^{-1} y, t) 
\end{align*}
Equation \ref{compcomp} clearly holds for the root vertex label $x$ that is $g$-invariant ; it holds for the label of the non root vertex 
$y \circ_2^t z$, since multiplying it by the angle $\theta^2(x)^{-1}$ yields $(\theta^2(x)^{-1} y) \circ_2^t (\theta^2(x)^{-1} z)$,
and the application of $g_{a_2,a_3,a_4}$ to this multiplies the upper label $\theta^2(x)^{-1} z$ exactly by 
$h^2(\theta^2(x)^{-1} y, t)^{-1}$. 

\medskip

{\bf General arity $k$}

$g_a(v)=v$ if $v \in FM(k) \subset WFM(k)$.
Let us consider a generic element $v \in WFM(k)$. This is a tree with internal edge lenghts in $[0,1]$ and vertex labels in $FM(k)$.
Then $g_a(v)$ produces a tree of the same shape, with same internal edge lengths, and rotates the vertex labels by appropriate
angles. Given a vertex $S=S_0$ of $v$, we define the rotation angle $\theta_S$ as follows.  
If $S$ is the root then we set $\theta_S=1$.  Otherwise
consider the edge $(S,S_1)$ coming out of the vertex $S$, and remember its length $s_1$. We cut the edge and look at the labelled tree $v_S$ below, that has a leaf $l$ corresponding to the cut edge. There is a unique directed path 
going to the root $R$ through vertices $S_1, \dots,S_ r=R$.  Let $s_i$ be the length of  the edge $(S_{i-1},S_i)$ for 
$i=2,\dots,r$. 
We define $m_i =\max(s_1,\dots,s_i)$ for $i=2,\dots,r$  and $l_i=s_i/m_i$.  
Let us modify in $v_S$ the length of the internal edge  $(S_{i-1},S_{i})$ to  $l_{i}$ for $i=2,\dots,r$.  This defines a new labelled tree
$v'_S$. We define $\theta_S:= \theta_{a'}^{l}(\beta^{-1}(v'_S))$, for an appropriate element of the simplex 
$a' \in \Delta_{ \{1,\dots,k\}/S }$ 
to be specified.  
This makes sense because the set of leaves of $v'_S$ (and $v_S$) is the quotient $\{1,\dots,k\}/S$. The leaf $l$ is the class of $S$ in the quotient.  Up to reindexing these labelled trees belong to $WFM(k-|S|+1)$. 

If  $j \in \{1,\dots, k\}$ and    $j \notin S$, follow the directed path from this leaf to the root, and let $S_i$ be the first vertex in $(S_1,\dots,S_r)$ 
encountered along this path. In other words $j \in S_i - S_{i-1}$. 
Then we set $a'_j =a_j m_{i}$.  This forces a unique choice for the remaining leaf 
$a'_l = \sum_{j \notin S}(a_j-a'_j) + \sum_{j \in S} a_j$, and completes the definition. 

\ 
The definition is continuous for labelled trees of a fixed shape by continuity of the maximum.
We check that the definition is compatible with edge collapse, the gluing relation defining $WFM$. 
If the internal edge $(S_{i-1},S_i)$ has length $s_i=0$ we have that $v$ is equivalent  
to a labelled tree $\bar{v}$ obtained by collapsing the edge and composing the labels of its ends. 
For $i>1$ both labelled trees $v_S$ and $\bar{v}_S$ have the same set of leaves.  
The new length of the internal edge $(S_{i-1},S_i)$ in $v'_S$ is $s_i/m_i=0$, and all other internal edge lengths 
of $v'_S$ and $\bar{v}'_S$ coincide (since $0$ is the minimum in $[0,1]$), so that the latter are equivalent labelled trees.  
Furthermore the simplicial coordinates  $a'$ and $\bar{a}'$ corresponding to the two trees are equal, and so $\theta_S$ is the same in both cases.
For $i=1$ we have that $\bar{v}_S$ is obtained from $v_S$ by replacing $S_1$ and the tree above it by a single leaf $l_1$. 
However $a'_j=0$ for all leaves $j \in S_1 -S_0$, and $a'_j = \bar{a}'_j$ otherwise,  so that  $\theta^l_{a'}(\beta^{-1}(v'_S)) = \theta^{l_1}_{\bar{a}'}(\beta^{-1}(\bar{v}'_S))$. The collapse of other edges does not have any influence on $\theta_S$. 

Next we need to show that our definition satisfies equation \ref{compcomp}. 
Suppose that the internal edge $(S,S_1)$ has length $s_1=1$ in the labelled tree $v$. Then 
$v_S=v'_S$, $a'_i=a_i$ for $i \notin S$, and $a'_l= \sum_{i \in S} a_i $. Then $\theta_S = \theta^{l}_{a'}(\beta^{-1}(v_S))$.  
 We can write $v=v_S \circ^{WFM}_S \bar{v}$, with $S$ root of the tree $\bar{v}$. 
 Let us set $\bar{a}_i=a_i/a'_l$ for $i \in S$.
 Clearly 
the rotation angle $\theta_T$ for a vertex $T$ of $v_S$ is the same if computed in $v_S$ or in $v$.  
Let us consider a vertex $T$ above $S$ in the tree $\bar{v}$. If we cut in $v$ the internal edge coming out of $T$,
and redefine edge lengths, then we obtain the tree $v'_T = v_S \circ^{WFM}_S \bar{v}'_T$, with $\bar{v}'_T$ obtained 
by the same cutting and relabelling procedure applied to $\bar{v}$.
The simplicial coordinates for the leaves of $v'_T$ are exactly 
$a'  \circ_S  \bar{a}'$, with $\bar{a}'$ simplicial coordinates for $\bar{v}'_T$. 
This implies that the rotation angle for $T$ in $v$ is 
\begin{equation}\theta_T = \label{teta}  \theta_S \cdot \theta^{T}_{\bar{a}'} (  \theta_S^{-1}    \beta^{-1} (           \bar{v}'_T)) \end{equation}
where we indicate $\theta_S=\theta^{l}_{a'}(\beta^{-1}(v_S))$.
Equation \ref{compcomp} (for a single composition)  
specializes in our case to
$$g_a(v)=  g_{a'}(v_S) \circ^{WFM}_S   g_{\bar{a}}(\theta_S^{-1} \bar{v})$$ 
For vertices not above $S$ we have verified that the rotation used to define $g_a(v)$ and $g_{a'}(v_S)$ is the same. 
For a vertex $T$ above $S$ we have that the rotation angle used in $g_a(v)$ is $\theta_T$, and 
the rotation angle of $T$ in $g_{\bar{a}}(\bar{v})$ is  $\theta^{T}_{\bar{a}'} (    \beta^{-1} (           \bar{v}'_T))$.
By the $SO(2)$-equivariance of $\beta$ and equation \ref{teta} we conclude.

 \end{proof}

We prove Theorem \ref{forst}.  

\begin{proof} 
Define the $a$-cellular structure on $FM(k)$ so that  $\Xi'_a:=\overline{\Xi}_a \circ \Lambda^{-1}_a$ is a cellular homeomorphism.  
By the definition \ref{def:xi} of $\overline{\Xi}$ on the boundary of $FM(k)$  and by Lemma \ref{remrot} we conclude.

\end{proof}

\subsection{Removing the weights} \label{sub:weights}
In this subsection we get rid of the weights, and complete the proof of Theorem \ref{princ} , showing that there is a cell decomposition of $FM$ such that the operad composition in $FM$ sends product of cells to cells, i.e. the cell decomposition is "operadic", that is statement (2) of Theorem \ref{princ}.
The strategy is to construct an operadic cell decomposition on $W(FM)$, with cells in bijective correspondence with those of $\bar{N}(\C)$, 
and then apply the operad isomorphism $\beta:FM cong W(FM)$ .  The operadic cells of $W(FM)$ will be obtained from the cells of $FM \subset WFM$ 
from the previous subsection by adding a partial collar to their intersection with the boundary $\de FM$.

\begin{figure}
\begin{center}
\begin{tikzpicture}
\draw (0,0) -- (0,3);
\draw  (2,0) -- (2,2)  (2,2)--(0,2)  (0,0)--(3,0) ; 
\draw [dotted] (2,2)--(2,3) (2,2)--(3,2) ;
\draw  (0,3)--(3,3)--(3,0) ;
\draw (0,2) -- (2,2) (2,0) -- (2,2) ; 
\node at (-0.4,-0.4) {$[P]$};
\node at (.3,.2)  {$(P)$};
\node at (1.7,0.2) {$(Q)$};
\node at (1,0.2) {$(e)$};
\node at (.3,1) {$(g)$};
\node at (1.7,1.8) {$(R)$};
\node at (1,1.8) {$(f)$};
\node at (0.3,1.8) {$(S)$};
\node at (3.4,-.4) {$[Q]$};
\node at (1.7,1) {$(h)$};
\node at (3.4,3.4) {$[R]$};
\node at (-.4,3.4) {$[S]$};
\node at (1.5,-0.4) {$[e]$};
\node at (1.5,3.4) {$[f]$};
\node at (-.4,1.5) {$[g]$};
\node at (3.4,1.5) {$[h]$};
\node at (1,1) {$(\sigma)$};
\node at (2.1,2.1) {$[\sigma]$};

\end{tikzpicture}  
\end{center}
\caption{Cells of $FM$ and their extensions in $WFM$}
\end{figure}
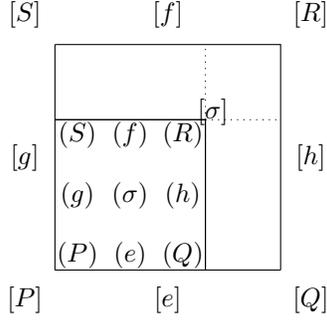

We warm up with the cases in low arity.
\medskip

{\bf Arity 2 } On  $WFM(2)=FM(2)$ we put the $b_2=(1/2,1/2)$-structure.

\medskip

{\bf Arity 3} For $WFM(3)$, consider the  $b_3=(1/3,1/3,1/3)$-decomposition of $FM(3) \subset WFM(3)$.
Any cell $\sigma \subset  \bar{N}_3(\C)$, such that $\de\sigma$ intersects the component of the boundary $(12)3$,
 but 
is not entirely contained in it, has the form $\sigma \cong [0,1] \times \sigma_1 \times \sigma_2$, with $\sigma_1,\sigma_2$ cells of $\C_2$. 
This gives a cell $(\Xi'_{b_3})^{-1}(\sigma) \subset FM(3)$ of the $b_3$-structure  homeomorphic to 
$[0,1] \times \sigma_1 \times \sigma_2$ via $\Xi'_{b_3}$.  
Notice that $$(\Xi'_{b_3})^{-1}(\sigma) \cap \de FM(3) = (\Xi'_{b_3})^{-1}( \{0\} \times \sigma_1 \times \sigma_2) $$
Let us consider $[0,1] \times FM(2) \times FM(2)$ as a subspace of $WFM(3)$ via $(t,x,y) \mapsto x \circ_1^t y$, and  
the embedding $e:[0,1] \times FM(2) \times FM(2) \to [0,1] \times \bar{N}_3(\C)$ 
defined by
$$e(x \circ_1^t y)=(t,\Xi'_{c(t)}(x \circ_1^{FM} y))$$
with $c(t)=(1-t)(1/3,1/3,1/3)+t(1/4,1/4,1/2)$. Notice that $(1/4,1/4,1/2)=b_2 \circ_1 b_2$ in the open simplex
operad. 
The union of the preimages $e^{-1}([0,1] \times \sigma_1 \times \sigma_2)$ and
$(\Xi'_{b_3})^{-1}(\sigma)$ along their common intersection $$(\Xi'_{b_3})^{-1}( \{0\} \times \sigma_1 \times \sigma_2)=e^{-1}(\{0\} \times \sigma_1 \times \sigma_2)$$  
gives us a single cell in $WF(3)$ homeomorphic to $[0,1] \times \sigma_1 \times \sigma_2 $, that we call $\sigma'$.
Finally we replace the cell $$e^{-1}(\{0\} \times \sigma_1 \times \sigma_2)   \subset \de FM(3)$$
 that is in the boundary
of $FM(3)$, by the cell 
$$e^{-1}(\{1\} \times \sigma_1 \times \sigma_2) \in \{1\} \times  FM(2) \times FM(2) \subset \de(WFM)(3)$$
that is exactly 
the composition $\sigma_1 \circ_1 \sigma_2$ in $WFM(3)$. 
This procedures, extended equivariantly via the action of the symmetric group,  accounts for all cells intersecting the boundary or contained in it.
If $\tau \subset FM(3) \subset WFM(3)$ is a cell not intersecting the boundary $\partial FM(3)$, then it is mapped by $\Xi'_{b_3}$ to a cell of $\C_3 \subset \bar{N}_3(\C)$.
We keep such cells.

\medskip
{\bf Arity 4}
Let us consider $FM(4) \subset WFM(4)$. In this case we consider the 

$b_4=(1/4,1/4,1/4,1/4)$-structure on $FM(4)$.
We keep all cells of the $b_4$-structure that do not intersect the boundary. These cells are mapped 
by $\Xi'_{b_4}$ to cells of $\C_4 \subset \bar{N}_4(\C)$

\medskip

For a given nested tree $\mT$ of the boundary with a root of valence 3, an internal edge, and another vertex of valence 2, like for example $T=12(34)$, 
we consider  
$$[0,1] \times FM(3) \times FM(2)  \subset WFM(4)$$ as a subspace
via $(t,x,y) \mapsto x \circ_\mT^t y$.
Consider the embedding
$$e_\mT:[0,1] \times FM(3) \times FM(2) \to [0,1] \times \bar{N}_4(\C)$$
defined by
$e_\mT(t,x,y)= (t, \Xi'_{c(t)}(x \circ_\mT y))$ 
with $c(t)=(1-t)b_4 + t  (b_3 \circ_\mT b_2)  $. 
Similarly as before, for each cell $\rho \subset \bar{N}_4(\C)$ of the form $[0,1] \times \rho_1 \times \rho_2$ 
with $\rho_1 \subset \C_3$ and $\rho_2 \subset \C_2$ we glue $(\Xi'_{b_4})^{-1}(\rho)$ and $e_\mT^{-1}([0,1] \times \rho_1 \times \rho_2)$ getting
a cell $\rho' \subset WFM(4)$. 

If the tree $\mT$ has a root of valence 2, an internal edge, and another vertex of valence 3, like for example $T=1(234)$, then the procedure is similar using the formula 
$c(t)=(1-t)b_4 + t (b_2 \circ_\mT b_3)$.

\medskip

Consider a cell $\sigma$ of the form  $\sigma \cong [0,1]^2  \times \nu_1 \times \nu_2 \times \nu_3$, with $\nu_1,\nu_2,\nu_3 \subset \C_2$ such
that the associated nested tree $\mT$  
has a root of valence 2, and two internal edges going into the root, like for example $(12)(34)$.
We divide the square $[0,1]^2$ with parameters $s,t\in [0,1]$ into two 2-simplexes $\Delta$ and $\Delta'$ corresponding to 
$s \leq t$ and $t \leq s$. Here $s$ is associated to the edge $e$ and $t$ to the edge $e'$.
We consider $$[0,1]^2 \times FM(2)^3 \subset WFM(4)$$ as the subspace of labelled trees of shape $\mT$. 
Let us define a map
$$e_\mT: [0,1]^2 \times FM(2)^3 \to [0,1]^2 \times \bar{N}_4(\C)$$ given by
$$e_\mT(s,t,x,y,z)=(s,t,\Xi'_{c(s,t)}(\circ_\mT(x,y,z))) $$
with $$c(s,t)=
\begin{cases}
 b_4(1-s)+ (s-t) (b_3 \circ_{\mT/e'} b_2) + t (\circ_\mT(b_2,b_2,b_2))               \quad{\rm if }\quad  t \leq s \\   
 b_4(1-t)+ (t-s)  (b_3 \circ_{\mT/e} b_2)  + s (\circ_{\mT}(b_2,b_2,b_2))              \quad {\rm if} \quad s \leq t 

\end{cases} $$
Then $\sigma'$ is the union of 
\begin{align*}
& \sigma_4= (\Xi'_{b_4})^{-1}(\sigma) , \\
& e_\mT^{-1}([0,1]^2 \times \nu_1 \times \nu_2 \times \nu_3), \\
& e_{\mT/e}^{-1}([0,1]  \times  ([0,1] \times \nu_1 \times \nu_2 \times \nu_3)  )  , \\
&  e_{\mT/e'}^{-1}([0,1] \times ([0,1] \times \nu_1 \times \nu_2  \times \nu_3 )) 
 \end{align*}
If $\mT$ has two consecutive internal edges $e',e$, of respective lengths $t,s$, like for example  $1(2(34))$, then 
the procedure is similar, but the formula for $c(s,t)$ is 
 $$c(s,t)=
\begin{cases}
 b_4(1-s)+ (s-t) (b_2 \circ_{\mT/e'} b_3) + t (\circ_\mT(b_2,b_2,b_2))               \quad{\rm if }\quad  t \leq s \\   
 b_4(1-t)+ (t-s)  (b_3 \circ_{\mT/e} b_2)  + s (\circ_{\mT}(b_2,b_2,b_2))              \quad {\rm if} \quad s \leq t 

\end{cases} $$

The cells contained in the boundary of $FM(4)$ are replaced by cells in the boundary of $WFM(4)$, that 
are operadic compositions of those defined in arity 2 and 3. 

\medskip

{\bf General arity $k$}

Remember from the proof of Theorem \ref{condo} 
that a cell of $\bar{N}_k(\C)$ corresponds to a nested tree $\mathcal{T}$  on $k$ leaves, with each vertex $u$ labelled 
by a cell $\sigma_u$  of the cactus complex $\C_{|u|}$, and a partition of the set of internal edges $I= I_0 \coprod I_1$.
This cell is homeomorphic to $$  \prod_{u \in \mathcal{T}} \sigma_u  \times \{0\}^{I_0} \times [0,1]^{I_1}$$
and is denoted by $(\sigma)= (\mT,(\sigma_u)_{u \in \mT}, I_0,I_1)$.
We will describe a cell decomposition of $WFM(k)$ with cells in bijective correspondence with those of $\bar{N}_k(\C)$.
The cell of $WFM(k)$ corresponding to the cell of $\bar{N}_k(\C)$ above is denoted with square brackets by $[\sigma]=[\mT,(\sigma_u)_{u \in \mT}, I_0,I_1]$.

\medskip

Cells of the {\em first type} correspond to corolla trees, that have no internal edges ($I_0=I_1=\emptyset$) with the root vertex labelled by 
a single cell $\sigma \in \C_k$. The corresponding cell of $WFM(k)$ is simply $[\sigma]:= (\Xi'_{b_k})^{-1}(\sigma)$.

\medskip

Cells of the {\em second type} have  $I_0=\emptyset$ but $I_1$ is not empty. 
For each nested tree $\ms$ on $k$ leaves and set of internal edges $\ms'$ there is a subspace of $WFM(k)$ 
consisting of labelled trees of shape $\ms$,  that we denote by $WFM(\ms)$, and is naturally homeomorphic to $$[0,1]^{\ms'} \times \prod_{u \in \ms} FM(|u|)$$
 We will construct an embedding $$e_{\ms}: WFM(\ms) \to [0,1]^{\ms'} \times \bar{N}_k(\C)$$

\medskip
\begin{Def} \label{defc}
We define an operad map $$c:WCom \to  \stackrel{\circ}{\Delta}$$
We visualize $WCom(n)$ as space of non-planar trees with $n$ marked leaves and internal edges of length between $0$ and $1$.
For each $n$ let $P_n \in Com(n) \subset WCom(n)$ be the corolla tree with $n$ leaves, the unique element fixed by the action of the symmetric
group $\Sigma_n$. We define $c(P_2)=b_2$ and extend the operad map by induction. Suppose that $c$ has been defined in arity less than $n$.
Then by operadic extension $c$ is uniquely defined on the space of decomposables $\de WCom(n)$, i.e. the space of trees which have at least an internal edge of length 1. We extend $c$ to the whole space $WCom(n)$  so that $c(P_n)=b_n$, by considering a convex combination:
for $T \in WCom(n)$, let $s$ be the greatest length of its internal edges. Let $T/s \in \de WCom(n)$ be the tree obtained from $T$ by 
substituting each length $l$ by $l/s$. Then $$c(T):=s \cdot c(T/s)+ (1-s)\cdot b_n \in \, \stackrel{\circ}{\Delta}_{n-1}$$

\end{Def}

The unique operad map $e:FM \to Com$ induces an
operad surjective map $$We:WFM \to WCom$$ projecting the stratum $WFM(\ms) \cong [0,1]^{\ms'} \times \prod_{u \in \ms} FM(|u|)$ 
 onto  the corresponding stratum of trees with shape $\ms$, denoted $WCom(\ms)=[0,1]^{\ms'} \subset WCom(n)$.

\medskip

Now given an element in $(t, (x_u)_{u \in \ms}) \in WFM(\ms)$, with $t \in [0,1]^{\ms'}$ and $x_u \in FM(|u|)$,
we define $$e_{\ms}(t, (x_u)_{u \in \ms})=  (t, \Xi'_{c(t)}  (\circ_{\ms}^{FM} (x_u)_{u \in \ms}) )$$

Recall from section \ref{sec:fulton}   that the strata $FM(\ms)$  of $FM(k)$ are indexed by nested trees $\ms$ on $k$ leaves. 
For each $a \in \stackrel{\circ}{\Delta}_{k-1}$ the homeomorphism $\Xi'_a:FM(k) \cong \bar{N}_k(\C)$ restricts to a homeomorphism
 $FM(\ms) \cong \bar{N}(\ms)$, where the latter is the space of labelled trees such that all edges out of vertices in $\ms$ have length $0$.
Therefore $$e_{\ms}: WFM(\ms) \to [0,1]^{\ms'} \times \bar{N}_k(\C)$$ is an embedding with image
$ [0,1]^{\ms'} \times \bar{N}(\ms)$.

The new cell of $WFM(k)$
corresponding to $(\mathcal{T},(\sigma_u)_{u \in \mT}, \emptyset,\mT')$
is
\begin{equation} \label{cupcup} 
[\mathcal{T},(\sigma_u)_{u \in \mT}, \emptyset,\mT']:= \bigcup_{\ms \subseteq \mT}
e_{\ms}^{-1}(\;[0,1]^{\ms'} \times (\mT,  (\sigma_u  )_{u \in \mT}  ,\,  \ms', \,  \mT' \setminus \ms'  )\; )
\end{equation}

\medskip

Cells of the {\em third (and last) type} arise for $I_0$ not empty.  They are defined to be the operadic composition of cells of first and second type, that correspond to the trees obtained by cutting the internal edges in $I_0$.
Intuitively there is a nested tree $\mathcal{A}$ with $\mathcal{A'} \cong I_0$, and a decomposition 
$I_1 = \coprod_{\alpha \in \mathcal{A}} I_1^{\alpha} $, such that the nested tree $\mT$ is obtained by grafting 
nested trees $\mT_{\alpha}$ along  $\mathcal{A}$ (replacing each vertex $\alpha$ by $\mT_{\alpha}$ ).
Then $$[\mT,(\sigma_u)_{u \in \mT}, I_0,I_1] := \circ_{\mathcal{A}}^{WFM} [\mT_{\alpha}, (\sigma_u)_{u \in \mT_{\alpha}},\emptyset,
I_1^{\alpha} ] $$

\


We need to prove that our new cells give indeed a regular $\Sigma_k$-equivariant CW-decomposition of $WFM(k)$ for each $k$.
We split the proof into a few lemmas.

\begin{Lem} \label{dila}
The subspace $[\sigma]=[\mT, (\sigma_u)_{u \in \mT}, I_0, I_1  ] \subset WFM(k)$ is homeomorphic to the cell
$(\sigma)=(\mT, (\sigma_u)_{u \in \mT}, I_0, I_1) \subset \bar{N}_k(\C)$ via a homeomorphism $\Gamma_\sigma:[\sigma] \to (\sigma)$
\end{Lem}
\begin{proof}

If $[\sigma]$ is of the first type ($\mT$ is a corolla tree) then $$\Xi'_{b_k} : [\sigma]=(\Xi'_{b_k})^{-1}((\sigma)) \to (\sigma)$$ is clearly a homeomorphism.
If $[\sigma]$ is of the second type, then it is obtained by adding to $(\Xi'_{b_k})^{-1}((\sigma)) \cong (\sigma)$ a partial collar, in a
 certain sense.  We observed that $e_{\ms}:WFM(\ms) \cong [0,1]^{\ms'} \times \bar{N}(\ms)$ is a homeomorphism. 
We choose a homeomorphism 
\begin{equation} \label{iden0}
(\sigma) \cong \prod_{u \in \mT} \sigma_u \times [0,1]^{\mT'}
\end{equation}
similarly as in section \ref{sec:fulton}
except that we replace $t \leftrightarrow 1-t$ in  $[0,1]$ for convenience. This 
 restricts to a homeomorphim
$$(\sigma_{\ms}):=(\mT,(\sigma_u)_{u \in \mT}, \ms' ,  \mT' \setminus \ms' ) \cong \prod_{u \in \mT} \sigma_u \times \{1\}^{\ms'} \times [0,1]^{\mT' \setminus \ms'}$$
By identifying $[0,1]\cong [1,2]$ via $t \mapsto t+1$ for the coordinates labelled by $\ms'$ we have that 
\begin{equation} \label{iden}
e_{\ms}^{-1}([0,1]^{\ms'} \times (\sigma_{\ms})) \cong [1,2]^{\ms'} \times \prod_{u \in \mT} \sigma_u \times [0,1]^{\mT' \setminus \ms'} 
\end{equation}
is homeomorphic to $(\sigma)$. By definition $[\sigma]$ is the union of  
$e_{\ms}^{-1}([0,1]^{\ms'} \times (\sigma_{\ms}))$
over all trees $\ms \subseteq \mT$.
Gluing is compatible with the identifications  (\ref{iden})
and so $$[\sigma] \cong \prod_{u \in \mT} \sigma_u \times [0,2]^{\mT'}$$
In view of (\ref{iden0}), a dilation gives a homeomorphism $\Gamma_{\sigma}:[\sigma] \cong (\sigma)$.

\medskip

A cell $[\sigma]$ of type 3 in $WFM$ is a product of cells of type 1 and 2 ( via the operadic composition maps of $WFM$, that
are embeddings),  and so is the corresponding cell $(\sigma)$ in $\bar{N}(\C)$. Since the lemma holds for cells of type 1 and 2, 
 $\Gamma_\sigma$ is an operadic product of the previous homeomorphisms. This terminates the proof.

\end{proof}
\begin{Lem} \label{unione}
The space $WFM(k)$
is union of closed cells of type 1 and 2 $$[\sigma]=[\mT,(\sigma_u)_{u \in \mT}, \emptyset ,\mT'], \, 
 \mT \in N_k$$

\end{Lem}
\begin{proof}   
We know that the union of the corresponding cells $(\sigma)$ is $\bar{N}_k(\C)$. 
We prove first that $FM(k) \subset WFM(k)$ is contained in the union of the cells $[\sigma]$. 
Considering the corolla tree $\{\{1,\dots,k\}\}=\ms \subseteq \mT$ in (\ref{cupcup}) we have that
$ (\Xi'_{b_k})^{-1}((\sigma)) \subseteq [\sigma]$,
and taking the union over all such cells $(\sigma) \subset \bar{N}_k(\C)$ proves the assertion since 
$FM(k)=(\Xi'_{b_k})^{-1}(\bar{N}_k(\C))$. 
Consider now a general stratum $WFM(\ms) \subset WFM(k)$. 
Recall that $e_{\ms}: WFM(\ms) \to [0,1]^{\ms'} \times \bar{N}_k(\C)$ is an embedding with image
$ [0,1]^{\ms'} \times \bar{N}(\ms)$.
We know that $\bar{N}(\ms)$ is the union of cells $(\sigma_\ms)=(\mT,(\sigma_u)_{u \in \mT}, \ms',\mT' \setminus \ms')$ for trees 
$\mT \supseteq \ms$.  Considering $\ms=\mT$ in  (\ref{cupcup})
$$e_{\ms}^{-1}([0,1]^{\ms'} \times (\sigma_\ms))  \  \subset   [\mT,(\sigma_u)_{u \in \mT},\emptyset , \mT'] $$
and so  $WFM(\ms)=e_{\ms}^{-1}([0,1]^{\ms'} \times \bar{N}(\ms))$ is contained in the union of the cells of first and second type.
The union of all strata $WFM(k)=\bigcup_{\ms \in N_k}WFM(\ms)$ is therefore union of cells of first and second type. 
\end{proof}

\begin{Lem} \label{bundari}
In $WFM(k)$
\ben
\item The boundary of a cell of type 1 is union of cells of type 1
\item The boundary of a cell of type 2 is union of cells of type 1,2,3 
\item The boundary of a cell of type 3 is union of cells of type 3
\item $[\tau] \subseteq [\sigma] $ if and only if $(\tau) \subseteq (\sigma)$ 
\een

\end{Lem}
\begin{proof}
The first assertion is clear since a cell of first type has the form $[\sigma]= (\Xi'_{b_k})^{-1}(( \sigma))$ and has subcells
$[\tau]= (\Xi'_{b_k})^{-1}((\tau))$, with $(\tau) \subset (\sigma)$.
For the second assertion, if $[\sigma]$ has type 2, then modulo the homeomorphism of lemma \ref{dila} 
we claim that the part of its boundary corresponding to $t_u=2$, with $u \in \mT'$, 
is the operadic composition in 
$WFM$ of the two cells of second type $[\sigma_1]$ and $[\sigma_2]$  associated to the trees $\mT_1$ and $\mT_2$ obtained by removing from $\mT$ the internal edge $u$,
that is a cell of type 3. For simplicity we consider the case when $\mT$ has a single internal edge $u$. Suppose that
 $\sigma_1 \subset \C_{m}$ decorates the root and $\sigma_2 \subset \C_{n}$ decorates the non-root vertex.   
 The cell $(\tau) \subset (\sigma)$ is obtained  by setting 
the internal edge length equal to 1, following the convention in Lemma \ref{dila}. 
 The cell $[\tau]$ is the operadic 
composition in $WFM$ of the type 1 cells  $[\sigma_1]=(\Xi'_{b_m})^{-1}(\sigma_1)$ and $[\sigma_2]=(\Xi'_{b_n})^{-1}(\sigma_2)$ along
$\mT$. Since we identify $$(x,y,2) =  x \circ_{\mT}^{WFM} y$$ and $$(x,y,1)= x \circ^{FM}_{\mT} y$$
the part of the boundary $\de[\sigma]$ corresponding to $t_u=2$ is
$$\Gamma_\sigma^{-1}((\tau)) = \{ x \circ_{\mT}^{WFM} y \; | \;   \Xi'_{c(1)}(x \circ^{FM}_{\mT} y) \in (\tau) \cong \sigma_1 \times \sigma_2            \}$$
with $c:[0,1] \to \stackrel{\circ}{\Delta}_k$ as in Definition \ref{defc},
 $c(1)=b_m \circ_{\mT} b_n$, and 
 $\Gamma_\sigma:[\sigma] \cong (\sigma)$ as in Lemma \ref{dila}.
 
By Theorem \ref{forst}
$$(\Xi'_{c(1)})^{-1}((\tau)) = (\Xi'_{b_m})^{-1}(\sigma_1) \circ^{FM}_{\mT} (\Xi'_{b_n})^{-1}(\sigma_2) $$
and so 
$$\Gamma_\sigma^{-1}((\tau))= (\Xi'_{b_m})^{-1}(\sigma_1) \circ^{WFM}_{\mT} (\Xi'_{b_n})^{-1}(\sigma_2) = [\tau]$$
The general case is similar, except that one needs to specialize to each intersection $[\sigma_1] \cap WF(\ms_1)$ and
$[\sigma_2] \cap WF(\ms_2)$, with $\ms_1 \subseteq \mT_1$ and $\ms_2 \subseteq \mT_2$  nested trees. 

The part of the boundary of $[\sigma]$ corresponding to $t_u=0$ contains the type 2 (or 1) cells obtained by collapsing the edge $u:v \to w$, and labelling the resulting vertex by a cell contained in $\sigma_v \circ_u \sigma_w$.
The remaining part of the boundary of $[\sigma]$ contains cells on a tree of the same shape as $[\sigma]$, with the same labels but one vertex $v$ labelled by a 
subcell of $\sigma_v$.
For the third assertion, a cell $[\sigma]$ of type 3 is an operadic product along some tree of cells  of type 1 and 2.
The result follows by the first and second assertion, and the fact that the operadic product is an embedding. 
For the fourth assertion, observe from the proof of Theorem \ref{condo}
 that a subcell of $(\sigma)$ in $N_k(\C)$ is obtained by the same combinatorics that we considered in the proof of the first three assertions. 
\end{proof}

\begin{Lem} \label{interiore}
The new cells $[\sigma]$ of $WFM(k)$ have disjointed interiors.
\end{Lem}

\begin{proof}
Given two distinct cells of first or second type $[\sigma]$ and $[\bar{\sigma}]$ 
it is sufficient to check that
$$[\sigma] \cap WFM(\ms)=e_{\ms}^{-1}( [0,1]^{\ms'} \times (\sigma_S) )  $$ and
$$[\bar{\sigma}] \cap WFM(\ms)=e_{\ms}^{-1}( [0,1]^{\ms'} \times (\bar{\sigma}_S) ) $$ 
do not intersect outside the boundary of $[\sigma] \cong [0,2]^{\mT'} \times \prod_u \sigma_u$ and
$[\bar{\sigma}] \cong [0,2]^{\bar{\mT'}} \times \prod_{\bar{u}} \bar{\sigma}_{\bar{u}}$, for any 
nested tree $\ms \in N_k$. Two cells of type 3 are operadic product of type 1 and 2 cells, and will intersect in the interior if and only if
they are operadic product along the same tree of cells of type 1 and 2 that intersect in the interior, and this is possible only if the cells coincide.  
\end{proof}

\begin{Thm} \label{ultimo}
The new cells $[\sigma]$ give a regular cell decomposition of $WFM$. 
\end{Thm}

\begin{proof}
Follows from lemmas \ref{dila}, \ref{unione}, \ref{bundari}, \ref{interiore}.
\end{proof}

Our operadic cell decomposition of $WFM$  determines an operadic cell decomposition of $FM$ via the operad isomorphism $\beta: FM \cong WFM$. 
This concludes the proof of Theorem \ref{princ}.

\end{document}